\documentclass[10pt]{amsart}

\usepackage{geometry,graphicx,amssymb,amsmath,amsbsy,eucal,amsfonts,mathrsfs,amscd,bm,esint,yhmath}
\usepackage[all]{xy}
\usepackage{tikz}
\usetikzlibrary{arrows,shapes}
\usetikzlibrary{arrows, decorations.markings,fit}
\usetikzlibrary{calc,3d}
\usepackage{tikz-3dplot}
\usepackage{soul}
\usepackage{lipsum}

\tikzset{
    >=stealth',
    punkt/.style={
           rectangle,
           rounded corners,
           draw=black, very thick,
           text width=6.5em,
           minimum height=2em,
           text centered},
    pil/.style={
           ->,
           thick,
           shorten <=2pt,
           shorten >=2pt,}
}

\usepackage{algorithm2e}
\usepackage{listings}

\numberwithin{equation}{section}

\allowdisplaybreaks[3]

\newtheorem{theorem}{Theorem}[section]
\newtheorem{lemma}[theorem]{Lemma}
\newtheorem{assumption}[theorem]{Assumption}

\theoremstyle{definition}
\newtheorem{definition}[theorem]{Definition}
\theoremstyle{remark}
\newtheorem{remark}[theorem]{Remark}

\newcommand{\ba}{\bm a}

\newcommand{\p}{{\partial}}

\newcommand{\nab}{\nabla}

\newcommand{\jump}[1]{\left[\hspace{-0.025in}\left[#1\right]\hspace{-0.025in}\right]}

\newcommand{\bld}[1]{\boldsymbol{#1}}
\newcommand{\bt}{\bld{t}}

\newcommand{\bI}{\bld{I}}

\newcommand{\bb}{\bld{b}}

\newcommand{\bv}{\bld{v}}
\newcommand{\bw}{\bld{w}}

\newcommand{\bp}{\bld{p}}
\newcommand{\bn}{\bld{n}}
\newcommand{\bu}{\bld{u}}

\newcommand{\bW}{\bld{W}}
\newcommand{\bV}{\bld{V}}

\newcommand{\bq}{\bld{q}}

\newcommand{\bPi}{\bld{\Pi}}
\newcommand{\bz}{\boldsymbol{z}}

\newcommand{\bH}{\bld{H}}

\newcommand{\bx}{\bld{x}}
\newcommand{\bC}{\bld{C}}
\newcommand{\bL}{\bld{L}}

\newcommand{\bX}{{\bm X}}

\newcommand{\bPhi}{{\bm \Phi}}
\newcommand{\bvarphi}{\bm \varphi}

\newcommand{\calE}{\mathcal{E}}

\newcommand{\bbR}{\mathbb{R}}

\newcommand{\calS}{\mathcal{S}}
\newcommand{\calT}{\mathcal{T}}

\newcommand{\bPsi}{{\bm \Psi}}

\newcommand{\calN}{\mathcal{N}}
\newcommand{\bnu}{{\bm \nu}}
\newcommand{\calP}{\mathcal{P}}
\newcommand{\calM}{\mathcal{M}}

\newcommand{\bell}{{\bm \ell}}
\newcommand{\calO}{\mathcal{O}}
\newcommand{\bbN}{\mathbb{N}}

\newcommand{\pt}{\wideparen}

\newcommand{\rev}[1]{\textcolor{black}{#1}}

\makeatletter
\def\widebreve{\mathpalette\wide@breve}
\def\wide@breve#1#2{\sbox\z@{$#1#2$}%
     \mathop{\vbox{\m@th\ialign{##\crcr
\kern0.08em\brevefill#1{0.8\wd\z@}\crcr\noalign{\nointerlineskip}%
                    $\hss#1#2\hss$\crcr}}}\limits}
\def\brevefill#1#2{$\m@th\sbox\tw@{$#1($}%
  \hss\resizebox{#2}{\wd\tw@}{\rotatebox[origin=c]{90}{\upshape(}}\hss$}
\makeatletter

\newcommand{\ipt}{\breve}

\title[Taylor-Hood for surface Stokes]{A Taylor-Hood finite element method for the surface Stokes problem without penalization}
\thanks{Last updated: \today}

\author[A. Demlow and M. Neilan]{
Alan Demlow\address{Department of Mathematics, Texas A\&M University, College Station, TX, 77843}
\email{demlow@tamu.edu}
\and{Michael Neilan}
\address{Department of Mathematics, University of Pittsburgh, Pittsburgh, PA 15260}
\email{neilan@pitt.edu}}

\thanks{ The first author was partial supported by NSF grant DMS-2012326.  The second author 
was partially supported by NSF grant DMS-2309425}


\begin{document}

\maketitle

\begin{abstract}
Finite element approximation of the velocity-pressure formulation of the surfaces Stokes equations is challenging because it is typically not possible to enforce both tangentiality and $H^1$ conformity of the velocity field.  Most previous works concerning finite element methods (FEMs) for these equations thus have weakly enforced one of these two constraints by penalization or a Lagrange multiplier formulation.  Recently in \cite{DemlowNeilan23} the authors constructed a surface Stokes FEM based on the MINI element which is tangentiality conforming and $H^1$ nonconforming, but  possesses sufficient weak continuity properties to circumvent the need for penalization.  The key to this method is construction of velocity degrees of freedom lying on element edges and vertices using an auxiliary Piola transform.  In this work we extend this methodology to construct Taylor-Hood surface FEMs.  The resulting method is shown to achieve optimal-order convergence when the edge degrees of freedom for the velocity \rev{space} are placed at Gauss-Lobatto nodes.  Numerical experiments confirm that this nonstandard placement of nodes is necessary to achieve optimal convergence orders. 
\end{abstract}

\thispagestyle{empty}

\section{Introduction}

This paper concerns the finite element approximation 
of the surface Stokes problem given by
\begin{equation}
\label{eqn:Stokes}
\begin{aligned}
-\bPi {\rm div}_\gamma {\rm Def}_\gamma \bu+\nab_\gamma p+\bu & = {\bm f}\qquad&\text{on }\gamma,\\
{\rm div}_\gamma \bu  = 0\text{ and } \bu\cdot \bnu &= 0 \qquad &\text{on }\gamma,
\end{aligned}
\end{equation}
where $\gamma \subset \bbR^3$ is a smooth, connected, and orientable two-dimensional surface without boundary.
Here $\bu:\gamma\to \mathbb{R}^3$ and $p:\gamma\to \mathbb{R}$ denote the velocity and pressure, respectively,
$\bPi$ is the tangential projection, and ${\rm Def}_\gamma$ is the surface deformation operator.
The letter $\bnu$ denotes the outward unit normal of $\gamma$, so that the constraint $\bu \cdot \bnu=0$ means
the velocity is tangent to the surface.

We focus on surface finite element methods (SFEMs) for \eqref{eqn:Stokes} based on the classical velocity-pressure formulation, where finite element spaces are defined on a piecewise-polynomial approximation of the surface. 
These methods mirror their Euclidean counterparts and are formulated
based on the PDE's variational formulation. As in the Euclidean case, the well-posedness and stability of
a SFEM for \eqref{eqn:Stokes} relies on a discrete inf-sup condition, showing that the choice
of finite element spaces cannot be chosen based on approximation properties alone. 
This naturally leads to the use of classical finite element Stokes pairs.  However  the tangential constraint $\bu\cdot \bnu=0$ imposes additional difficulties not found in the Euclidean setting. As explained in, e.g., \cite{DemlowNeilan23}, the simultaneous enforcement 
of tangentiality and $\bH^1$-conformity of the velocity approximation is not feasible on non-$C^1$ surfaces, and therefore at least one of these conditions must be relaxed in the numerical method.
For example, \cite{Fries18,GJOR18,OQRY18,Maxim19,HansboLarsonLarsson20,JORZ21,BJPRV22} use finite element spaces
based on $\bH^1$-conforming Stokes pairs and enforce the tangentiality constraint weakly through penalization or Lagrange multipliers.
This approach introduces additional degrees of freedom, as one approximates the redundant normal component of the velocity,
and they require a high-order approximation of the outward unit normal of $\gamma$ to achieve optimal-order convergence,
at least for low-order methods (cf.~\cite{HP23}).  Alternatively, one may construct exactly tangential velocity approximations
by utilizing $\bH({\rm div})$-conforming finite element spaces \cite{SurfaceStokes1,SurfaceStokes2}. This strategy
requires additional consistency and stability terms in the method, either by interior penalty techniques or through stabilization,
which can potentially increase the complexity of the scheme. The use of larger $\bH({\rm div})$-conforming finite element spaces also 
results in a larger system of unknowns compared to a $\bH^1$-conforming discretization.


Recently the authors proposed a novel finite element  method for the surface Stokes problem 
in \cite{DemlowNeilan23} that addresses many of the shortcomings of the aforementioned schemes. 
Similar to \cite{SurfaceStokes1,SurfaceStokes2}, this method enforces the tangentiality constraint exactly within the finite element space but relaxes function continuity.
The discrete velocity space is  ${\bm H}({\rm div})$-conforming,
yet in contrast to \cite{SurfaceStokes1,SurfaceStokes2}, it possesses sufficient weak continuity properties to ensure sufficient consistency
of the scheme without the inclusion of  edge-integral terms involving jumps and averages.  These properties are obtained by a novel design of the degrees of freedom, notably by incorporating an auxiliary Piola transform to transfer function information at nodal degrees of freedom without using data from the exact surface. As a result, the method in \cite{DemlowNeilan23} is the first convergent finite element method for \eqref{eqn:Stokes} that exactly mirrors the problem's variational formulation. 

The method in \cite{DemlowNeilan23} is based on the lowest-order Euclidean MINI pair \cite{mini}
 defined on  piecewise linear surface approximations, resulting in a first-order approximation scheme
 with respect to the $\bH^1\times L_2$ energy norm. Numerical experiments reported in \cite{DemlowNeilan23} also 
suggest that the $\bL_2$ velocity error converges with second-order, although a theoretical justification of this behavior was not provided.
While \cite{DemlowNeilan23} only considered one choice of Stokes pair, 
the methodology presented there is applicable to any finite element pair with nodal degrees of freedom.  
The theoretical analysis (e.g., inf-sup stability) must however be carried out in a case-by-case basis. 

 The objective of this paper is to build upon \cite{DemlowNeilan23}
and generalize the methodology to higher-order schemes defined on high-order geometry approximations.
 We apply the methodology in \cite{DemlowNeilan23} to the classical Taylor-Hood pair,
 taking the velocity space to consist of (mapped) piecewise polynomials
 of degree $k\ge 2$,  and the pressure space to be the (mapped) continuous, piecewise polynomials
of degree $(k-1)$. Similar to \cite{DemlowNeilan23}, nodal degrees of freedom of the velocity space are defined through
a novel use of the Piola transform to ensure tangentiality, ${\bm H}({\rm div})$-conformity,
and weak continuity properties of the finite element space. This construction is done with respect to the surface approximation, and the underlying
algorithm does not depend on the exact surface $\gamma$.

Although the design of the proposed finite element spaces and 
method builds directly on our previous work,
the convergence analysis is non-trivial and requires new theoretical tools.
For example, to establish inf-sup stability of the  Taylor--Hood pair,
we employ a surface-adapted version of  Verf\"urth's trick. In this approach, we first derive
an inf-sup stability result with respect to a weighted $H^1$-type pressure norm, and then use a specially designed
interpolation operator to extend this result to the $L_2$-norm.

Another key difficulty lies in
designing finite element spaces with sufficient weak continuity properties.
The analysis indicates and numerics confirm that a judicious placement of degrees of freedom
is crucial to ensure the method converges with optimal order. Specifically
edge degrees of freedom (which are absent in the MINI element) must coincide
with the Gauss-Lobatto nodes so that the inconsistency of the scheme is of high
enough order to guarantee optimal-order convergence in the isoparametric case.  In the case of the practially important lowest-order Taylor-Hood pair $\mathbb{P}^2-\mathbb{P}^1$, the interior edge Gauss-Lobatto node is the edge midpoint, so the nodes we use to define our degrees of freedom are precisely the Lagrange nodes and we thus obtain optimal convergence of the classical lowest-order Taylor-Hood method.  For $\mathbb{P}^3-\mathbb{P}^2$ and higher-order elements, the edge Gauss-Lobatto points do not coincide with standard Lagrange points, and so our degrees of freedom are not the classical ones.  We note that the recent work \cite{DurstNeilan24} also uses  Gauss-Lobatto points as degrees of freedom for the Piola-transformed Scott-Vogelius element
for the Euclidean Stokes problem on smooth domains. The present work shows that this essential idea carries over to the analogous surface problem, albeit with significant additional technical challenges in the analysis.

We emphasize that there are two distinct ``variational crimes'' or nonconformity errors present in our method:  a ``geometric error'' resulting from the approximation of $\gamma$ by a discrete counterpart, and a ``nonconformity error'' resulting from the lack of $H^1$ conformity in our method.  As is typical of surface finite element methods, our analysis indicates that the order of the geometric error increases with the degree of the polynomial discrete surface used to approximate $\gamma$.  In contrast, as confirmed in our numerical experiments, the order of the nonconformity error does not increase with the polynomial degree of the discrete surface.  Rather, the placement of the edge degrees of freedom is the key to obtaining optimal convergence.  The situation is thus analogous to the classical nonconforming Crouzeix-Raviart element (cf.~\cite{CR73,BrennerCR15}), where placing the degrees of freedom precisely at edge midpoints is essential to obtaining optimal order convergence. 

Finally, we prove optimal-order velocity error estimates in the $\bL_2$-norm, again a result
missing in \cite{DemlowNeilan23}.
This is obtained using a standard duality argument, combined with the 
weak continuity properties of the finite element space
and improved geometric consistency estimates.

The rest of the paper is organized as follows. In the next section,
we set the notation, outline our assumptions, and state several preliminary results.
In Section \ref{sec-3}, we define the finite element spaces and prove
approximation properties and inf-sup stability results.
Section \ref{sec-Consistency} states several consistency results,
estimating the non-conformity of the finite element spaces.
In Section \ref{sec-FEM} we state the finite element method
and prove error estimates in  the $\bH^1\times L_2$ energy
norm and derive an $\bL_2$ error estimate of the velocity approximation.
The paper concludes with numerical experiments
in Section \ref{sec-numerics}.

\subsection{\rev{Notation Summary}}\

\begin{table}[h!]
\centering
\rev{
\begin{tabular*}{0.9\textwidth}{@{\extracolsep{\fill}}ll@{\hspace{1cm}}ll}
\multicolumn{2}{l}{\emph{Exact surface}} 
 & \multicolumn{2}{l}{\emph{Polyhedral surface}} \\[3pt]
$\gamma$            & exact, smooth surface
 & $\bar{\Gamma}_h$    & $O(h^2)$ polyhedral approximation to $\gamma$\\
$\bnu$              & outward unit normal of $\gamma$
 & $\bar{\bnu}_h$      & outward unit normal of $\bar{\Gamma}_h$\\
$d$                 & signed distance function
 & $\bar{\mathcal{T}}_h$ & set of faces of $\bar{\Gamma}_h$\\
$\bp$               & closest point projection
 & $\bar{\bPi}_h$      & tangential projection with respect to $\bar{\Gamma}_h$\\
${\bf H}$           & Weingarten map
 & $\hat{K}$           & reference triangle\\
$\bPi$              & tangential projection w.r.t.\ $\gamma$
 & $F_{\bar{K}}$       & affine diffeomorphism\\
${\bm H}_T^1$       & $\bH^1$ tangential vector fields
 &                   & \\[6pt]
\multicolumn{2}{l}{\emph{High-order surface}} 
 & \multicolumn{2}{l}{\emph{Pullbacks}} \\[3pt]
$\Gamma_{h,k}$      & $O(h^{k+1})$ approximation to $\gamma$
 & $\bv^e$             & extension of $\bv$\\
$\mathcal{T}_{h,k}$ & triangulation of $\Gamma_{h,k}$
 & $\bv^\ell$          & lift of $\bv$\\
$\bPsi$             & polynomial diffeomorphism $\bar{\Gamma}_h \to \Gamma_{h,k}$
 & $\mathcal{P}_{\bPhi}$ & Piola transform w.r.t.\ $\bPhi$\\
$\ba_K$             & polynomial diffeomorphism
 & $\pt\bv = \mathcal{P}_{\bp}\bv$ & Piola transform w.r.t.\ $\bp$\\
$\mathcal{N}_{h,k}$ & Lagrange nodes of $\mathcal{T}_{h,k}$
 & $\ipt\bv = \mathcal{P}_{\bp^{-1}}\bv$ & Piola transform w.r.t.\ $\bp|_{\Gamma_{h,k}}^{-1}$\\
\end{tabular*}
}
\end{table}

\section{\rev{Preliminary Results}}\label{sec-2}

The signed distance function of $\gamma$ is denoted by $d$, which is well-defined in a tubular region of $\gamma$, denoted by $U_\delta$.
We set $\bnu = \nab d$, which is the extension of the outward unit normal of $\gamma$.
 The Weingarten map (or shape operator) is denoted by ${\bf H}(x) = D^2 d = \nab \bnu$,
 and the principal curvatures of $\gamma$ $\{\kappa_1,\kappa_2\}$ are the eigenvalues of ${\bf H}$, whose
 corresponding eigenvectors are orthogonal to $\bnu$.
Set $\bPi = {\bf I} - \bnu\otimes \bnu$ to be the tangential projection (where ${\bf I}$ denotes the $3\times 3$ identity matrix) 
and set the closest point projection $\bp:U_\delta\to \gamma$ as
$\bp(x) = x -d(x)\bnu(x)$. The space of $H^1$ tangential vectorfields is given by
\[
\bH^1_T(\gamma) = \{\bv\in \bH^1(\gamma):\ \bv\cdot \bnu = 0\}.
\]


\subsection{Triangulations}
We let $\bar \Gamma_h$ be a polyhedral surface approximation of $\gamma$
with simplicial faces such that $d$ is well defined on $\bar \Gamma_h$ with
$d(x) = \calO(h^2)$ for $x\in \rev{\bar \Gamma_h}$.  Let $\bar \calT_h$
 be the set of faces, and denote by $\bar \bnu_h$ the outward unit normal
of $\bar \Gamma_h$. The tangential projection operator with respect to \rev{$\bar \Gamma_h$} is given by
$\bar \bPi_h = {\bf I}-\bar \bnu_h\otimes \bar \bnu_h$.

Let $\hat K\subset \bbR^2$ be the reference triangle
with vertices $(1,0),(0,1),(0,0)$, and let $F_{\bar K}:\hat K\to \bar K$
be an affine diffeomorphism. We assume that $\bar \calT_h$ is shape-regular
in the sense that the following estimates are satisfied
\begin{subequations}
\label{eqn:FbarEst}
\begin{equation}
\|\nab F_{\bar K}\|_{L_\infty(\hat K)}\lesssim h_{\bar K},\qquad \det(\nab F_{\bar K}^\intercal \nab F_{\bar K}) \approx h_{\bar K}^4,
\end{equation}
where $h_{\bar K} = {\rm diam}(\bar K)$, and $\nab F_{\bar K}\in \bbR^{3\times 2}$ is the Jacobian of $F_{\bar K}$.
Here, we use
$a \lesssim b$ (resp., $a\gtrsim b$) to mean  there exists a generic constant $C>0$ independent
of the discretization parameter $h$ such that $a\le C b$ (resp., $a\ge C b$). 
The statement $a\approx b$ means $a\lesssim b$ and $a\gtrsim b$.

Since $(\nab F_{\bar K}^\intercal \nab F_{\bar K})^{-1} 
= {\rm det}(\nab F_{\bar K}^\intercal \nab F_{\bar K})^{-1}{\rm adj}(\nab F_{\bar K}^\intercal \nab F_{\bar K})$,
where ${\rm adj}(\cdot)$ denotes the matrix adjugate,
the above estimates yield
\begin{equation}
 \|(\nab F_{\bar K}^\intercal \nab F_{\bar K})^{-1}\|_{L_\infty(\hat K)}\lesssim h^{-2}_{\bar K}.
\end{equation}
\end{subequations}

We now define a family of high-order discrete surfaces $\Gamma_{h,k}$, $k \in \bbN$.  For $k=1$ we take 
$\Gamma_{h,1}=\bar \Gamma_h$ to be the affine surface approximation defined previously.  We assume
there exists  a continuous map $\bPsi: \bar \Gamma_h \rightarrow \mathbb{R}^3$ such that 
$\bPsi_{\bar K}:=\bPsi |_{\bar K} \in [\mathbb{P}_k(\bar K)]^3$, $\bar K \in \calT_h$ and satisfies
\begin{equation}\label{eqn:HighOrderGeometry}
|\bPsi(x)-\bp(x)|\lesssim h^{k+1},\quad \text{and}\quad |\nab \bPsi(x)-\nab \bp(x)|\lesssim h^k.
\end{equation}  
Let $\Gamma_{h,k}=\bPsi(\bar \Gamma_h)$ be the high-order surface approximation to $\gamma$,
and set $\calT_{h,k}= \{\bPsi(\bar{K}): \bar{K} \in \bar{\calT}_h\}$ to be the associated
triangulation of $\Gamma_{h,k}$.  The set of edges of $\bar \calT_h$ is denoted by $\bar \calE_h$,
and the set of edges of $\calT_{h,k}$ is $\calE_{h,k} = \{\Psi(\bar e):\ \bar e\in \bar \calE_h\}$.
We denote by $\bnu_h$ the outward unit normal
of $\Gamma_{h,k}$, and use the notation $\bnu_K = \bnu_h|_K$ for $K\in \calT_{h,k}$.
It follows from \eqref{eqn:HighOrderGeometry} that
\begin{equation}\label{eqn:normalApprox}
|\bnu- \bnu_h|\lesssim h^k.
\end{equation}
Similar to above, we set the tangential projection operator
with respect to $\Gamma_{h,k}$ as $\bPi_h = {\bf I}-\bnu_h\otimes \bnu_h$.
For a point on the discrete surface $a\in \Gamma_{h,k}$, we denote
by $\calT_a\subset \calT_{h,k}$ the set of faces containing $a$, i.e.,
$\calT_a = \{K\in \calT_{h,k}:\ a\in {\rm cl}(K)\}$.  We further define
\[
\calT_K = \{K'\in \calT_{h,k}:\ {\rm cl}(K')\cap {\rm cl}(K)\neq \emptyset\},\qquad \omega_K = \bigcup_{K\in \calT_K} K.
\]
Set $h_K = h_{\bar K}$, where $K = \bPsi(\bar K)$, and set $h = \max_{K\in \calT_{h,k}} h_K$.  We assume $\bar \calT_h$ is quasi-uniform in the sense that $h_K = h_{\bar K}\approx h$
for all $\bar {K}\in \bar \calT_h$.

The mapping $\bPsi$ may be viewed as an elementwise degree-$k$ polynomial mapping over $\bar \Gamma_h$, but concretely it is realized as a polynomial 
mapping from the reference element $\hat{K} \subset \mathbb{R}^2$.  In particular, we have
\[ 
\bPsi_{\bar{K}}= \ba_K \circ F_{\bar K}^{-1},
\]
where $\ba_K: \hat{K} \rightarrow K$ is a \rev{$k$th-degree polynomial diffeomorphism}
with $K = \bPsi(\bar K)\in \calT_{h,k}$.
We assume this polynomial mapping satisfies the following estimates.
\begin{equation}\label{eqn:akGood}
|\ba_K|_{W^{m}_\infty(\hat K)}\lesssim h^m,\qquad
{|\ba_K^{-1}|_{W^m_\infty( K)}\lesssim h^{-1}}\text{ ($1\le m\le (k+1))$},\quad
\det(\nab \ba_K^\intercal \nab \ba_K) \approx h^4,
\end{equation}
and note  these
 estimates imply  $L_\infty$ Sobolev norms of $\bPsi$ are bounded:
\begin{align}\label{eqn:bPsiHSBound}
|\bPsi|_{W^{m}_\infty(\bar K)} \lesssim |F_{\bar K}^{-1}|_{W^1_\infty(\bar K)}^{m} |\ba_K|_{W^{m}_\infty(\hat K)}\lesssim 1
\qquad \forall \bar K\in \bar \calT_h,\ \forall m\in \bbN.
\end{align}

We further assume that $F_{\bar K}(\hat a) = \ba_K(\hat a)$ for all vertices $\hat a$ of $\hat K$, i.e.,
$F_{\bar K}$ is the linear interpolant of $\ba_K$. It then follows
from the Bramble-Hilbert lemma and \eqref{eqn:akGood} that
\begin{equation}\label{eqn:baPert}
\|\ba_K - F_{\bar K}\|_{W^m_\infty(\hat K)}\lesssim h^2\quad \forall m\in \bbN_0:=\bbN\cup \{0\}.
\end{equation}

\begin{remark}
The estimates \eqref{eqn:HighOrderGeometry} and \eqref{eqn:akGood}
 hold if $\bPsi$ is the $k$th-degree Lagrange interpolant
of the closest-point projection $\bp$ \cite{Demlow09,Bernardi89}. 
However we do not necessarily assume this construction
in this paper. Rather, we assume  $\Gamma_{h,k}$ is a high-order approximation
to $\Gamma$ (in the sense of \eqref{eqn:HighOrderGeometry}) and
the mesh $\calT_{h,k}$ is shape-regular (in the sense of \eqref{eqn:akGood}).
\end{remark}

For $K\in \calT_{h,k}$, we let $K^\gamma = \bp(K)$ denote its image
on the exact surface $\gamma$ via the closest-point projection.
The collection of such elements
is $\calT_{h,k}^\gamma = \{K^\gamma:\ K\in \calT_{h,k}\}$.
We use the notation $\bv_K = \bv|_K$, the restriction of $\bv$ to $K$.
For a sub-mesh $\mathcal{D}_h\subset \calT_{h,k}$, 
we define the piecewise $H^m$-norms ($m\in \mathbb{N}_0$)
\[
\|\bv\|_{H^m_h(D_h)}^2 = \sum_{K\in \mathcal{D}_h} \|\bv\|_{H^m(K)}^2,\quad
\|\bv\|_{H^m_h(D^\gamma_h)}^2 = \sum_{K\in \mathcal{D}_{h}} \|\bv\|_{H^m(K^\gamma)}^2,
\]
where $D_h = \cup_{K\in \mathcal{D}_h} K$ and $D_h^\gamma = \bp(D_h)$.
We denote by $\bH^m_h(\gamma)$ the space of functions
satisfying $\|\bv\|_{H^m_h(\gamma)}<\infty$, \rev{where $\|\cdot\|_{H^m_h(\gamma)}$ 
is defined with respect to the partition $\calT_{h,k}^\gamma$.}

We end this section with a technical result
that will be used later on in Section \ref{sec-Consistency}.
\begin{lemma}
For an edge $e\in \calE_{h,k}$ (resp., $\bar e\in \bar \calE_h$), 
let $\bt_e$ (resp., $\bt_{\bar e}$) denote a unit tangent to $e$.
Then there holds for $h$ sufficiently small
\begin{equation}\label{eqn:Ptang}
\begin{split}
&|\nab \bp \bt_e|\gtrsim 1\text{ on every curved edge $e\in \calE_{h,k}$, and }\\
&|\nab \bPsi \bt_{\bar e}|\gtrsim 1\text{ on every affine edge $\bar e \in \bar \calE_h$.}
\end{split}
\end{equation}
\end{lemma}
\begin{proof}
Fix $e\in \calE_{h,k}$, and 
let $\bp(e)$ be the projected edge on $\gamma$
with unit tangent vector $\bt_{\bp(e)}$. We extend $\bt_{\bp(e)}$
to a neighborhood of $\gamma$ in the normal direction, 
and note that $|\bt_{\bp(e)}-\bt_e|\lesssim h^k$ on $e$.
Since $\nab \bp \bt_{\bp(e)} = (\rev{\bPi}-d{\bf H})\bt_{\bp(e)} = \bt_{\bp(e)} +\calO(h^{k+1}),$
we have
\[
|\nab \bp \bt_e| \ge |\nab \bp \bt_{\bp(e)}|-|\nab \bp| |\bt_{\bp(e)} - \bt_e| \ge 1 - C h^k \gtrsim 1
\]
for $h$ sufficiently small.

Likewise, for a fixed $\bar e\in \bar \calE_h$, let $\bp(\bar e)$ be the projected edge
on $\gamma$ with unit tangent $\bt_{\bp(\bar e)}$. Then by \eqref{eqn:HighOrderGeometry} and $|\bt_{\bar e}-\bt_{\bp(\bar e)}|\lesssim h$, we have
\[
|\nab \bPsi \bt_{\bar e}| 
\ge |\nab \bp \bt_{\bp(\bar e)}| - |\nab \bp| |\bt_{\bp(\bar e)}-\bt_{\bar e}| - |\nab \bPsi-\nab \bp| |\bt_{\bar e}| \ge 1-C h \gtrsim 1.
\]
\end{proof}

\subsection{Extensions and Lifts}

For a (scalar or vector-valued)  function $w$ defined on the exact surface $\gamma$,
its extension $w^e$ is given by $w^e(x) = \rev{w(\bp(x))}$ for $x\in U_\delta$.
On the other hand, for a function $w$ defined on either $\bar \Gamma_h$
or $\Gamma_{h,k}$, we let $\tilde w$ satisfy $\tilde w(x) = w(\bq(x))$ for all $x\in \gamma$,
where $\bq$ satisfies $\bp(\bq(x)) = x$. We then define the lift $w^\ell(x) = \tilde w( \bp(x))$ 
for $x\in U_\delta$.

Setting $\mu_h:\Gamma_{h,k}\to \bbR$ and $\bar \mu_h:\bar \Gamma_h\to \bbR$ as
\[
\mu_h  = \bnu\cdot \bnu_h(1-d \kappa_1)(1-d\kappa_2),\qquad \bar \mu_h  = \bnu\cdot \bar \bnu_h(1-d \kappa_1)(1-d\kappa_2),
\]
we have \cite{Demlow09}
\[
\int_{\Gamma_{h,k}} w^e \mu_h =  \int_{\bar \Gamma_h} w^e \bar \mu_h = \int_{\gamma} w,\qquad \forall w\in L_1(\gamma).
\]
We also have the following change of variables
result with respect to integration across edges. 
Its proof is given in the appendix.
\begin{lemma}\label{lem:ChangeVarE}
Let edge $\bar e$ and $e$ be edges in $\bar \calT_h$ and $\calT_{h,k}$, respectively, with $e = \bPsi(\bar e)$,
and let $ \bt_{\bar e}$ denote a unit tangent vector to $\bar e$.
There holds
\begin{align*}
\int_e q 
& = \int_{\bar e} \mu_e q\circ \bPsi\qquad \forall q\in L_1(e),
\end{align*}
where $\mu_e = |\nab \bPsi \bt_{\bar e}|.$
\end{lemma}

\subsection{Piola Transforms}
We follow the setup and notation in \cite{Steinmann08,CD16,SurfaceStokes1,DemlowNeilan23}
to describe the Piola transform between two surfaces in $\bbR^3$.
Let $\bPhi:\calS_0\to \calS_1$ be a diffeomorphism 
between two surfaces $\calS_0$ and $\calS_1$.
 Let $d\sigma_i$ be
the surface measure of $\calS_i$,
and let $\mu$  satisfy $\mu d\sigma_0 = d\sigma_1$.
The Piola transform of a vectorfield $\bv:\calS_0\to \bbR^3$ 
with respect to $\bPhi$
is given by 
\[
(\calP_{\bPhi} \bv)\circ \bPhi= \mu^{-1} \nab\bPhi \bv,
\]
whereas the Piola transform for a vectorfield $\bv:\calS_1\to \bbR^3$
with respect to the inverse mapping $\bPhi^{-1}$ is given by 
\[
(\calP_{\bPhi^{-1}} \bv)\circ \bPhi^{-1} = (\mu\circ \bPhi^{-1}) \nab \bPhi^{-1} \bv.
\]
The mappings $\calP_{\bPhi}:\bH({\rm div}_{\calS_0};\calS_0)\to \bH({\rm div}_{\calS_1};\calS_1)$
and $\calP_{\bPhi^{-1}}:\bH({\rm div}_{\calS_1};\calS_1)\to \bH({\rm div}_{\calS_0};\calS_0)$
are bounded, and in particular, satisfy
${\rm div}_{\calS_0} \bv = \mu {\rm div}_{\calS_1} \calP_{\bPhi}  \bv$ for all $\bv\in \bH({\rm div};\calS_0).$

{\begin{remark}
In what follows, we make a slight abuse of notation
and treat the symbol of the Piola transform $\calP_{(\cdot)}$ as both an operator 
and a matrix. For example, we write $(\calP_{\bPhi} \bv)\circ \bPhi(x) = \calP_{\bPhi}(x) \bv(x)$,
where on the left-hand side $\calP_{\bPhi} \bv:\calS_1\to \bbR^3$, whereas on 
the right-hand side, both $\calP_{\bPhi}= \mu^{-1}\nab \bPhi$ and $\bv$ are functions defined on $\calS_0$.
This leads to (admittedly unusual) statements such as $\|\calP_{\bPhi}\bv\|_{L_p(\calS_1)} = \|\mu^{1/p} \calP_{\bPhi} \bv\|_{L_p(\calS_0)}$.
\end{remark}}

In the case $\bPhi = \bp$, $\calS_0 = \Gamma_{h,k}$,
and $\calS_1 = \gamma$,
the Piola transform of $\bv:\Gamma_{h,k}\to \bbR^3$ with $\bPi_h \bv= \bv$
is  \cite{CD16,SurfaceStokes1,DemlowNeilan23}
\begin{equation}\label{eqn:Piola}
\pt \bv \circ \bp := \calP_{\bp}  \bv =  \frac1{\mu_h} \big[\bPi - d {\bf H}\big]  \bv,
\end{equation}
whereas the Piola transform of $\bv:\gamma\to \bbR^3$
with respect to the inverse $\bp^{-1}$ is given by
\begin{equation}\label{eqn:invPiola}
\ipt \bv := \calP_{\bp^{-1}}\bv = \mu_h \Big[{\bf I}- \frac{\bnu\otimes \bnu_h}{\bnu\cdot \bnu_h}\Big][{\bf I}-d {\bf H}]^{-1} (\bv \circ \bp). 
\end{equation}

Due to the smoothness of $\bp$ and local smoothness of $\bnu_h$, 
the function $\mu_h = \bnu\cdot \bnu_h \det({\bm I}-d {\bf H})$ is smooth
on each $K\in \calT_{h,k}$, and so is its reciprocal. In particular $|\mu_h^{-1}|_{W^m_\infty(K)}\lesssim 1$ 
and $\|\mu\|_{W^m_\infty(K)}\lesssim 1$ for $m\in \bbN_0$.
It thus follows that
\begin{equation}
\label{eqn:PiolaPBounds}
\|\calP_{\bp}\|_{W^m_\infty(K)}\lesssim 1,\quad \text{and}\quad \|\calP_{\bp^{-1}}\|_{W^m_\infty(K^\gamma)}\lesssim 1
\end{equation}
on each $K\in \calT_{h,k}$ and $K\in \calT_{h,k}^\gamma$, and so
\begin{align}\label{eqn:ugh}
\|\pt{\bv}\|_{W^{m}_p(K^\gamma)}\approx  \|\bv\|_{W^{m}_p(K)}\qquad \forall \bv\in {\bm W}_p^m(K),\ \forall K\in \calT_{h,k},\ m\in \mathbb{N}_0,\ p\in [1,\infty],
\end{align}
with $\pt{\bv} = \calP_{\bp} \bv$.

The proofs of the following three lemmas are found in the appendix.
\begin{lemma}\label{lem:xPert}
Assume \eqref{eqn:FbarEst}--\eqref{eqn:baPert}, and let $\bar K\in \calT_h,\ K\in \calT_{h,k}$.
Then there holds
\begin{equation}
\begin{aligned}
|\calP_{\bPsi_{\bar K}} \bx - \bx|\lesssim h |\bx|\qquad &&\text{for all $\bx$ tangent to $\bar K$,}\\
|\calP_{\bPsi^{-1}_{K}} \bx - \bx|\lesssim h |\bx|\qquad &&\text{for all $\bx$ tangent to $K$.}
\end{aligned}
\end{equation}
\end{lemma}

\begin{lemma}\label{lem:PiolaOfa}
There holds for all $K\in \calT_{h,k}$, $\bar K\in \bar \calT_h$, and $m \in \mathbb{N}_0$
\begin{alignat}{2}
\label{eqn:PiolaOfa}
|\calP_{\ba_K}|_{W^m_\infty(\hat K)}&\lesssim h^{m-1},\qquad
&&
|(\calP_{\ba_K})^\dagger|_{W^m_\infty(\hat K)}\lesssim h^{1+m},\\
\label{eqn:PiolaBounds}
\|\calP_{\bPsi}\|_{W^m_\infty(\bar K)}&\lesssim 1, \qquad &&\|\calP_{\bPsi^{-1}}\|_{W^m_\infty(K)}\lesssim 1,
\end{alignat}
where $(\calP_{\ba_K})^\dagger$ is the pseudo-inverse of $\calP_{\ba_K}$. Consequently 
if $\bw\in \bH^m(K)$ and $\hat \bw\in \bH^m(\hat K)$ are related via $\bw = \calP_{\ba_K} \hat \bw$, then there holds
\begin{align}\label{eqn:ScalingOfPiolaA}
 |{\bw}|_{H^m(K)}\lesssim \sum_{\ell=0}^m h^{-\ell} |\hat \bw|_{H^\ell(\hat K)},\qquad |\hat \bw|_{H^m(\hat K)}\lesssim h^m\|{\bw}\|_{H^m(K)}.
\end{align}
Moreover, for all $\bar \bv\in \bH^m(\bar K)$ there holds
\begin{align}
\label{eqn:PhiPiolaNormEquiv}
\|\calP_{\bPsi} \bar \bv\|_{H^m(K)}\approx \|\bar \bv\|_{H^m(\bar K)},\quad K = \bPsi(\bar K).
\end{align}
\end{lemma}

\begin{lemma}\label{lem:DefCons}
There holds  for $\bv\in \bH^1_T(K)$
\begin{equation}\label{eqn:DefChain}
|{\rm Def}_\gamma \pt{\bv} - ({\rm Def}_{\Gamma_{h,k}} \bv)\circ \bp^{-1}|\lesssim h^{k} \left(|(\nab_{\Gamma_{h,k}} \bv)\circ \bp^{-1}| +
|\bv\circ \bp^{-1}|\right) \quad \text{on }K^\gamma.
\end{equation}
\end{lemma}

\section{Finite Element Spaces}\label{sec-3}
Recall in the Euclidean case, the Taylor--Hood pair takes the 
velocity space to be the vector-valued  Lagrange space of degree $r\ge 2$,
whereas the pressure space is the (scalar)  Lagrange space of degree $(r-1)$.
Analogously, on the affine surface $\bar \Gamma_h$ we let  $\bar Q_h$ be the space of continuous, piecewise  polynomials of degree $(r-1)$.
On the high-order surface approximation, we define the pressure space via standard composition:
\begin{equation}\label{eqn:THQhDef}
Q_h = \{q\in H^1(\Gamma_{h,k})\cap \mathring{L}_2(\Gamma_{h,k}):\ q\circ \bPsi|_{\bar K}\in \mathbb{P}_{r-1}(\bar K)\ \forall \bar K\in \bar \calT_h\},
\end{equation}
where $\mathbb{P}_{r-1}(\bar K)$ denotes the space of polynomials of degree $\le (r-1)$
with domain $\bar K$, and $\mathring{L}_2(\Gamma_{h,k})$ is the space of square-integrable
functions on $\Gamma_{h,k}$ with vanishing mean.
Note that $Q_h$ is equivalent
to the space of functions $q$ in $H^1(\Gamma_{h,k})\cap \mathring{L}_2(\Gamma_{h,k})$
such that $q\circ \ba_K|_{\hat K}\in \mathbb{P}_{r-1}(\hat K)$ for each $K\in \calT_{h,k}$.

Let $\bar \calN_h$ denote the set of degrees of freedom (DOFs) of the $r$th-degree Lagrange finite element
space on $\bar \calT_h$, i.e., $\bar \calN_h$ consists of (i) the vertices in $\bar \calT_h$;
(ii) $(r-1)$ points on each (open) edge $\bar e\in \bar \calE_h$; (iii) $\frac12(r-1)(r-2)$ points
in the interior of each $\bar K\in \bar \calT_h$, chosen such that they uniquely determine
a polynomial of degree $(r-3)$. The corresponding set of points are mapped to $\Gamma_{h,k}$ via $\bPsi$:
\[
\calN_{h,k} = \{\bPsi(\bar a):\ \bar a\in \rev{\bar \calN_h}\}.
\]
We also set for each $K\in \calT_{h,k}$, 
\[
\calN_K = \{a\in \calN_{h,k}:\ a\in {\rm cl}(K)\}
\]
to be the location of these degrees of freedom on $K$.
We make the mild assumption that these degrees
of freedom are the image of a set of reference DOFs under $\ba_K$, 
i.e., there exists a set $\calN_{\hat K}\subset {\rm cl}(\hat K)$
such that $\calN_K = \{\ba_K(\hat a):\ \hat a\in \calN_{\hat K}\}$.

For each $K\in \calT_{h,k}$, we define the local velocity space,
which is the space of  Piola-mapped polynomials of degree $r$:
\[
\bV(K) = \{\calP_{\ba_K} \hat \bv:\ \hat \bv\in [\mathbb{P}_r(\hat K)]^2\}.
\]
To piece together these local spaces and to define
the global velocity space, we require a definition along with a technical result.
\begin{definition}\label{def:MaK}
For each DOF $a\in \calN_{h,k}$ in the triangulation, 
we arbitrarily choose a single (fixed) face $K_a\in \calT_a$.
For $K\in \calT_a$, we define
$\calM_a^K:\bbR^3\to \bbR^3$ by
\begin{equation}\label{eqn:calMDef}
\calM_a^K \bx= \Big(\bnu_{K_a}(a) \cdot \bnu_K(a) \Big[ {\bf I} - \frac{\bnu_{K_a}(a) \otimes \bnu_{K}(a)}{\bnu_{K_a}(a)\cdot \bnu_{K}(a)}\Big]\Big) \bx,
\end{equation}
where we recall $\bnu_{K_a}$ and $\bnu_{K}$ are the outward unit normals of $K_a$ and $K$, respectively.
In particular, $\calM_a^K \bx$ is the Piola transform of $\bx$ with respect
to the inverse of the closest point projection onto the plane tangent to $K_a$ at $a$ \cite{DemlowNeilan23}.
\end{definition}

\begin{remark} While our algorithm functions by choosing a master element $K_a$ at each node $a$, 
it is more appropriately interpreted as assigning an approximate tangent plane at each node (cf.~Figure \ref{fig:dofs}).
The algorithm does not have access to the continuous normal $\bnu$ or equivalently the tangent plane to $\gamma$.  In essence, our algorithm functions by choosing a reasonable but somewhat arbitrary approximation to $\bnu$ at each node and then computing degrees of freedom based on this choice.  
In the absence of the actual tangent plane of $\gamma$,
there does not appear to be a canonical or ``best'' choice of approximate tangent plane for this purpose.  
\end{remark}

\begin{figure}
\begin{center}
\begin{tikzpicture}[line join = round, line cap = round,scale=2.25]

\coordinate(A) at (0,0,1.5);
\coordinate(B) at (0,0,-1.5);
\coordinate(C) at (2,-1,0);
\coordinate(D) at (-1.5,-0.5,0);

\coordinate(D) at (-2.5,-1,-0.5);
\coordinate(E) at (1.5,0,-0.75);

\coordinate(cent1) at ($(C)!0.5!(B)!0.33333333!(A)$);
\coordinate(cent2) at ($(D)!0.5!(B)!0.33333333!(A)$);
\coordinate(cent3) at ($(E)!0.5!(B)!0.33333333!(C)$);
\coordinate(norm1) at (0.447213595500000,0.894427191000000,0);
\coordinate(norm2) at (-0.316227766050000, 0.948683298150000,0);
\coordinate(norm3) at (0.408248290500000, 0.408248290500000,-0.816496581000000);

\draw[thick,fill=green,opacity=0.2] (A)--(B)--(C)--(A);

\draw[thick,fill=yellow,opacity=0.2] (A)--(B)--(D)--(A);

\draw[thick,fill=blue,opacity=0.2]  (B)--(C)--(E)--(B);

\node[below] at (B) {$a$};

 \node[inner sep = 0pt,minimum size=2.5pt,fill=black!100,circle] (n2) at (B) {};



\node at ($(cent3)+(0.2,-0.2)$) {$K_3$};
\node at ($(cent2)+(0.15,0)$) {$K_1$};
\node at ($(cent1)+(-0.2,0)$) {$K_2$};

\coordinate(t) at (0.9786040069128327, 0.14185302893081675, -0.149036626732259);

\coordinate(offset) at (0,0,0);
\coordinate(BO) at ($(B)+(offset)+(t)$);


\draw[->,very thick,magenta] 
  ($(B)+(offset)$) -- ($(BO)$)
  node[midway, above, yshift=1pt,black] { ${\bm t}_{\rm ref}$};

\coordinate(M1) at (0.7748791201606419, -0.387439560080321, 0.05087101482465431);
\coordinate(M2) at (0.821883420784295, 0.2739611402614317, -0.17748803448398376);
\coordinate(M3) at (0.5927175697169547, -0.426761269903828, 0.08297814990656333);

\draw[->,very thick] (B) -- ($(B)+(M1)$)
node[pos=0.85, below right, yshift=-8pt] { ${\bm t}_3 = ((\bnu_{\rm ref}\cdot \bnu_3) {\bf I}_3-\bnu_{\rm ref}\otimes \bnu_3){\bm t}_{\rm ref}$};

\coordinate(norm) at ($ 0.333333333*(norm1) + 0.333333333*(norm2) + 0.333333333*(norm3) $);

\coordinate(P1) at (-2.34138088, -0.32264498, -3.93594658);
\coordinate(P2) at ( 2.34138088, -1.31386322, -3.57646310);
\coordinate(P3) at ( 2.34138088,  0.32264498,  0.43594658);
\coordinate(P4) at (-2.34138088,  1.31386322,  0.07646310);

\filldraw[fill=magenta, fill opacity=0.15, draw=magenta, thick]
  (P1) -- (P2) -- (P3) -- (P4) -- cycle;

\draw[->,magenta,very thick] (B) -- ($(B)+0.9*(norm)$);
\node[above right] at ($(B)+0.9*(norm)$) { $\mathbf{\bnu}_{\mathrm{ref}}$};

\end{tikzpicture}
\caption{\label{fig:dofs}\footnotesize 
\rev{A visual description
of constructing tangent planes at each node
via the operator $\calM_a^K$. In the} \rev{construction of the finite element space, 
we take $\bnu_{\rm ref} = \bnu_{K_a}(a)$.}
}
\end{center}
\end{figure}

\begin{lemma}\label{lem:PiolaD} 
Fix $a\in \calN_{h,k}$ and let $\bu$ lie in the tangent plane of $\gamma$ at 
$\bp(a)$.  For $K\in \calT_a$, let $\ipt{\bu}_K = \calP_{\bp^{-1}} \bu|_K$ be the Piola transform of $\bu$ to $K$ via the inverse of the closest point projection.
Then 
\begin{align}
\label{piola_defect}
|\ipt{\bu}_K-\mathcal{M}_a^K \ipt \bu_{K_a}|
\lesssim h^{2k} |\ipt \bu_{K_a}| \le h^{2k} |\bu|.
\end{align} 
\end{lemma}
\begin{proof}
The proof for the case $k=1$ is found in \cite[Lemma 2.5]{DemlowNeilan23}. The arguments
given there generalize to arbitrary $k\ge 1$ to yield \eqref{piola_defect}; see \cite[Remark 2.6]{DemlowNeilan23}.
\end{proof}

Motivated by Lemma  \ref{lem:PiolaD},
we define the global velocity space as
\begin{equation}\label{eqn:THVhDef}
\bV_h = \{\bv\in \bL_2(\Gamma_{h,k}):\ \bv_K\in \bV(K),\ \forall K\in \calT_{h,k}; \bv_K(a) = \calM_a^K(\bv_{K_a}(a))\ \forall K\in \calT_a,\ \forall a\in \calN_{h,k}\}.
\end{equation}

\begin{lemma}\label{lem:vn3Points}
Let $\bv\in \bV_h$, and 
for an edge $e$ with $e = \p K_1\cap \p K_2$ ($K_1,K_2\in \calT_{h,k}$),
set $\bv_j = \bv|_{K_j}$.  Let $\bn_j$ be the outward unit co-normal
of $\p K_j$ restricted to $e$. 
Then there holds
\[
(\bv_1\cdot \bn_1 + \bv_2 \cdot \bn_2)(a) = 0\qquad \forall {\rm cl}(e) \ni a\in \calN_{h,k}.
\]
\end{lemma}
\begin{proof}
This result  follows from the (in-plane) normal-preserving properties
of the Piola transform. See \cite[Proposition 3.1]{DemlowNeilan23} for details.
\end{proof}

\begin{lemma}
There holds $\bV_h\subset \bH({\rm div}_{\Gamma_{h,k}};\Gamma_{h,k})$.
\end{lemma}
\begin{proof}
It suffices to show that co-normal components functions in $\bV_h$
are single-valued on the edges in $\calT_{h,k}$.

Let $\bv\in \bV_h$, set $\bar \bv = \calP_{\bPsi^{-1}} \bv$,
and let $\bar K\in \bar \calT_h$ and $K\in \calT_{h,k}$ be related
via $K = \bPsi(\bar K)$.
Writing $\bv_K :=\bv|_K =  \calP_{\ba_K} \hat \bv$ for some $\hat \bv\in [\mathbb{P}_r(\hat K)]^2$,
we have $\bar \bv_{\bar K} :=\bar \bv|_{\bar K} =  \calP_{\bPsi^{-1}} \calP_{\ba_K} \hat \bv = \calP_{F_{\bar K} \circ \ba_K^{-1}} \calP_{\ba_K} \hat \bv= \calP_{F_{\bar K}} \hat \bv$.
Thus $\bar \bv_K$ is a polynomial of degree $\le r$.

Let $e$ be an edge of $\calT_{h,k}$ and $\bar e$ an edge of $\bar \calT_h$
with $e = \bPsi(\bar e)$.  Write $e = \p K_1\cap \p K_2$ with $K_1,K_2\in \calT_{h,k}$,
and let $\bar K_1,\bar K_2\in \bar \calT_h$ be the corresponding
affine elements.  The normal-preserving properties
of the Piola transform read
\begin{align*}
\int_{e}  (\bv_j \cdot \bn_j)q = \int_{\bar e} (\bar \bv_j\cdot \bar \bn_j)\bar q\qquad \forall q\in L_2(e),
\end{align*} 
where $\bv_j = \bv|_{K_j},\ \bar\bv_j = \bar\bv|_{\bar K_j}$, and  $\bar q = q\circ \bPsi$.  
An application of Lemma \ref{lem:ChangeVarE} yields
\begin{align*}
\int_{\bar e} (\bar \bv_j\cdot \bar \bn_j) \bar q = \int_e (\bv_j \cdot \bn_j)q = \int_{\bar e} \mu_e (\bv_j \cdot \bn_j)\circ \bPsi \bar q\qquad \forall q\in L_2(e).
\end{align*}
Thus, $\bar \bv_j\cdot \bar \bn_j = \mu_e ( \bv_j\cdot  \bn_j)\circ \bPsi$ and so
\begin{align*}
\bv_1 \cdot \bn_1 + \bv_2 \cdot \bn_2 = \Big(\mu_e \big(\bar \bv_1\cdot \bar \bn_1+\bar \bv_2\cdot \bar \bn_2)\Big)\circ \bPsi^{-1}\qquad \text{on }e.
\end{align*}
Lemma \ref{lem:vn3Points} shows  that the left-hand side vanishes at $(r+1)$ distinct points on ${\rm cl}(e)$,
and therefore $\bar \bv_1\cdot \bar \bn_1+\bar \bv_2\cdot \bar \bn_2$ vanishes at $(r+1)$ points on $\bar e$.
Since $\bar \bv\cdot \bar \bn_j$ is a   polynomial of degree $\le r$ on $\bar e$,
we conclude $\bar \bv_1 \cdot \bar \bn_1 + \bar \bv_2 \cdot \bar \bn_2=0$ on $\bar e$.
Therefore $ \bv_1 \cdot  \bn_1 +  \bv_2 \cdot  \bn_2=0$ on $ e$, and
so $\bv\in \bH({\rm div}_{\Gamma_{h,k}};\Gamma_{h,k})$. 
\end{proof}

\subsection{Scaling and Approximation Properties of $\bV_h$}

In this section, we gather some scaling and approximation results
of the velocity space $\bV_h$, showing that it has similar
properties as the (isoparametric) Lagrange finite element space.
For example, we have the following inverse estimate.
\begin{lemma}
There holds for all $\bv\in \bV(K)$ and $m,s\in \mathbb{N}_0$ with $s\le \min\{r,m\}$,
\begin{align}\label{eqn:HOIE}
\|\bv\|_{W^{m}_\infty(K)}\lesssim h^{-1+s-\min\{r,m\}} \|\bv\|_{H^s(K)}\qquad \forall \bv\in \bV(K),\ \forall K\in \calT_{h,k}.
\end{align}
\end{lemma}
\begin{proof}
Write $\bv = \calP_{\ba_K} \hat \bv$ for some $\hat \bv\in [\mathbb{P}_r(\hat K)]^2$.
By similar arguments as in the proof of Lemma \ref{lem:PiolaOfa} (cf. \eqref{eqn:wbound}), we have for any $m\in \mathbb{N}$
\[
|\bv|_{W^{m}_\infty(K)}\lesssim h^{-1} \sum_{\ell=0}^{m}  h^{-\ell} |\hat \bv|_{W^\ell_\infty(\hat K)} = h^{-1} \sum_{\ell=0}^{\min\{r,m\}}  h^{-\ell} |\hat \bv|_{W^\ell_\infty(\hat K)}
\lesssim  \sum_{\ell=0}^{s-1} h^{-\ell-1} |\hat \bv|_{H^\ell(\hat K)}+h^{-\min\{r,m\}-1}|\hat \bv|_{H^s(\hat K)},
\]
where we used the fact that $\hat \bv$ is a polynomial of degree at most $r$ and equivalence of norms in the last two steps.
We then map back to $K$, using Lemma \ref{lem:PiolaOfa},
\[
|\bv|_{W^{m}_\infty(K)}\lesssim  h^{-1} \| \bv\|_{H^{s-1}(\hat K)}+h^{-1+s-\min\{r,m\}} \| \bv\|_{H^s(K)}\lesssim h^{-1+s-\min\{r,m\}}\|\bv\|_{H^s(K)}.
\]
\end{proof}

\begin{lemma}\label{lem:Scaling}
There holds
\begin{equation}
h^{m-1}\|\bv\|_{H^m(K)} \lesssim \sum_{a\in \calN_K} |\bv(a)|\qquad \forall \bv\in \bV(K),\ \forall K\in \calT_{h,k}\qquad m=0,1,\ldots,r.
\end{equation}
\end{lemma}
\begin{proof}
Write $\bv = \calP_{\ba_K} \hat \bv$ with $\hat \bv\in [\mathbb{P}_k(\hat K)]^2$.
Applying Lemma \ref{lem:PiolaOfa} and equivalence of norms we have
\begin{align*}
|\bv|^2_{H^m( K)}\lesssim h^{-2m}\|\hat \bv\|^2_{H^m(\hat K)}\lesssim h^{-2m} \sum_{\hat a\in \calN_{\hat K}} |\hat \bv(\hat a)|^2.
\end{align*}
We then apply the estimate \eqref{eqn:PiolaOfa} to obtain
\begin{align*}
|\bv|^2_{H^m( K)}
\lesssim h^{-2m} \sum_{\hat a\in \calN_{\hat K}} |(\calP_{\ba_K})^\dagger \calP_{\ba_K} \hat \bv(\hat a)|^2
\lesssim h^{-2m} \|(\calP_{\ba_K})^\dagger\|_{L_\infty(\hat K)}^2  \sum_{\hat a\in \calN_{\hat K}}|\bv(a)|^2\le h^{-2m+2}
\sum_{\hat a\in \calN_{\hat K}}|\bv(a)|^2.
\end{align*}
\end{proof}

\begin{lemma}\label{lem:ApproxProps}
For  $\bw\in \bC(\gamma)\cap \bH^1_T(\gamma)\cap \bH^s_h(\gamma)\ (2\le s\le r+1)$, 
\rev{let $\ipt\bw = \calP_{\bp^{-1}} \bw$.}
Then there exists $\bI_h \ipt{\bw}\in \bV_h$ such that
\begin{align}\label{eqn:InterpolantEstimate}
h^m \|\ipt{\bw}-\bI_h \ipt{\bw}\|_{H^m(K)}\lesssim h^s \|\bw\|_{H^s(K^\gamma)}+h^{2k}\|\bw\|_{H^2(\omega_{K^\gamma})},\quad m=0,1,\ldots,r \hbox{ and } m \le s.
\end{align}
\end{lemma}
\begin{proof}
This result essentially follows from the arguments in 
\cite[Lemma 3.2]{DemlowNeilan23}, so we only sketch
the main ideas.

Given $\bw\in \bC(\gamma)\cap \bH^1_T(\gamma)\cap \bH^s_h(\gamma)$,
we uniquely define  $\bv =\bI_h \ipt{\bw} \in \bV_h$ such that
\[
\bv_{K_a}(a) = \ipt{\bw}_{K_a}(a)\qquad \forall a\in \calN_{h,k},
\]
where $\ipt{\bw} = \calP_{\bp^{-1}} \bw$, with $\bp^{-1}:\gamma \to \Gamma_{h,k}$.

Define $\ipt \bw_I$ such that $(\ipt \bw_I)_K\in \bV(K)$ for each $K\in \calT_{h,k}$
and $(\ipt \bw_I)_K(a) = \ipt{\bw}_K(a)$ for all nodes $a\in \calN_{h,k}$
and elements $K\in \calT_{h,k}$.
Then the arguments in \cite[Lemma 3.2]{DemlowNeilan23} and \eqref{piola_defect} yield
\[
|(\bv-\ipt\bw_I)_K(a)|\lesssim h^{2k} |\ipt{\bw}_{K_a}(a)|.
\]
Thus by
Lemma \ref{lem:Scaling} and a scaled Sobolev embedding we obtain
\begin{align*}
h^m \|\ipt{\bw}_I - \bv\|_{H^m(K)}
&\lesssim h \max_{a\in \calN_K} |(\ipt{\bw}_I-\bv)_K(a)|\lesssim h^{2k+1}\max_{a\in \calN_K} |\ipt{\bw}_{K_a}(a)|\lesssim h^{2k}\|\ipt \bw\|_{H^2_h({\omega_{K}})}.
\end{align*}
A scaling argument and the Bramble-Hilbert lemma
yields $h^m \|\ipt{\bw}-\ipt{\bw}_I\|_{H^m(K)}\lesssim h^s \|\ipt\bw\|_{H^s(K)}$.
Therefore by the triangle inequality and \eqref{eqn:ugh},
we obtain
\[
h^m \|\ipt\bw - \bv\|_{H^m(K)}\lesssim h^s \|\ipt \bw\|_{H^s(K)} + h^{2k} \|\ipt \bw\|_{H^2_h(\omega_K)}{\lesssim h^s \|\bw\|_{H^s(K^\gamma)}+h^{2k} \|\bw\|_{H^2_h(\omega_{K^\gamma})}.}
\] \end{proof}

\begin{lemma}
Given $\bw\in \bH^1_T(\gamma)$,
there exists ${\bI_{h,1}} \bw\in \bV_h$ such that
\begin{equation}\label{eqn:IhwProp}
\|{\bI_{h,1}} \bw\|_{H^1_h(\Gamma_{h,k})}\lesssim \|\bw\|_{H^1(\gamma)}\quad \text{and}\quad 
\left(\sum_{K\in \calT_{h,k}} h_K^{-2} \|\bw-\pt{\bI_{h,1} \bw}\|^2_{L_2(K^\gamma)}\right)^{1/2}\lesssim \|\bw\|_{H^1(\gamma)}.
\end{equation}
\end{lemma}
\begin{proof}
\rev{For a given $\bw\in \bH^1_T(\gamma)$, we first map the function to the 
polyhedral surface by composition, $\bar{\bw}:=\bw \circ \bp \circ \bPsi$.
We then let $\bar{\bw}_h$ be the componentwise Scott-Zhang interpolant of $\bar{\bw}$ into the continuous piecewise tri-linear Lagrange space \cite{ScottZhang90} with respect to $\bar \calT_h$.
We map this interpolant back to the surface and apply the tangential projection to obtain 
$\bw_h := \bPi (\bar{\bw}_h\circ \bPsi^{-1} \circ \bp^{-1}) \in \bC(\gamma) \cap \bH_T^1(\gamma) \cap \bH_h^2(\gamma)$. 
Due to the regularity of $\bw_h$, its interpolant $\bI_{h,1} \bw:=\bI_h \ipt \bw_h\in \bV_h$ given in Lemma \ref{lem:ApproxProps} is well-defined.}

Given $K \subset \Gamma_{h,k}$, applying norm equivalences, \eqref{eqn:InterpolantEstimate}, the fact that second-order
derivatives of $\bar \bw_h$ vanish on $\bar \Gamma_h$, 
and standard properties of the Scott-Zhang interpolant yields
$$\begin{aligned}
\|\bI_{h,1} \bw \|_{H^1(K)} & \le  \|\bI_h \ipt{\bw}_h-\ipt{\bw}_h\|_{H^1(K)} + \|\ipt{\bw}_h\|_{H^1(K)}
\\ & \lesssim h \|\bw_h \|_{H^2_h(\omega_{K^\gamma})} + \|{\bw}_h\|_{H^1( K^\gamma)}
\\ &\lesssim h \|\overline{\bw}_h\|_{H^2_h(\omega_{\bar K})} + \|\bar \bw_h\|_{H^1(\bar K)}
\\ &\lesssim \|\overline{\bw}_h\|_{H^1(\omega_{\bar K})}
\\ & \lesssim \|\bar{\bw}\|_{H^1(\omega'_{\bar K})} \lesssim \|\bw\|_{H^1(\omega'_{K^\gamma})},
\end{aligned} $$
where
\[
\omega'_{\bar K} = \mathop{\bigcup_{\bar K\in \bar \calT_h}}_{{\rm cl}(\bar K)\cap {\rm cl}(\omega_{\bar K})\neq \emptyset} K,
\]
and an analogous definition for $\omega'_{K^\gamma}$.

By similar arguments, we have
\begin{align*}
\|\bw-\pt{\bI_{h,1} \bw}\|_{L_2(K^\gamma)}
&\le \|\bw-\bw_h\|_{L_2(K^\gamma)}+\|\bw_h-\pt{\bI_{h,1}\bw}\|_{L_2(K^\gamma)}\\
&\lesssim \|\bar \bw-\bar \bw_h\|_{L_2(\bar K)} + \|\ipt{\bw}_h-\bI_{h}\ipt{\bw}_h\|_{L_2(K)}\\
&\lesssim h \|\bar \bw\|_{H^1(\omega_{\bar K})} + h^2 \| \bar \bw_h\|_{H^2_h(\omega_{K^\gamma})}
\lesssim h\|\bw\|_{H^1(\omega'_{K^\gamma})}.
\end{align*}
The desired result \eqref{eqn:IhwProp} now follows after summing over $K\in \calT_{h,k}$.
\end{proof}

\subsection{Inf-Sup stability}
In this section we establish that $\bV_h\times Q_h$ is an inf-sup stable
 pair for the surface Stokes problem. The proof relies
 on a variant of Verf\"urth's trick, analogous to the standard
 stability argument for the Euclidean Taylor--Hood pair.
We begin by showing 
a stability result with respect to a weighted $H^1$-norm.

\begin{lemma} 
There holds 
\begin{equation}\label{eqn:THKMInfSup}
\sup_{\bv\in \bV_h} \frac{\int_{\Gamma_{h,k} } ({\rm div}_{\Gamma_{h,k}} \bv)q}{\|\bv\|_{H^1_h(\Gamma_{h,k})}} \ge c_1 \|h \nab_{\Gamma_{h,k}} q\|_{L_2(\Gamma_{h,k})}\qquad \forall q\in Q_h,
\end{equation}
where $c_1>0$ is independent of $h$.
\end{lemma}
\begin{proof}
Let $\calN_{h,k}^{\calE}\subset \calN_{h,k}$ denote the set of DOFs that lie
in the interior of edges, and let $\calN_{h,k}^{\calT}\subset \calN_{h,k}$ denote
the set of DOFs that lie in the interior of a (curved) simplex. The corresponding sets
defined on $\bar \calT_h$ are denoted by $\bar \calN_h^{\bar \calE}$ and $\bar \calN_h^{\bar \calT}$.

For given $q \in Q_h$, let $\bar{q}=q \circ \bPsi \in \bar{Q}_h$ and define $\bv \in \bV_h$ as follows.  

\noindent(i) {\em Vertex DOFs:} We set $\bv(a)=0$ at all vertices $a$ in $\calT_{h,k}$.

\noindent(ii) {\em Edge DOFs:} If $a\in \calN_{h,k}^{\calE}$ lies in the interior of an edge $e$, 
then we set $\bar e = \bPsi^{-1}(e)$, let  $\bt_{\bar e}$ be a unit tangent vector to $\bar e$,
 and set $\bar a = \bPsi^{-1}(a)$ to be the corresponding point on $\bar e$.
Let $\phi_{\bar e}$ be the continuous, piecewise quadratic polynomial (bubble) that takes the value 
$1$ at the edge midpoint of $\bar e$ and vanishes at the vertices and other edges of $\bar \calT_h$.  
We set
\[
\bv_{K_a}(a) =-h_e^2 \phi_{\bar e}(\bar a)({\bt}_{\bar e}\cdot \nabla_{\bar \Gamma_h} \bar{q}(\bar a))  (\calP_{\bPsi_{\bar K_{ a}}} {\bt}_{\bar e})(a)\qquad a\in \calN_{h,k}^{\calE},
\]
where $\bar K_{a} = \bPsi^{-1}(K_{a})$, and we note $ \bt_{\bar e} \cdot \nabla_{\bar \Gamma_h} \bar{q}(\bar a)$ is single-valued 
because $\bar q$ is continuous. 
We see  that $(\calP_{\bPsi_{\bar K_{a}}}{\bt}_{\bar e})(a)$ is tangent to $\Gamma_{h,k}$
by properties of the Piola transform. 
\rev{For a unit co-normal $ \bn_e$ of $ e$, let $\bar \bn_e = \nab \bPsi^\intercal  \bn_e/ |\nab \bPsi^\intercal  \bn_e|$
be the corresponding co-normal of $\bar e$.
We then have}
 $0 = \big(\mu_e(\bar a)\big)^{-1}( \bt_{\bar e} \cdot \bar \bn_e)(\bar a) =(\calP_{\bPsi_{\bar K_{a}}} \bt_{\bar e})(a)\cdot \bn_e(a)$,
 \rev{where we recall $\mu_e$ is given in Lemma \ref{lem:ChangeVarE}.}
 Thus, $(\calP_{\bPsi_{\bar K_{a}}}{\bt}_{\bar e})(a)$ is parallel to the edge unit tangent $\bt_e(a)$.

For both $K \ni a$, we then have using \eqref{eqn:calMDef} and $\bnu_K(a) \cdot \bt_e(a)=0$ that 
\[
\bv_{K}(a) = \calM_a^K \bv_{K_a}(a)=-h_e^2 \phi_{\bar e}(\bar a)  (\bnu_{K_a}(a) \cdot \bnu_{K}(a) ) ({\bt}_{\bar e}\cdot \nabla_{\bar \Gamma_h} \bar{q}(\bar a)) (\calP_{\bPsi_{\bar K_a}} {\bt}_{\bar e})(a),
\quad K\in \calT_a,\ \ a\in \calN_{h,k}^{\calE}.
\]

\noindent(iii) {\em Interior DOFs:} If $a\in \calN_{h,k}^{\calT}$ is a DOF that lies in the interior of $K\in \calT_{h,k}$ (so that $K_a = K$), then we set
\[
\bv_{K_a}(a) =  -\sum_{e\subset \p K} h_e^2 \phi_{\bar e}(\bar a)( \bt_{\bar e}\cdot \nab_{\bar \Gamma_h} \bar q(\bar a))(\calP_{{\bPsi}_{\bar K_a}}  \bt_{\bar e})(a)\qquad a\in \calN_{h,k}^{\calT}.
\]

Letting $\bar{\bv}=\calP_{\bPsi^{-1}}\bv\in \bH({\rm div}_{\bar \Gamma_h};\bar \Gamma_h)$, we  have 
\[
\bar{\bv}_{\bar{K}}(\bar a)=(\calP_{\bPsi_K^{-1}} \bv_{K})(\bar a) =
\left\{
\begin{array}{ll}
-h_e^2\phi_{\bar e}(\bar a)  (\bnu_{K_a}(a) \cdot \bnu_{K}(a) ) ({\bt}_{\bar e}\cdot \nabla_{\bar \Gamma_h} \bar{q}(\bar a)) \calP_{\bPsi_K^{-1}}(a) (\calP_{\bPsi_{\bar K_a}} {\bt}_{\bar e})(a) 
& a\in \calN_{h,k}^{\calE},\\
-\sum_{e\subset \p K} h_e^2 \phi_{\bar e}(\bar a)( \bt_{\bar e}\cdot \nab_{\bar \Gamma_h} \bar q(\bar a)) \bt_{\bar e}& a\in \calN_{h,k}^{\calT},
\end{array}\right.
\]
and $\bar \bv$ vanishes at the vertices of $\bar \Gamma_h$.

Recall that, for $a\in \calN_{h,k}^{\calE}$, $(\calP_{\bPsi_{\bar K_a}} {\bt}_{\bar e})(a) = \calP_{\bPsi_{\bar K_{a}}}(\bar{a}) {\bt}_{\bar e}$ is parallel to $\bt_e(a)$,
and so $\calP_{\bPsi_K^{-1}}(a) \calP_{\bPsi_{K_{a}}}(\bar{a}) {\bt}_{\bar e}$ is parallel to $ \bt_{\bar e}$.
Thus we can write
\begin{equation}\label{eqn:LineA}
(\bnu_{K_{a}}(a)  \cdot \bnu_K(a)) \calP_{\bPsi_K^{-1}}(a) \calP_{\bPsi_{K_{a}}}(\bar{a}) {\bt}_{\bar e}= C_{K, e,a} {\bt}_{\bar e}
\end{equation}
for some $C_{K,e,a}\in \bbR$.
An application of Lemma \ref{lem:xPert} yields
\begin{equation*}
\begin{split}
\left|\calP_{\bPsi_K^{-1}}(a) \calP_{\bPsi_{\bar K_{a}}}(\bar a)  \bt_{\bar e}- \bt_{\bar e}\right|
&\le \left|\calP_{\bPsi_K^{-1}}(a) \calP_{\bPsi_{\bar K_{a}}}(\bar a)  \bt_{\bar e}- \calP_{\bPsi_{\bar K_{a}}}(\bar a)  \bt_{\bar e}\right|\\
&\qquad+\left|\calP_{\bPsi_{\bar K_{a}}}(\bar a)  \bt_{\bar e}- \bt_{\bar e}\right|\\
&\lesssim h | \bt_{\bar e}| = h.
\end{split}
\end{equation*}
This estimate along with $\bnu_K(a)\cdot \bnu_{K_a}(a) = 1+O(h^{2k})$ and \eqref{eqn:LineA}
shows $C_{K,e,a} = 1+\calO(h)$ and
\[
\bar{\bv}_{\bar{K}}(\bar a)=
\left\{
\begin{array}{ll}
-h_e^2 C_{K,e,a}  \phi_{\bar e}(\bar a)  ({\bt}_{\bar e}\cdot \nabla_{\bar \Gamma_h} \bar{q}(\bar a))  {\bt}_{\bar e}(\rev{\bar a}) 
& \rev{\bar a}\in \bar \calN_{h}^{\bar \calE},\\
-\sum_{e\subset \p K} h_e^2 \phi_{\bar e}(\bar a)(\bt_{\bar e}\cdot \nab_{\bar \Gamma_h} \bar q(\bar a)) \bt_{\bar e}& \rev{\bar a}\in \bar \calN_{h}^{\bar \calT},
\end{array}\right.
\]
\rev{where we use the fact that $\bt_{\bar e}$ is constant in the second line.}
We see from a simple scaling argument that $\|\bar \bv\|_{H^1_h(\bar \Gamma_h)}\lesssim \|h \nab_{\bar \Gamma_h} \bar q\|_{L_2(\bar \Gamma_h)}$.

Now define $\tilde \bv:\bar \Gamma_h\to \bbR^3$ as
\[
\tilde{\bv} (\bar x)= -\sum_{\bar e \in \bar \calE_h}  h_e^2 \phi_{\bar e}(\bar x)({\bt}_{\bar e} \cdot \nab_{\Gamma_h} \bar{q}(\bar x)) {\bt}_{\bar e}
\]
and note that both $\tilde \bv$ and $\bar \bv$ are  polynomials of degree $r$ on each $\bar K\in \bar \calT_h$ with
\[
(\tilde \bv -\bar{\bv})_{\bar{K}}(\bar a)=
\left\{
\begin{array}{ll}
-h_e^2 (1-C_{K,e,a})  \phi_{\bar e}(\bar a)  ({\bt}_{\bar e}\cdot \nabla_{\bar \Gamma_h} \bar{q}(\bar a))  {\bt}_{\bar e}(a) 
& a\in \bar \calN_{h}^{\bar \calE},\\
0& a\in \bar \calN_{h}^{\bar \calT}.
\end{array}\right.
\]
A scaling argument and the inequality $|1-C_{K,e,a}|\lesssim h$ then yields
\begin{align*}
\|h^{-1}(\tilde \bv-\bar \bv)\|_{L_2(\bar \Gamma_h)}^2 
&\lesssim \sum_{\bar a\in \bar \calN_h^{\bar \calE}}  |(\tilde \bv-\bar \bv)(\bar a)|^2\lesssim 
 \sum_{\bar a\in \bar \calN_h} h^6 |\nab_{\bar \Gamma_h} \bar q(\bar a)|^2\lesssim h^2 \|h \nab_{\bar \Gamma_h} \bar q\|_{L_2(\bar \Gamma_h)}^2.
\end{align*}

Using properties of the Piola transform and integrating by parts, rearranging the sum, and using norm equivalence, we then obtain 
for $h$ sufficiently small
\begin{align*}
\int_{\Gamma_{h,k}} ({\rm div}_{\Gamma_{h,k}} \bv) q & = \int_{\bar \Gamma_h} ({\rm div}_{\bar \Gamma_h} \bar{\bv}) \bar{q}  = -\int_{\bar \Gamma_h} \tilde{\bv} \cdot \nabla_{\bar \Gamma_h} \bar{q}
-\int_{\bar \Gamma_h} (\bar \bv-\tilde \bv)\cdot \nab_{\bar \Gamma_h} \bar q
 \\ & \ge \sum_{\bar{K} \in \bar{\calT}_h} \sum_{\bar e \in \bar \calE_h: \bar{e} \subset \partial \bar{K}} \int_{\bar{K}}  h_e^2 \phi_{\bar e} ({\bt}_{\bar e} \cdot \nab_{\bar \Gamma_h} \bar{q})^2
 - \|h^{-1}(\bar \bv-\tilde \bv)\|_{L_2(\bar\Gamma_h)}  \|h \nab_{\bar \Gamma_h} \bar q\|_{L_2(\bar \Gamma_h)}
\\ & \gtrsim \sum_{e\in \calE_h} h^2_e \sum_{K\in \calT_{h,k}: e\subset \p K} \int_{K} |\nab_{\Gamma_h} \bar{q}\cdot {\bt}_{\bar e}|^2-  h\|h \nab_{\bar \Gamma_h} \bar q\|_{L_2(\bar \Gamma_h)}^2
\gtrsim \|h \nab_{\Gamma_h} \bar{q}\|_{L_2(\bar \Gamma_h)}^2
\\ & \gtrsim \|h \nab_{\Gamma_{h,k}} q\|_{L_2(\Gamma_{h,k})}^2.
\end{align*}
The proof is complete upon using this estimate
and
\[
\|\bv\|_{H^1_h(\Gamma_{h,k})}\lesssim \|\bar \bv\|_{H^1_h(\bar \Gamma_h)}\lesssim \|h \nab_{\bar \Gamma_h} \bar q\|_{L_2(\bar \Gamma_h)}\lesssim \|h \nab_{\Gamma_{h,k}} q\|_{L_2(\Gamma_{h,k})}.
\]
\end{proof}

\begin{theorem}
There holds
\begin{align}\label{eqn:THInfSup}
\sup_{\bv\in \bV_h} \frac{\int_{\Gamma_{h,k}} ({\rm div}_{\Gamma_{h,k}} \bv) q}{\|\bv\|_{H^1_h(\Gamma_{h,k})}}
\ge \beta \|q\|_{L_2(\Gamma_{h,k})}\qquad \forall q\in Q_h.
\end{align}
\end{theorem}
\begin{proof}
For a fixed $q\in Q_h$,
let $\bw\in \bH_T^1(\gamma)$ satisfy \cite{JankuhnEtal18}
\[
{\rm div}_\gamma \bw = (\mu_{h}^{-1} q)^\ell,\qquad \|\bw\|_{H^1(\gamma)}\lesssim \|(\mu_{h}^{-1} q)^\ell\|_{L_2(\gamma)}\lesssim \|q\|_{L_2(\Gamma_{h,k})}.
\]
Let $\bI_{h,1} \bw\in \bV_h$ be the approximation of $\bw$ satisfying \eqref{eqn:IhwProp}.
Then
\begin{align*}
\frac{\int_{\Gamma_h} ({\rm div}_{\Gamma_h} (\bI_{h,1} \bw)) q}{\|\bI_{h,1} \bw\|_{H^1_h(\Gamma_h)}} = 
\frac{\int_{\gamma} ({\rm div}_\gamma \pt{\bI_{h,1} \bw})q^\ell}{\|\bI_{h,1} \bw\|_{H^1_h(\Gamma_h)}}
\gtrsim  \frac{\int_{\gamma} ({\rm div}_\gamma \pt{\bI_{h,1} \bw})q^\ell}{\|\bw\|_{H^1(\gamma)}}.
\end{align*}
We then add and subtract $\bw$, integrate by parts, and use \eqref{eqn:IhwProp}:
\begin{equation}\label{eqn:ASw}
\begin{split}
\frac{\int_{\Gamma_{h,k}} ({\rm div}_{\Gamma_{h,k}} (\bI_{h,1} \bw)) q}{\|\bI_h \bw\|_{H^1_h(\Gamma_{h,k})}} 
&\gtrsim  \frac{\int_{\gamma} ({\rm div}_\gamma \bw)q^\ell}{\|\bw\|_{H^1(\gamma)}}
-\frac{\int_\gamma ({\rm div}_\gamma (\bw-\pt{\bI_{h,1} \bw})) q^\ell}{\|\bw\|_{H^1(\gamma)}}\\
&\geq c_2\|q\|_{L_2(\Gamma_{h,k})}
+\frac{\int_\gamma  (\bw-\pt{\bI_{h,1} \bw}) \cdot \nab_\gamma q^\ell}{\|\bw\|_{H^1(\gamma)}}\\
&{\geq c_2} \|q\|_{L_2(\Gamma_{h,k})} 
- { c_3} \left(\sum_{K\in \calT_{h,k}} h_K^2 \|\nab_\gamma q^\ell\|_{L_2(K^\gamma)}^2\right)^{1/2}.
\end{split}
\end{equation}
for some $c_2,c_3>0$ independent of $h$.
%
Note that $\|\nab_\gamma q^\ell\|_{L^2(K^\gamma )}\approx \|\nab_{\Gamma_h} q\|_{L_2(K)}$ on each $K\in \calT_{h,k}$. 
Consequently,
\begin{align*}
\left(\sum_{K\in \calT_{h,k}} h_K^2 \|\nab_\gamma q^\ell\|_{L_2(K^\gamma)}^2\right)^{1/2}\lesssim \|h \nab_{\Gamma_{h,k}} q\|_{L_2(\Gamma_{h,k})}.
\end{align*}
Thus we conclude from \eqref{eqn:ASw},
\begin{equation}\label{eqn:Verf}
\sup_{\bv\in \bV_h} \frac{\int_{\Gamma_{h,k}} ({\rm div}_{\Gamma_{h,k}} \bv) q}{\|\bv\|_{H^1_h(\Gamma_{h,k})}}
\ge \frac{\int_{\Gamma_{h,k}} ({\rm div}_{\Gamma_{h,k}} (\bI_{h,1} \bw)) q}{\|\bI_{h,1} \bw\|_{H^1_h(\Gamma_{h,k})}} 
\ge  c_2 \|q\|_{L_2(\Gamma_{h,k})} - c_4 \|h \nab_{\Gamma_{h,k}} q\|_{L_2(\Gamma_{h,k})}
\end{equation}
with $c_4>0$ independent of $h$.
Finally, we combine \eqref{eqn:THKMInfSup} 
and \eqref{eqn:Verf} to obtain 
\begin{align*}
c_2  \|q\|_{L_2(\Gamma_{h,k})} 
&\le \sup_{\bv\in \bV_h} \frac{\int_{\Gamma_{h,k}} ({\rm div}_{\Gamma_{h,k}} \bv) q}{\|\bv\|_{H^1_h(\Gamma_{h,k})}}
+c_4 \|h \nab_{\Gamma_{h,k}} q\|_{L_2(\Gamma_{h,k})}\\
&\le \left(\sup_{\bv\in \bV_h} \frac{\int_{\Gamma_{h,k}} ({\rm div}_{\Gamma_{h,k}} \bv) q}{\|\bv\|_{H^1_h(\Gamma_{h,k})}}\right)
+\frac{c_4}{c_1} \left(\sup_{\bv\in \bV_h} \frac{\int_{\Gamma_{h,k}} ({\rm div}_{\Gamma_{h,k}} \bv) q}{\|\bv\|_{H^1_h(\Gamma_{h,k})}}\right)\\
& = (1+ \frac{c_4}{c_1})  \left(\sup_{\bv\in \bV_h} \frac{\int_{\Gamma_{h,k}} ({\rm div}_{\Gamma_{h,k}} \bv) q}{\|\bv\|_{H^1_h(\Gamma_{h,k})}}\right).
\end{align*}
Thus, we obtain  the inf-sup condition \eqref{eqn:THInfSup}
with $\beta = c_2 (1+ \frac{c_4}{c_1})^{-1}$.
\end{proof}

\section{Consistency Estimates}\label{sec-Consistency}
In this section, we derive several
estimates that deal with the lack of $\bH^1$-conformity
of the finite element space $\bV_h$.
These results show that, while functions in this space
are not globally continuous on $\Gamma_{h,k}$
they possess sufficient ``weak continuity'' properties
to ensure the finite element method presented in the subsequent section
converges with optimal order.
To derive such estimates, we make the following assumption
on the location of the edge degrees of freedom.
\begin{assumption}
We assume that the location of the edge DOFs of $\bV_h$
correspond to the $\bPsi$-mapped $(r+1)$-point Gauss-Lobatto rule.
We further assume  the surface $\gamma$ (and hence $\bp$) is smooth with 
\begin{equation}
\label{eqn:GammaPsmooth}
\rev{\|\bp\|_{W^{2r+1}_\infty(U_\delta)}\lesssim 1.}
\end{equation}
\end{assumption}

We define the jump of a given elementwise-smooth (with respect to $\calT_{h,k}^\gamma$) vector-valued function $\bv$ across $e^\gamma$
as
\begin{align*}
\jump{\bv}|_{e^\gamma} = \bv_+\otimes \bn^\gamma_+ + \bv_- \otimes \bn^\gamma_-,
\end{align*}
where $e^\gamma = \p K^\gamma_+\cap \p K^\gamma_-\ (K^\gamma_{\pm}\in \calT_{h,k}^\gamma)$, $\bv_{\pm} = \bv|_{K^\gamma_\pm}$, and $\bnu^\gamma_{\pm}$
is the outward unit co-normal with respect to $\p K_{\pm}^\gamma$.
Likewise, given an elementwise-smooth quantity $v$, let $\{v \}|_e = \frac12(v_++v_-)$ denote the average value of $v$ on the mesh skeleton.

Recall $\bV(K)$ is the space of functions in $\bV_h$ restricted to $K\in \calT_{h,k}$.
For a given $\bv\in \bV_h$, we define
$\bv_I$ such that $\bv_I|_K\in \bV(K)$ for all $K\in \calT_{h,k}$ and 
\begin{equation}\label{eqn:bvcConstruction}
\bv_I|_K (a) = \calP_{\bp|_K^{-1}} \calP_{\bp} \bv_{K_a}(a)\qquad \forall K\in \calT_a,\ \forall a\in \calN_{h,k}.
\end{equation}
In further detail, we have
\begin{align*}
\bv_I|_K(a) 
 & = \frac{(\bnu\cdot \bnu_K)(a)}{(\bnu\cdot \bnu_{K_a})(a)} \left[
 {\bf I}- { \frac{ \bnu(a) \otimes \bnu_K(a)}{\bnu(a) \cdot \bnu_K(a)} } \right]\bPi(a) \bv_{K_a}(a)\qquad \forall K\in \calT_a,\ \forall a\in \calN_{h,k}.
\end{align*}
By co-normal-preserving properties
of the Piola transform, there holds $\bv_I\in \bH({\rm div}_{\Gamma_{h,k}};\Gamma_{h,k})$.
Furthermore, using Lemma \ref{lem:PiolaD}, we have $|(\bv-\bv_I)_K(a)|\lesssim h^{2k} |\bv_{K_a}(a)|$. Because $\bv,\bv_I|_K\in \bV(K)$, 
a scaling estimate yields the following result.
\begin{lemma}\label{lem:NTS1}
For a given $\bv\in \bV_h$, let $\bv_I$ be uniquely determined by \eqref{eqn:bvcConstruction}.
Then there holds
\begin{equation}
\label{eqn:NTS1}
\|\pt \bv-\pt \bv_I\|_{L_2(\gamma)}+h \|\pt \bv - \pt \bv_I\|_{H^1(\gamma)}\lesssim h^{2k} \|\pt \bv\|_{L_2(\gamma)}.
\end{equation}
\end{lemma}
\begin{proof} 
By a change of variables and \eqref{eqn:PiolaBounds}, we have ($m=0,1$) 
\begin{align*}
\|\pt \bv - \pt \bv_I\|_{H^m(K^\gamma)}
&\lesssim \|\bv -\bv_I\|_{H^m(K)} = \|\calP_{\bPsi} (\bar \bv- \bar \bv_I)\|_{H^m(K)}\\
&\lesssim \|\calP_{\bPsi}\|_{W^m_\infty(\bar K)} \|\bar \bv-\bar \bv_I\|_{H^m(\bar K)}\lesssim \|\bar \bv-\bar \bv_I\|_{H^m(\bar K)},
\end{align*}
where $\bar \bv = \bv \circ \bPsi$ and $\bar \bv_I = \bv_I \circ \bPsi$.
Both $\bar \bv$ and $\bar \bv_I$ are polynomials of degree $r$ on $\bar K$, so a simple scaling argument and \eqref{eqn:PiolaBounds} yields
\begin{align*}
\|\bar \bv-\bar \bv_I\|_{H^m(\bar K)}^2
&\lesssim \sum_{\bar a\in \bar \calN_{\bar K}} h^{2-2m} |(\bar \bv-\bar \bv_I)(\bar a)|^2\\
& = \sum_{\bar a\in \bar \calN_{\bar K}} h^{2-2m} |\calP_{\bPsi^{-1}} \calP_{\bPsi} (\bar \bv-\bar \bv_I)(\bar a)|^2\\
&\le \|\calP_{\bPsi^{-1}}\|^2_{L_\infty(K)} \sum_{a\in \calN_K} h^{2-2m} |(\bv-\bv_I)(a)|^2\lesssim   \sum_{a\in \calN_K} h^{2-2m+4k}|\bv_{K_a}(a)|^2.
\end{align*}
A similar scaling argument shows $h^2 \sum_{a\in \calN_K} |\bv_{K_a}(a)|^2\lesssim  \|\bv\|_{L_2(\omega_K)}^2$.
Therefore
\[
\|\pt \bv - \pt \bv_I\|_{H^m(K^\gamma)}\lesssim h^{2k-m} \|\bv\|_{L_2(\omega_K)}\lesssim h^{2k-m} \|\pt{\bv}\|_{L_2(\omega_{K^\gamma})},
\]
and the result \eqref{eqn:NTS1} follows by summing over $K^\gamma\in \calT^\gamma_{h,k}$.
\end{proof}

Note that $\pt \bv_I$ is {not} continuous on $\gamma$, but is single-valued at the 
(projected) nodal points on $\gamma$. We  exploit this fact
to derive estimates of the jump across edges of 
the mesh $\calT_{h,k}^\gamma$ 
using error estimates of Gauss-Lobatto
integration rules.

\begin{lemma}\label{lem:JUMP}
For given $\bv\in \bV_h$, let $\bv_I$ be defined by \eqref{eqn:bvcConstruction}, 
and let ${\bf G}$ be a piecewise smooth matrix-valued function on $\gamma$.
Then there holds on all $e\in \calE_{h,k}$,
\begin{equation}
\label{eqn:JumpEst}
\left|\int_{e^\gamma} \{{\bf G}\} :\jump{\pt \bv_I}\right| \lesssim h^{r+s}  \big(\|{\bf G}\|_{W^{2r}_\infty(K_+^\gamma)}+\|{\bf G}\|_{W^{2r}_\infty(K_-^\gamma)}\big) \big( \| \bv_I \|_{H^s( K_+)}+\|\bv_I\|_{H^s(K_-)}\big)\quad s=0,1,2,\ldots,r,
\end{equation}
where $e = \p K_+\cap \p K_-$ and $e^\gamma = \bp(e)= {\p K^\gamma_+\cap \p K^\gamma_-}$.
\end{lemma}
\begin{proof}
Let $e$ be an edge of $\calT_{h,k}$, and let $\bar e$ be an edge of $\bar \calT_h$
with $e = \bPsi(\bar e)$. 
By a change of variables, we have
\begin{align*}
\int_{e^\gamma}  \{{\bf G}\} :\jump{\pt \bv_I} = \int_e ( \{{\bf G}\} :\jump{\pt \bv_I})\circ \bp |\nab \bp \bt_e| = \int_{\bar e} \big((\{{\bf G}\}  :\jump{\pt \bv_I}) \circ \bp |\nab\bp \bt_e|\big)\circ \bPsi  |\nab \bPsi \bt_{\bar e}|.
\end{align*}
Set $\bar{\pt \bv}_I = \pt{\bv}_I \circ \bp \circ \bPsi$, $\bar {\bf G} = {\bf G}\circ \bp \circ \bPsi$, and $\sigma_{\bar e} =  |\nab\bp \bt_e|\circ \bPsi  |\nab \bPsi \bt_{\bar e}|$ so that
\begin{align*}
\int_{e^\gamma}  \{{\bf G}\} :\jump{\pt \bv_I} = \int_{\bar e}  \{ \bar {\bf G}\}  :\big[{\bar{\pt\bv}_I}\big] \sigma_{\bar e},
\end{align*}
where $\big[{\bar{\pt\bv}_I}\big] = \bar{\pt \bv}_I\big|_{\bar K_+} \otimes (\bn_{K^\gamma_+}\circ \bp\circ \bPsi)+
\bar{\pt \bv}_I\big|_{\bar K_-} \otimes (\bn_{K^\gamma_-}\circ \bp\circ \bPsi)$.

Because  ${\bar{\pt\bv}_I}$ is single valued at the Gauss-Lobatto nodes and $(\bn_{K^\gamma_+}\circ \bp\circ \bPsi) = -(\bn_{K^\gamma_-}\circ \bp\circ \bPsi)$ on $\bar e$,
we conclude  $ \{{\bf G}\}  :\big[{\bar{\pt\bv}_I}\big]$ vanishes at the Gauss-Lobatto points of $\bar e$. By the error estimate of the $(r+1)$-point Gauss-Lobatto rules, we have
\begin{align}\label{eqn:GLEstimate}
\left|\int_{e^\gamma}  \{ {\bf G}\}  :\jump{\pt \bv_I}\right|\lesssim h^{2r+1} \big(\|\bar {\bf G}\|_{W^{2r}_\infty(\bar K_+)}+\|\bar {\bf G}\|_{W^{2r}_\infty(\bar K_-)}\big)\big( \| \bar {\pt\bv}_I\|_{W^{2r}_\infty(\bar K_+)}+
\|\bar{\pt\bv}_I\|_{W^{2r}_\infty(\bar K_-)}\big) \|\sigma_{\bar e}\|_{W^{2r}_{\infty}(\bar e)},
\end{align}
where $\bar K_{\pm}\in \bar \calT_h$ with $\bar e = \p \bar K_+\cap  \p \bar K_-$. 

By \eqref{eqn:bPsiHSBound} and \eqref{eqn:Ptang}, there holds $\| |\nab \bPsi \bt_{\bar e}|\|_{W^{2r}_\infty(\bar e)} \lesssim 1$.
Likewise, since $\bt_e\circ \bPsi = \nab \bPsi \bt_{\bar e}/|\nab \bPsi \bt_{\bar e}|$, we have $\|\bt_e\|_{W^{2r}_\infty(e)}\lesssim 1$
by \eqref{eqn:bPsiHSBound} and \eqref{eqn:Ptang}. It then follows 
from the chain rule, \eqref{eqn:bPsiHSBound}, \eqref{eqn:GammaPsmooth},  \eqref{eqn:Ptang} , and a change of variables (cf.~Lemma \ref{lem:ChangeVarE}), 
that $\||\nab \bp \bt_e|\circ \bPsi\|_{W^{2r}_\infty(\bar e)}\lesssim \||\nab \bp \bt_e|\|_{W^{2r}_\infty(e)}\lesssim 1$.
Consequently, we have
\begin{align}\label{eqn:SigmaBound}
\|\sigma_{\bar e}\|_{W^{2r}_{\infty}(\bar e)}\lesssim 1.
\end{align}
By \eqref{eqn:GammaPsmooth}, \eqref{eqn:bPsiHSBound} and a change of variables, we have 
\begin{equation}
\label{eqn:ptbvcBound}
\begin{split}
\|\bar {\pt \bv}_I \|_{W^{2r}_\infty(\bar K_+)} + \|\bar {\pt\bv}_I \|_{W^{2r}_{\infty}(\bar K_-)}&\lesssim \| {\pt\bv}_I \|_{W^{2r}_{\infty}( K^\gamma_+)}+\|\pt\bv_I\|_{W^{2r}_\infty(K^\gamma_-)},\\
 \|\bar  {\bf G} \|_{W^{2r}_{\infty}(\bar K_+)}+ \|\bar  {\bf G} \|_{W^{2r}_{\infty}(\bar K_-)}&\lesssim  \| {\bf G} \|_{W^{2r}_{\infty}( K^\gamma_+)}+\|{\bf G}\|_{W^{2r}_\infty(K^\gamma_-)} .
 \end{split}
 \end{equation}
  We use \eqref{eqn:GLEstimate}, \eqref{eqn:SigmaBound}, and \eqref{eqn:ptbvcBound} along with \eqref{eqn:ugh} to conclude
\begin{align}\label{eqn:GLEstimate2}
\left|\int_{e^\gamma}  \{{\bf G}\} :\jump{\pt \bv_I}\right|\lesssim h^{2r+1} \big(\| {\bf G} \|_{W^{2r}_{\infty}( K^\gamma_+)} +\|{\bf G}\|_{W^{2r}_\infty(K^\gamma_-)}\big)\big(  \| \bv_I \|_{W^{2r}_{\infty}( K_+)}
+\| \bv_I \|_{W^{2r}_{\infty}( K_-)}\big),
\end{align}
and the desired result now follows from the inverse estimate \eqref{eqn:HOIE}. 
 \end{proof}


\begin{lemma}
\label{lem:edgejumpest}
For given $\bv\in \bV_h$, let $\bv_I$ be defined by \eqref{eqn:bvcConstruction}.  Then for $0 \le s \le r$, 
\begin{align}\label{eqn:NTS3}
\sum_{e\in \calE_{h,k}} h^{-1}  \|\jump{\pt{\bv}_I}\|_{L_2(e^\gamma)}^2 \lesssim h^{2s} \|\pt\bv_I\|_{H_{h}^s(\gamma)}^2.
\end{align}
\end{lemma}
\begin{proof}
We construct $\bv_c \approx \bv_I$ such that $\pt{\bv}_c \in {\bm C}(\gamma)$.
Let $\bar \bw$ be the continuous, {tri}-piecewise polynomial of degree $r$ (w.r.t. $\bar \calT_h$) such that $\bar \bw(\bar a) = \pt \bv_{K_a}(\bp(a))= \pt \bv_I(\bp(a))$ for $a\in \calN_{h,k}$, where $a = \bPsi(\bar a)$.  We then set 
\begin{align}
\label{vcdef}
\pt{\bv}_c = \bPi \bar \bw \circ \bPsi^{-1} \circ \bp^{-1}, 
\end{align}
and $\bar \bv_I = \calP_{\bPsi^{-1}} \bv_I$.
Using norm equivalence, we then have for $m=0,1$ and all $K^\gamma$
\[
\begin{aligned} 
\|\pt \bv_c-\pt \bv_I\|_{H^m(K^\gamma)} 
& = \|\bPi( \bar \bw \circ \bPsi^{-1} \circ \bp^{-1}- \calP_{\bp} \calP_{\bPsi} \bar \bv_I)\|_{H^m(K^\gamma)} 
\\ & \lesssim \|\bar \bw \circ \bPsi^{-1} \circ \bp^{-1}- \calP_{\bp} \calP_{\bPsi} \bar \bv_I\|_{H^m(K^\gamma)} \lesssim \|\bar \bw-(\calP_{\bp} \circ \bPsi) \calP_{\bPsi} \bar \bv_I\|_{H^m(\bar K)}.
\end{aligned} 
\]
Because $\pt \bv_I$ is single-valued at nodes and $\bar\bw$ coincides with $\pt \bv_I$ at nodes, we have that $\bar \bw$ is the elementwise degree-$r$ Lagrange interpolant of  $\pt \bv_I=\calP_{\bp} \calP_{\bPsi} \bar \bv_I$.  
We thus have 
\begin{align}
\label{vcest}
 \begin{aligned} 
h^m  \|\pt \bv_c-\pt \bv_I\|_{H^m(K^\gamma)} &  \lesssim h^{r+1} |(\calP_{\bp} \circ \bPsi) \calP_{\bPsi} \bar \bv_I|_{H^{r+1}(\bar K)}
\\ & \lesssim h^{r+1} \sum_{j=0}^{r} |(\calP_{\bp} \circ \bPsi) \calP_{\bPsi} |_{W_\infty^{r+1-j}(\bar K)} |\bar \bv_I|_{H^j(\bar K)} 
\\ & \lesssim h^{s+1} \|\bar \bv_I\|_{H^s(\bar K)} \lesssim h^{s+1} \|\pt \bv_I\|_{H^s(K^\gamma)}.
\end{aligned} \end{align}
In the last line we have used the inverse estimate $|\bar \bv_I|_{H^j(\bar K)} \lesssim h^{s-j} \|\bar \bv_I\|_{H^s(\bar K)}$, $s \le j \le r$, along with \eqref{eqn:bPsiHSBound}, \eqref{eqn:PiolaBounds}, and \eqref{eqn:PiolaPBounds}.
The desired estimate then follows by noting that $\|\jump{\pt{\bv}_I}\|_{L_2(e^\gamma)}=  \|\jump{\pt{\bv}_I-\pt \bv_c}\|_{L_2(e^\gamma)}$ and then applying a standard scaled trace inequality followed by the above approximation bound.  
\end{proof}

%

\begin{lemma}\label{lem:JumpDef}
Let $\bw\in \bH^{t+1}(\gamma)$ with $1 \le t \le r-1$.
For given $\bv\in \bV_h$, let $\bv_I$ be defined by \eqref{eqn:bvcConstruction}.
Then there holds
\begin{align}\label{eqn:JumpDef}
 \sum_{e\in \calE_{h,k}} \int_{e^\gamma} {\rm Def}_\gamma \bw: \jump{\pt \bv_I}\lesssim h^{s+t} \|\bw \|_{H^{t+1}(\gamma)} \|\pt \bv_I \|_{H^s_h(\gamma)},\qquad s=0,1,\ldots,r.
 \end{align}
 Consequently, by scaling and Lemma \ref{lem:NTS1},
\begin{align}\label{eqn:JumpDef2}
 \sum_{e\in \calE_{h,k}} \int_{e^\gamma} {\rm Def}_\gamma \bw: \jump{\pt \bv}\lesssim h^{s+t} \|\bw \|_{H^{t+1}(\gamma)} \|\pt \bv\|_{H^s_h(\gamma)}+h^{2k-1} \|\bw\|_{H^2(\gamma)} \|\pt \bv\|_{L_2(\gamma)},\quad s=0,1,\ldots,r.
 \end{align}
\end{lemma}
\begin{proof}
Let ${\bf G}_{h} \circ \bp \circ \bPsi$ be the componentwise $L_2$ projection on $\bar\Gamma_h$ of ${\rm Def}_\gamma \bw \circ \bp \circ \bPsi$ onto the (discontinuous) piecewise polynomials of degree $(r-2)$ on $\bar\calT_h$.  Using standard scaled trace and approximation bounds, along with inverse estimates, we have
 \begin{equation}\label{eqn:DefConst}
 \sum_{e\in \calE_{h,k}} h \|{\rm Def}_\gamma \bw -\{{\bf G}_h\} \|_{L_2(e^\gamma)}^2 \lesssim h^{2t} \|{\rm Def}_\gamma \bw\|_{H^{t}(\gamma)}^2\lesssim h^{2t} \|\bw\|_{H^{t+1}(\gamma)}^2,
\end{equation}
and
\begin{equation}
\label{eqn:DefStabG}
\begin{split}
\sum_{K\in \calT_{h,k}} &  \|{\bf G}_h\|^2_{W^{2r}_\infty(K^\gamma)}  \lesssim   \sum_{K\in \calT_{h,k}} h^{-2} \|{\bf G}_h\|^2_{H^{2r}(K^\gamma)}
\lesssim  \sum_{K\in \calT_{h,k}}  h^{-2} \|{\bf G}_h\|^2_{H^{r-2}(K^\gamma)}
\\ & \lesssim  \sum_{K\in \calT_{h,k}}  h^{-2} \|{\bf G}_h\|^2_{H^{r-1}(K^\gamma)} \lesssim h^{-2+2(t-r+1)}\sum_{K\in \calT_{h,k}}   \|{\bf G}_h\|_{H^t(K^\gamma)}^2 
\lesssim  h^{2(t-r)} \|\bw\|_{H^{t+1}(\gamma)}^2.
\end{split}
\end{equation}

We write
\begin{equation}
\label{eqn:ASD}
\begin{split}
 \sum_{e\in \calE_{h,k}} \int_{e^\gamma} {\rm Def}_\gamma \bw: \jump{\pt \bv_I}
 & = 
  \sum_{e\in \calE_{h,k}} \int_{e^\gamma} \big({\rm Def}_\gamma \bw-\{{\bf G}_h\} \big): \jump{\pt \bv_I}
  +   \sum_{e\in \calE_{h,k}} \int_{e^\gamma} \{{\bf G}_h\} :\jump{\pt \bv_I}.
  \end{split}
  \end{equation}
  
  By \eqref{eqn:NTS3} and \eqref{eqn:DefConst},  there holds
\begin{align}\label{eqn:FirstTerm}
\sum_{e\in \calE_{h,k}} \int_{e^\gamma} \big({\rm Def}_\gamma \bw-\{{\bf G}_h\} \big): \jump{\pt \bv_I}\lesssim h^{s+t} \|\bw\|_{H^{t+1}(\gamma)} \|\pt\bv_I\|_{H^s(\gamma)},
\end{align}
and by \eqref{eqn:JumpEst} and \eqref{eqn:DefStabG}, we have
\begin{align}\label{eqn:SecondTerm}
  \sum_{e\in \calE_{h,k}} \int_{e^\gamma} \{{\bf G}_h\} :\jump{\pt \bv_I}\lesssim h^{s+t} \|\bw\|_{H^{t+1}(\gamma)}\|\pt\bv_I\|_{H^s(\gamma)}.
  \end{align}
Combining \eqref{eqn:ASD} and \eqref{eqn:FirstTerm}--\eqref{eqn:SecondTerm} yields the desired result.
\end{proof}

The following discrete Korn inequalities are easy consequences of several tools developed above.
\begin{lemma}
Given $\bv \in \bV_h$, there holds
\begin{align}
\label{korn:contsurf}
\|\pt{\bv}\|_{H_h^1(\gamma)} \lesssim \|\pt{\bv} \|_{L_2(\gamma)} + \|{\rm Def}_{\gamma, h} \pt{\bv} \|_{L_2(\gamma)}
\end{align}
and
\begin{align}
\label{korn:discsurf}
\|\bv\|_{H_h^1(\Gamma_{h,k})} \lesssim \|\bv \|_{L_2(\Gamma_{h,k})} + \|{\rm Def}_{\Gamma_{h,k},h}\bv\|_{L_2(\Gamma_{h,k})}.
\end{align}
\end{lemma}
\begin{proof}
Let $\pt{\bv}_c\in {\bm C} (\gamma) \cap {\bm H}_T^1(\gamma)$ be as defined in \eqref{vcdef}.  Combining \eqref{eqn:NTS1} and \eqref{vcest} with $m=1$ and $s=0$ yields
$$\|\pt{\bv}-\pt{\bv}_c\|_{H_h^1(\gamma)} \lesssim \|\pt{\bv}\|_{L_2(\gamma)}.$$
With this relationship and \eqref{eqn:DefChain} in hand, the proofs of \eqref{korn:contsurf} and \eqref{korn:discsurf} are completed essentially as in \cite[Lemma 3.5]{DemlowNeilan23} and \cite[Lemma 3.6]{DemlowNeilan23}, respectively. 
\end{proof}

 \rev{We end this section by remarking on differences in our approach to analyzing nonconformity errors here as compared with the predecessor work \cite{DemlowNeilan23}.  There estimates for the nonconformity errors were proved by direct comparison in norm with a fully conforming relative (cf. Lemma 3.4 of that work).  As we see below, here we instead analyze nonconformity errors by using the results directly above to bound edge jump terms resulting from an integration by parts.  Also, while we do employ a fully conforming relative here in the proof of Lemma \ref{lem:edgejumpest}, our main tool in this regard is the ``near-conforming'' relative $\bv_I$ that possesses full continuity at mesh nodes but not elsewhere.  This nodal continuity is sufficient to exploit properties of Gauss-Lobatto quadrature rules and thus obtain optimal-order convergence.  }

\section{Finite element method and convergence analysis}\label{sec-FEM}
\rev{Recall that the discrete velocity space is not $\bH^1$-conforming,
and therefore the method must be based on a piecewise variational formulation.}
We define the bilinear forms
\begin{align*}
a(\bv,\bw) 
&= \int_{\gamma} \left(({\rm Def}_{\gamma,h}\bv):
 ({\rm Def}_{\gamma,h}\bw)+ \bv\cdot \bw\right),\\
 b(\bv,q)
 & = -\int_{\gamma} ({\rm div}_\gamma \bv)q,
 \end{align*}
 where ${\rm Def}_{\gamma,h}\bv|_{K^\gamma} = {\rm Def}_{\gamma}\bv|_{K^\gamma}\ \forall K^\gamma\in \calT_{h,k}^\gamma$
 is the piecewise defined deformation operator.
The variational formulation of the surface Stokes problem then reads:
 find $(\bu,p)\in \bH^1_T(\gamma)\times \mathring{L}_2(\gamma)$ such that
\begin{equation}
\label{eqn:StokesVariational}
 \begin{aligned}
 a(\bu,\bv) + b(\bv,p) & = \int_\gamma {\bm f}\cdot \bv\qquad &&\forall \bv\in \bH^1_T(\gamma),\\
 b(\bu,q)& = 0\qquad &&\forall q\in \mathring{L}_2(\gamma).
 \end{aligned}
 \end{equation}
The relevant continuous inf-sup condition for this problem is found in \cite{JankuhnEtal18}; well-posedness follows.  Here the mass term is added to the standard Stokes equations in order to avoid possible degeneracies due to the presence of Killing fields \cite{SurfaceStokes1}.  Finally, $H^2$ regularity results for surface Stokes equations posed are contained in \cite{ORZ21}, while higher-order regularity is proved in \cite{BNSPP}.  Below we assume these results without direct citation.
 
 The analogous bilinear forms defined on $\Gamma_{h,k}$ are given by
\begin{align*}
a_h(\bv,\bw) 
&= 
\int_{\Gamma_{h,k}}\left(({\rm Def}_{\Gamma_{h,k},h}\bv):
 ({\rm Def}_{\Gamma_{h,k},h}\bw)+ \bv\cdot \bw\right),\\
 b_h(\bv,q)
 & = -\int_{\Gamma_{h,k}} ({\rm div}_{\Gamma_{h,k}} \bv)q,
 \end{align*}
 where $ {\rm Def}_{\Gamma_{h,k},h}\bv|_K =  {\rm Def}_{\Gamma_{h,k}} \bv|_K\ \forall K\in \calT_{h,k}$.
 The finite element method seeks $(\bu_h,p_h)\in \bV_h \times Q_h$ such that 
 \begin{equation}
\label{eqn:FEM}
 \begin{aligned}
 a_h(\bu_h,\bv) + b_h(\bv,p_h) & = \int_{\Gamma_{h,k}} {\bm f}_h \cdot \bv\qquad &&\forall \bv\in \bV_h,\\
 b_h(\bu_h,q)& = 0\qquad &&\forall q\in Q_h.
 \end{aligned}
 \end{equation}
 The existence and uniqueness of a solution to \eqref{eqn:FEM} follows
from the inf-sup condition \eqref{eqn:THInfSup} and the discrete Korn inequality \eqref{korn:discsurf}.
Furthermore, there holds $\|\bu_h\|_{H^1_h(\Gamma_{h,k})}+\|p_h\|_{L_2(\Gamma_{h,k})}\lesssim \|{\bm f}_h\|_{L_2(\Gamma_{h,k})}$.

 \subsection{Geometric consistency estimates}

 \begin{lemma}
 \label{lem:consist1}
 Define the bilinear form
 $G_h:\bH^1_h(\Gamma_{h,k})\times \bH_h^1(\Gamma_{h,k}) \rightarrow \mathbb{R}$ by
 \begin{align}
 \label{eqn:GhDef}
 G_h(\bw,\bv) = a(\pt{\bw},\pt{\bv}) - a_h(\bw,\bv).
 \end{align}
 There holds
 \begin{align}\label{eqn:GhEstimate}
 |G_h(\bv,\bw)|\lesssim h^k \|\pt{\bv}\|_{H^1_h(\gamma)} \|\pt{\bw}\|_{H^1_h(\gamma)}.
 \end{align}
 \end{lemma}
 \begin{proof}
The arguments in \cite[(4.8)]{DemlowNeilan23} and surrounding may be easily adapted to the current case of higher-order surface approximations to yield
\begin{align}
\label{l2geo}
\int_{\Gamma_{h,k}} \bw\cdot \bv -\int_\gamma \pt{\bw} \cdot \pt{\bv} \lesssim h^{k+1}\|\bw\|_{L_2(\Gamma_{h,k})} \|\bv \|_{L_2(\Gamma_{h,k})}.
\end{align} 
 The desired result then directly follows from \eqref{l2geo}, \eqref{eqn:DefChain}, and the arguments in \cite[Lemma 4.2]{DemlowNeilan23}.   Therefore its proof is omitted.
 \end{proof}

 The following lemma improves the geometric consistency error, provided
 the inputted functions are sufficiently regular. 
 Its proof is found in Appendix \ref{app:GhBoundProof}.
 \begin{lemma}\label{lem:GhIBound}
 There holds for all $\bv,\bw\in \bH^2_T(\gamma)$
 \begin{equation}
 |G_h(\ipt{\bv},\ipt{\bw})|\lesssim h^{k+1} \|\bv\|_{H^2(\gamma)} \|\bw\|_{H^2(\gamma)}.
 \end{equation}
 \end{lemma}
 
\subsection{Convergence Analysis}
We start the convergence analysis by introducing
the discretely divergence-free subspace of $\bV_h$:
\[
\bX_h = \{\bv\in \bV_h:\ b_h(\bv,q)=0\ \forall q\in Q_h\}.
\]
We also set
\[
{\bm F}_h = \big[{\bf I}-d {\bf H}]^{-1} \left[ {\bf I} - \frac{\bnu_h\otimes \bnu}{\bnu_h\cdot \bnu}\right] \bPi_h {\bm f}_h,
\]
so that (cf.~\cite{DemlowNeilan23})
\begin{equation}\label{eqn:CapF}
\int_{\Gamma_{h,k}} {\bm f}_h \cdot \bv = \int_{\Gamma_{h,k}} (\bPi_h {\bm f}_h) \cdot \bv = \int_{\gamma} {\bm F}^\ell_h \cdot \pt{\bv}\qquad \forall \bv\in \bL_2(\Gamma_{h,k}),\ \bv\cdot \bnu_h = 0.
\end{equation}
\begin{remark}[Choice of ${\bm f}_h$]
Using $[{\bf I}-d {\bf H}]^{-1} \bnu = \bnu$ and $\bnu \otimes \bnu \frac{\bnu_h \otimes \bnu}{\bnu\cdot \bnu_h} = \bnu \otimes \bnu$, we have
\begin{align*}
{\bm F}_h 
&= \left(\big[{\bf I}-d {\bf H}\big]^{-1} \left[ {\bf I} - \frac{\bPi \bnu_h\otimes  \bnu}{\bnu\cdot \bnu_h}\right] {-} \bnu\otimes \bnu\right)\bPi_h{\bm f}_h\\
&= \big[{\bf I}-d {\bf H}\big]^{-1} \left[ {\bf I} - \frac{\bPi \bnu_h\otimes  \bPi_h \bnu}{\bnu\cdot  \bnu_h}\right] {\bm f}_h - ((\bPi_h {\bm f}_h)\cdot \bnu)\bnu
\end{align*}
Consequently, using {$\pt{\bv} \cdot \bnu=0$,} $|\bPi \bnu_h|\lesssim h^k$, $|\bPi_h \bnu|\lesssim h^k$, and $|d|\lesssim h^{k+1}$, we conclude
\begin{align}\label{eqn:FsDiff}
\int_\gamma ({\bm f}-{\bm F}_h^\ell)\cdot \pt{\bv}\lesssim \left(\|{\bm f}-\bPi {\bm f}_h^\ell\|_{L_2(\gamma)} + h^{k+1} \|{\bm f}_h^\ell\|_{L_2(\gamma)}\right)\|\pt{\bv}\|_{L_2(\gamma)}\qquad \forall \bv\in \bV_h.
\end{align}
Choosing for example either ${\bm f}_h = {\bm f}^\ell$ or ${\bm f}_h = \ipt{\bm f}$, we obtain $\|{\bm f}-\bPi {\bm f}_h^\ell\|_{L_2(\gamma)} \lesssim h^{k+1} \|{\bm f}\|_{L_2(\gamma)}$.  Thus
\begin{align}
\label{fhest}
\int_\gamma ({\bm f}-{\bm F}_h^\ell)\cdot \pt{\bv} \lesssim h^{k+1} \|{\bm f}\|_{L_2(\Gamma)} \hbox{ and } \|{\bm f}_h\|_{L_2(\Gamma_{h,k})} \lesssim \|{\bm f}\|_{L_2(\gamma)}.
\end{align}
It is possible to obtain optimal-order convergence under more general assumptions on the choice of ${\bm f}_h$, but we assume \eqref{fhest} below in order to simplify our presentation.

\end{remark}

\begin{theorem}\label{thm:EnergyConv}
Let $(\bu_h,p)\in \bV_h \times Q_h$ satisfy the finite element
method \eqref{eqn:FEM} and assume the exact solution satisfies $(\bu,p)\in \bH^{r+1}(\gamma)\times H^{r}(\gamma)$.  Assume also that ${\bm f}_h$ is chosen so that \eqref{fhest} is satisfied. 
Then there holds
\begin{align}\label{eqn:EnergyError}
\|\bu-\pt{\bu}_h\|_{H^1_h(\gamma)}  +\|p-p^\ell\|_{L_2(\gamma)}  &\lesssim h^{r} (\|\bu\|_{H^{r+1}(\gamma)}+\|p\|_{H^r(\gamma)}) + h^k \|{\bm f}\|_{L_2(\gamma)}. 
\end{align}
\end{theorem}
\begin{proof}
Given $\bv \in \bX_h$, we add and subtract terms and use \eqref{eqn:GhEstimate} to obtain
\begin{equation}
\label{eqn:ConstStart}
\begin{split}
a(\bu-\pt{\bu}_h,\pt{\bv}) 
&= \Big[a(\bu,\pt{\bv}) - \int_\gamma {\bm f}\cdot \pt{\bv}\Big] + \int_\gamma ({\bm f}-{\bm F}^\ell_h)\cdot \pt{\bv} + G_h(\bu_h,\bv)\\
&\lesssim \Big[a(\bu,\pt{\bv}) - \int_\gamma {\bm f}\cdot \pt{\bv}\Big] +\big(\|{\bm f}-{\bm F}^\ell_h\|_{L_2(\gamma)} +h^k \|\pt{\bu}_h\|_{H^1_h(\gamma)}\big)\|\pt{\bv}\|_{H^1_h(\gamma)}.
%
\end{split}
\end{equation}
%
%

Using integration by parts, we find
\begin{equation}\label{eqn:aIBP}
\begin{split}
a(\bu,\pt{\bv}) 
&= \sum_{K\in \calT_{h,k}} \int_{K^\gamma} {\rm Def}_{\gamma} \bu:{\rm Def}_\gamma \pt{\bv} + \int_\gamma \bu\cdot \pt{\bv}\\
&= -\int_{\gamma} {\rm div}_\gamma ({\rm Def}_{\gamma} \bu ) \cdot \pt{\bv}
+  \sum_{K\in \calT_{h,k}} \int_{\p K^\gamma} ({\rm Def}_{\gamma} \bu) \bn_{\p K^\gamma} \cdot {\pt{\bv}} + 
\int_\gamma \bu\cdot \pt{\bv}\\
%
%
& = \int_{\gamma} {\bm f}\cdot \pt{\bv} - b(\pt{\bv},p)
+  \sum_{e\in \calE_{h,k}} \int_{e^\gamma} ({\rm Def}_{\gamma} \bu) : \jump{\pt{\bv}}.
\end{split}
\end{equation}

Therefore by the divergence-preserving properties of the Piola transform {and since $\bv$ is discretely divergence-free},
{
\begin{align*}
a(\bu,\pt{\bv}) & = \int_{\gamma} {\bm f}\cdot \pt{\bv} - b(\pt{\bv},p-q^\ell)
+  \sum_{e\in \calE_{h,k}} \int_{e^\gamma} ({\rm Def}_{\gamma} \bu) : \jump{\pt{\bv}}\quad \forall \bv\in \bX_h,\ \forall q\in Q_h.
\end{align*}
}
We then find using Lemma \ref{lem:JumpDef} with $t=r-1$ and $s=1$ that 
\begin{align*}
a(\bu,\pt{\bv}) - \int_\gamma {\bm f}\cdot \pt{\bv}
&\lesssim \|p-q^\ell\|_{L_2(\gamma)} \|\pt\bv\|_{H^1_h(\gamma)} + h^{2k-1} \|\bu\|_{H^2(\gamma)} \|\pt{\bv}\|_{L_2(\gamma)} + h^{r} \|\bu\|_{H^{r}(\gamma)} \|\pt \bv\|_{H^1_h(\gamma)}\\
&\lesssim \big(\|p-q^\ell\|_{L_2(\gamma)}+ h^{2k-1} \|\bu\|_{H^2(\gamma)}  + h^{r} \|\bu\|_{H^{r}(\gamma)}\big) \|\pt\bv\|_{H^1_h(\gamma)} .
\end{align*}
We apply this estimate to \eqref{eqn:ConstStart} along with   
$\|\pt{\bu}_h\|_{H^1_h(\gamma)}\lesssim \|\bu_h\|_{H^1_h(\Gamma_{h,k})}\lesssim \|{\bm f}_h\|_{L_2(\Gamma_{h,k})}$, $\|\bu\|_{H^2(\gamma)}\lesssim \|{\bm f}\|_{L_2(\gamma)}$,  \eqref{fhest}, and $k \le 2k-1$
to conclude
\begin{align*}
a(\bu-\pt{\bu}_h,\pt{\bv}) 
&\lesssim 
\Big(\|p-q^\ell\|_{L_2(\gamma)}  + h^{r} \|\bu\|_{H^{r}(\gamma)}+ h^k \|{\bm f}\|_{L_2(\gamma)} \Big) \|\pt\bv\|_{H_h^1(\gamma)}
\end{align*}
for all $q\in Q_h$ and $\bv\in \bX_h$. We then take $q\in Q_h$ to satisfy $\|p-q^\ell\|_{L_2(\gamma)}\lesssim h^{r} \|p\|_{H^{r}(\gamma)}$, apply
Strang's lemma, and use the approximation properties of the finite element space $\bV_h$ (cf.~Lemma \ref{lem:ApproxProps})
 to obtain
\begin{align}\label{eqn:velocityH1Error}
\|\bu-\pt{\bu}_h\|_{H^1_h(\gamma)}\lesssim h^{r} (\|\bu\|_{H^{r+1}(\gamma)}+\|p\|_{H^{r}(\gamma)})
+ h^k \|{\bm f}\|_{L_2(\gamma)}\rev{.} 
\end{align}

Next, we use \eqref{eqn:FEM}, \eqref{eqn:CapF} and \eqref{eqn:GhDef},
along with the divergence-preserving properties of the Piola transform\rev{,} to write
\begin{align*}
b_h(\bv,p_h - q)
& = \int_\gamma {\bm F}_h^\ell \cdot \pt{\bv} - a(\pt{\bu}_h,\pt{\bv}) -b(\pt{\bv},q^\ell) + G_h(\bu_h,\bv)\\
& = 
\left[ \int_\gamma {\bm f}\cdot \pt{\bv} - a(\bu,\pt{\bv})-b(\pt{\bv},p)\right] 
+\int_\gamma ({\bm F}_h^\ell-{\bm f})\cdot \pt{\bv} +a(\bu-\pt{\bu}_h,\pt{\bv}) -b(\pt{\bv},q^\ell-p)+G_h(\bu_h,\bv)
\end{align*}
for all $\bv\in \bV_h$ and $q\in Q_h$. 

We let $q\in Q_h$ satisfy $\|p-q^\ell\|_{L_2(\gamma)}\lesssim h^{r}\|p\|_{H^{r}(\gamma)}$,
and apply 
 \eqref{eqn:GhEstimate} and \eqref{fhest} to obtain
\begin{align*}
&b_h(\bv,p_h - q)\lesssim {\left[ \int_\gamma {\bm f}\cdot \pt{\bv} - a(\bu,\pt{\bv})-b(\pt{\bv},p)\right]}\\
&\qquad+ \left(h^{k+1} \|{\bm f}\|_{L_2(\gamma)}
+\|\bu-\pt{\bu}_h\|_{H^1_h(\gamma)} + h^{r}\|p\|_{H^{r}(\gamma)}+h^{k}\|{\bm f}_h\|_{L_2(\Gamma_{h,k})}\right)\|\pt{\bv}\|_{H^1_h(\gamma)}.
\end{align*}
Using \eqref{eqn:aIBP} and Lemma \ref{lem:JumpDef}, we have
\begin{align*}
{ \int_\gamma {\bm f}\cdot \pt{\bv} - a(\bu,\pt{\bv})-b(\pt{\bv},p) }
&= 
 -\sum_{e\in \calE_{h,k}} \int_{e^\gamma} ({\rm Def}_\gamma \bu): \jump{\pt{\bv}}\\
 &\lesssim h^{r} \|\bu\|_{H^{r}(\gamma)} \|\pt{\bv}\|_{H_h^1(\gamma)}+h^{2k-1}\|\bu\|_{H^2(\gamma)}\|\pt{\bv}\|_{L_2(\gamma)}.
 \end{align*}
{Thus by \eqref{eqn:velocityH1Error}, \eqref{fhest}, and $\|\bu\|_{H^2(\gamma)}\lesssim \|{\bm f}\|_{L_2(\gamma)}$,
\begin{align*}
b_h(\bv,p_h -q)  \lesssim   \left ( h^{r} (\|\bu\|_{H^{r+1}(\gamma)}+\|p\|_{H^{r}(\gamma)})
+ h^k\|{\bm f}\|_{L_2(\gamma)} \right ) \|\pt{\bv}\|_{H_h^1(\gamma)}.
\end{align*}
} 

Applying a triangle inequality, norm equivalence, and the discrete inf-sup condition \eqref{eqn:THKMInfSup} then yields for properly chosen 
$\bv \in \bV_h$ with $\|\bv\|_{H_h^1(\Gamma_{h,k})}=1$ that
$$
\begin{aligned}
\|p-p_h^\ell\|_{L_2(\gamma)} & \lesssim \|p-q^\ell\|_{L_2(\gamma)} + b_h(\bv, p_h-q)
\\ & \lesssim h^{r} (\|\bu\|_{H^{r+1}(\gamma)}+\|p\|_{H^{r}(\gamma)}) + h^{k} \|{\bm f}\|_{L_2(\gamma)}. 
\end{aligned} $$
This estimate combined with \eqref{eqn:velocityH1Error}  yields the desired result \eqref{eqn:EnergyError}.
\end{proof}

{\begin{theorem}\label{thm:L2Estimate}
Suppose the exact solution to the Stokes problem 
satisfies $(\bu,p)\in \bH^{r+1}(\gamma)\times H^r(\gamma)$, and that ${\bm f}_h$ is chosen so that \eqref{fhest} holds.  Then 
\begin{align*}
\|\bu-\pt{\bu}_h\|_{L_2(\gamma)}&\lesssim  h^{r+1}(\|\bu\|_{H^{r+1}(\gamma)}+\|p\|_{H^r(\gamma)})+(h^{2k-1}+h^{k+1})\|{\bm f}\|_{L_2(\gamma)}.
\end{align*}
\end{theorem}
}
\begin{proof}
Let $(\bvarphi,s)\in \bH^1_T(\gamma)\times \mathring{L}_2(\gamma)$ satisfy
\begin{equation}\label{eqn:DualProblem}
\begin{aligned}
a(\bv,\bvarphi)+b(\bv,s) & = \int_\gamma (\bu-\pt \bu_h)\cdot \bv\qquad &&\forall \bv\in \bH^1_T(\gamma),\\
b(\bvarphi,q) & = 0\qquad &&\forall q\in \mathring{L}_2(\gamma),
\end{aligned}
\end{equation}
and let $(\bvarphi_h,s_h)\in \bV_h\times Q_h$ be the corresponding finite element solution:
\begin{equation}\label{eqn:DiscreteDualProblem}
\begin{aligned}
a_h(\bv,\bvarphi_h)+b_h(\bv,s_h) & = \int_{\gamma} (\bu- \pt\bu_h)\cdot \pt\bv\qquad &&\forall \bv\in \bV_h,\\
b_h(\bvarphi_h,q) & = 0\qquad &&\forall q\in Q_h.
\end{aligned}
\end{equation}
Theorem \ref{thm:EnergyConv} and the elliptic regularity 
estimate $\|\bvarphi\|_{H^2(\gamma)}+\|s\|_{H^1(\gamma)}\lesssim \|\bu-\pt{\bu}_h\|_{L_2(\gamma)}$
yield
\begin{align}\label{eqn:DualErrorEstimate}
\|\bvarphi-\pt\bvarphi_h\|_{H^1_h(\gamma)}+ \|s-s^\ell_h\|_{L^2(\gamma)}\lesssim h \|\bu-\pt{\bu}_h\|_{L^2(\gamma)}.
\end{align}
This estimate, along with inverse estimates, approximation properties of $\bV_h$,
  and elliptic regularity, yields
\begin{equation}\label{varphihER}
\|\pt \bvarphi_h\|_{H^2_h(\gamma)} 
 \lesssim \|\bu-\pt{\bu}_h\|_{L_2(\gamma)}.
\end{equation}


%
We take $\bv = \bu$ in \eqref{eqn:DualProblem} and $\bv = \bu_h$ 
in \eqref{eqn:DiscreteDualProblem} and subtract to obtain
\begin{equation}
\label{eqn:DualityStart}
\begin{split}
\|\bu-\pt{\bu}_h\|_{L_2(\gamma)}^2 
&= a(\bu,\bvarphi) - {a_h({\bu}_h,\bvarphi_h)}\\
& = \underbrace{a(\bu-\pt{\bu}_h,\bvarphi)}_{:=I} +\underbrace{\big[a(\pt\bu_h,\bvarphi) - a_h(\bu_h,\bvarphi_h)\big]}_{=:II}.
\end{split}
\end{equation}

To estimate $I$, we  add and subtract terms and apply 
\eqref{eqn:DualErrorEstimate}:
\begin{equation}
\label{eqn:ITermStart}
\begin{split}
I = a(\bu-\pt{\bu}_h,\bvarphi) 
& = a(\bu-\pt{\bu}_h,\bvarphi-\pt{\bvarphi}_h)+a(\bu-\pt{\bu}_h,\pt{\bvarphi}_{h})\\ 
&\lesssim  h\|\bu-\pt{\bu}_h\|_{H^1_h(\gamma)}\|\bu-\pt{\bu}_h\|_{L_2(\gamma)} + a(\bu-\pt{\bu}_h,\pt{\bvarphi}_h).
\end{split}
\end{equation}

We then write using \eqref{eqn:GhDef} and \eqref{eqn:CapF} and the fact $\bvarphi_h$ is discretely divergence-free
\begin{align*}
a(\bu-\pt{\bu}_h,\pt\bvarphi_{h})  
&= a(\bu,\pt{\bvarphi}_{h}) - \int_\gamma {\bm f}\cdot \pt{\bvarphi}_{h} + \int_\gamma {\bm f}\cdot \pt{\bvarphi}_{h} - a(\pt{\bu}_h,\pt{\bvarphi}_{h})\\
&= a(\bu,\pt{\bvarphi}_{h}) - \int_\gamma {\bm f}\cdot \pt{\bvarphi}_{h} + \int_\gamma {\bm f}\cdot \pt{\bvarphi}_{h} - a_h({\bu}_h,{\bvarphi}_{h})
{-} G_h(\bu_h,\bvarphi_{h})\\
%
%
&= a(\bu,\pt{\bvarphi}_{h}) - \int_\gamma {\bm f}\cdot \pt{\bvarphi}_{h} + \int_\gamma ({\bm f}-{\bm F}_h^\ell)\cdot \pt{\bvarphi}_{h}
{-} G_h(\bu_h,\bvarphi_{h}).
\end{align*}
%

 To bound the geometric consistency error, we have by {Lemma \ref{lem:consist1}}, Lemma \ref{lem:GhIBound}, and \eqref{eqn:DualErrorEstimate},
\begin{equation}
\label{eqn:GeoInProof}
\begin{split}
G(\bu_h,\bvarphi_{h}) 
&= G(\bu,\bvarphi)+G(\bu,\bvarphi_{h}-\bvarphi)+G(\bu_h-\bu,\bvarphi_{h})\\
&\lesssim G(\bu,\bvarphi)+h^k\|\bu\|_{H^1(\gamma)} \|\pt\bvarphi_{h}-\bvarphi\|_{H^1(\gamma)}+h^k \|\bu-\pt{\bu}_h\|_{H^1(\gamma)}\|\pt{\bvarphi}_{h}\|_{H^1(\gamma)}\\
&\lesssim (h^{k+1}\|\bu\|_{H^2(\gamma)} +h^{k} \|\bu-\pt{\bu}_h\|_{H^1(\gamma)}) \|\bu-\pt{\bu}_h\|_{L_2(\gamma)}.
\end{split}
\end{equation}

Applying \eqref{fhest} and elliptic regularity we therefore obtain
\begin{equation}
\label{eqn:PartOfI}
\begin{split}
a(\bu-\pt{\bu}_h,\pt\bvarphi_{h})  
&\lesssim \left[a(\bu,\pt\bvarphi_{h}) - \int_\gamma {\bm f}\cdot \pt\bvarphi_{h}\right] \\
&\qquad
+ \left( 
 h^{k+1}\|{\bm f}\|_{L_2(\gamma)}
+h^k \|\bu-\pt{\bu}_h\|_{H^1_h(\gamma)}\right) \|\bu-\pt{\bu}_h\|_{L_2(\gamma)}.
\end{split}
\end{equation}

An application of Lemma \ref{lem:JumpDef} with $t=r-1$ and $s=2$ along with \eqref{varphihER}  yields
\begin{align*}
a(\bu,\pt{\bvarphi}_{h}) - \int_\gamma {\bm f}\cdot \pt{\bvarphi}_{h} 
&= \sum_{e\in \calE_{h,k}} \int_{e^\gamma} {{\rm Def}_\gamma} \bu: \jump{\pt{\bvarphi}_{h}} {-b(\pt{\bvarphi}_{h},p)}
\\ & \lesssim h^{r+1} \|\bu\|_{H^r(\gamma)}\|\pt{\bvarphi}_{h}\|_{H^2_h(\gamma)}+h^{2k-1}\|\bu\|_{H^2(\gamma)}\|\pt{\bvarphi}\|_{L_2(\gamma)}{-b(\pt{\bvarphi}_{h},p)}
\\ & \lesssim (h^{r+1}\|\bu\|_{H^r(\gamma)} +h^{2k-1}\|{\bm f}\|_{L_2(\gamma)})\|\bu-\pt{\bu}_h\|_{L_2(\gamma)}{-b(\pt{\bvarphi}_{h},p)}.
\end{align*}

Now since $\bvarphi_h\in \bH({\rm div}_{\Gamma_h};\Gamma_h)$ is discretely divergence-free and 
$\bvarphi$ is exactly divergence-free,
we have by the divergence-preserving properties of the Piola transform,
\begin{align*}
-b(\pt\bvarphi_h,p) = b(\bvarphi-\pt{\bvarphi}_h,p) \lesssim \|\bvarphi-\pt{\bvarphi}_h\|_{H^1_h(\gamma)}\inf_{q\in Q_h} \|p-q^\ell\|_{L_2(\gamma)}\lesssim h^{r+1}\|p\|_{H^r(\gamma)}\|\bu-\pt{\bu}_h\|_{L_2(\gamma)},
\end{align*}
where we used the approximation properties of $Q_h$ and \eqref{eqn:DualErrorEstimate} in the last step.

%
Thus, we have
\begin{align*}
a(\bu,\pt{\bvarphi}_{h}) - \int_\gamma {\bm f}\cdot \pt{\bvarphi}_{h} \lesssim \left(h^{r+1}(\|\bu\|_{H^r(\gamma)}+\|p\|_{H^{r}(\gamma)})+h^{2k-1}\|{\bm f}\|_{L_2(\gamma)}\right)\|\bu-\pt{\bu}_h\|_{L_2(\gamma)},
\end{align*}
and so by inserting this estimate into \eqref{eqn:PartOfI},  and by using $\|\bu\|_{H^2(\gamma)}+\|p\|_{H^1(\gamma)}\lesssim \|{\bm f}\|_{L_2(\gamma)}$, we have
\begin{align*}
a(\bu-\pt{\bu}_h,\pt{\bvarphi}_{h})
&\lesssim \left( h^{r+1}(\|\bu\|_{H^r(\gamma)} + \|p\|_{H^r(\gamma)})  \right. 
\\
&\qquad\left.+(h^{k+1}+h^{2k-1})\|{\bm f}\|_{L_2(\gamma)}  +h^{k}\|\bu-\pt{\bu}_h\|_{H^1_h(\gamma)}\right) \|\bu-\pt{\bu}_h\|_{L_2(\gamma)}.
\end{align*}
We insert this estimate into \eqref{eqn:ITermStart} and use Theorem \ref{thm:EnergyConv}, thus obtaining
\begin{equation}\label{eqn:IBound}
\begin{split}
I
&\lesssim \left(h \|\bu-\pt{\bu}_h\|_{H^1_h(\gamma)}+ h^{r+1}(\|\bu\|_{H^r(\gamma)} + \|p\|_{H^r(\gamma)}) 
+(h^{k+1}+h^{2k-1})\|{\bm f}\|_{L_2(\gamma)}\right) \|\bu-\pt{\bu}_h\|_{L_2(\gamma)}\\
%
&\lesssim  \left(h^{r+1}(\|\bu\|_{H^{r+1}(\gamma)} + {\|p\|_{H^r(\gamma)}})
+(h^{k+1}+h^{2k-1})\|{\bm f}\|_{L_2(\gamma)}
\right) \|\bu-\pt{\bu}_h\|_{L_2(\gamma)}.
\end{split}
\end{equation}


To bound $II$, we use integration by parts analogous to \eqref{eqn:aIBP} and write
\begin{equation*}
\begin{split}
II  
& = \left[a(\pt\bu_h,\bvarphi) -\int_\gamma \pt\bu_h \cdot  (\bu-\pt{\bu}_h)\right] + \left[\int_\gamma \pt\bu_h \cdot  (\bu-\pt{\bu}_h) -a_h(\bu_h,\bvarphi_h)\right]\\
%
& = \sum_{e\in \calE_{h,k}} \int_{e^\gamma} {\rm Def}_\gamma \bvarphi: \jump{\pt\bu_h} - b(\pt\bu_h,s)+b_h(\bu_h,s_h)\\ 
& = \sum_{e\in \calE_{h,k}} \int_{e^\gamma} {\rm Def}_\gamma \bvarphi: \jump{\pt\bu_h} +b(\bu-\pt\bu_h,s-s_h^\ell). 
\end{split}
\end{equation*}
We then use \eqref{eqn:GeoInProof}, \eqref{eqn:DualErrorEstimate}, and Lemma \ref{lem:JumpDef} with $t=1$ and $s= r$:
\begin{equation}
\label{eqn:IIStart}
\begin{split}
II
&\lesssim  (h^{r+1}\|\pt{\bu}_{h}\|_{H^r_h(\gamma)} + h^{2k-1} \|\pt{\bu}_h\|_{L_2(\gamma)}) \|\bvarphi\|_{H^2(\gamma)}+ \|\bu-\pt{\bu}_h\|_{H^1_h(\gamma)}\|s-s_h^\ell\|_{L_2(\gamma)}\\
%
&\lesssim \left(h^{r+1}\|\pt{\bu}_{h}\|_{H^r_h(\gamma)} +h^{2k-1} \|\pt{\bu}_h\|_{L_2(\gamma)} + h \|\bu-\pt{\bu}_h\|_{H^1_h(\gamma)}\right) \|\bu-\pt{\bu}_h\|_{L_2(\gamma)}
\end{split}
\end{equation}
Using Lemma \ref{lem:ApproxProps}, we let  $\bv\in \bV_h$ satisfy $\|\bu-\pt \bv\|_{H^r_h(\gamma)}+h^{1-r}\|\bu-\pt\bv\|_{H^1_h(\gamma)} \lesssim h \|\bu\|_{H^{r+1}(\gamma)}+h^{2k-r}\|{\bm f}\|_{L_2(\gamma)}$. 
We then add and subtract terms
and apply inverse estimates to obtain
%
\begin{align*}
\|\pt{\bu}_h\|_{H^r_h(\gamma)} 
&\le \|\bu-\pt \bv\|_{H^r_h(\gamma)}+\|\bu_h-\pt \bv\|_{H^r_h(\gamma)}+\|\bu\|_{H^r(\gamma)}\\
&\lesssim  h \|\bu\|_{H^{r+1} (\gamma)} + h^{2k-r}\|{\bm f}\|_{L_2(\gamma)}+h^{1-r} \|\pt \bu_h-\pt \bv\|_{H^1_h(\gamma)}+\|\bu\|_{H^r(\gamma)}\\
&\lesssim  h \|\bu\|_{H^{r+1} (\gamma)} +h^{2k-r}\|{\bm f}\|_{L_2(\gamma)}+h^{1-r} \|\bu - \pt \bu_h\|_{H^1_h(\gamma)}+\|\bu\|_{H^r(\gamma)}.
\end{align*}

We apply this estimate to \eqref{eqn:IIStart} 
and then employ Theorem \ref{thm:EnergyConv} to obtain
\begin{equation}
\label{eqn:IIBoundF}
\begin{split}
II
&\lesssim \left(h^{r+1}\|{\bu}\|_{H^{r+1} (\gamma)} +h^{2k-1} \|\pt{\bu}_h\|_{L_2(\gamma)} +h^{2k+1}\|{\bm f}\|_{L_2(\gamma)}+h \|\bu-\pt{\bu}_h\|_{H^1_h(\gamma)} \right) \|\bu-\pt{\bu}_h\|_{L_2(\gamma)}\\
%
%
&\lesssim \Big(h^{r+1}(\|{\bu}\|_{H^{r+1} (\gamma)} +\|p\|_{H^{r}(\gamma)})+   (h^{k+1}+h^{2k-1})\|{\bm f}\|_{L_2(\gamma)} 
\Big) \|\bu-\pt{\bu}_h\|_{L_2(\gamma)}.
\end{split}
\end{equation}
Finally, we obtain the desired result by combining \eqref{eqn:DualityStart}, \eqref{eqn:IBound} and \eqref{eqn:IIBoundF}.
\end{proof}

\begin{remark}
In the isoparametric case $k=r\ge 2$, Theorems \ref{thm:EnergyConv}--\ref{thm:L2Estimate} yield
optimal-order error estimates in the energy and $L^2$ norms:
\begin{align*}
\|\bu-\pt \bu_h\|_{L_2(\gamma)}+ h\big(\|\bu-\pt \bu_h\|_{H^1_h(\gamma)}+\|p-p_h^\ell\|_{L_2(\gamma)}\big) \lesssim h^{r+1}\big(\|\bu\|_{H^{r+1}(\gamma)}+ \|p\|_{H^r(\gamma)}+ \|{\bm f}\|_{L_2(\gamma)}\big).
\end{align*}
\end{remark}

\begin{remark}[Case of equidistant Lagrange nodes]
In the case of standard Lagrange nodes, where the edge degrees
of freedom are equidistant,  Lemma \ref{lem:JumpDef} is no longer
valid. Instead, using the error estimate of the $(r+1)$-point Newton-Cotes rules
in Lemma \ref{lem:JUMP}, the estimate \eqref{eqn:JumpDef2} would be replaced by
\begin{align*}
 \sum_{e\in \calE_{h,k}} \int_{e^\gamma} {\rm Def}_\gamma \bw: \jump{\pt \bv}\lesssim 
 h^{r'-2r+s+1+t} \|\bw\|_{H^{t+1}(\gamma)} \|\pt \bv\|_{H^s_h(\gamma)} + h^{2k-1} \|\bw\|_{H^2(\gamma)} \|\pt{\bv}\|_{L_2(\gamma)}\quad s=0,1,\ldots r,
 \end{align*}
 where $r'=r+1$ if $r$ is even and $r' = r$ if $r$ is odd. Using this estimate instead of \eqref{eqn:JumpDef2}
 in the proof of Theorem \ref{thm:EnergyConv} (where we take the parameters $t=r-1$ and $s=1$) yields
 a convergence rate in the energy norm of order $r'-r+1$, i.e., first order if $r$ is odd and second order if $r$ is even.
 However, numerical experiments below suggest  these rates are not sharp for $r \ge 3$.
\end{remark}

\section{Numerical Experiments}\label{sec-numerics}

We took $\gamma$ to be the ellipsoid given by $\Psi(x,y,z):=\frac{x^2}{1.1^2} + \frac{y^2}{1.2^2} + \frac{z^2}{1.3^2}=1$.  The test solution is 
{$\bu=\bPi (-z^2, x, y)^\intercal$}.  Note that $\bPi={\bf I}-\bnu\otimes \bnu$ with $\bnu=\frac{\nabla \Psi}{|\nabla \Psi |}$ on $\gamma$, so $\bu$ is componentwise a rational function and not a polynomial.  The pressure is $p=xy^3+z$.  The incompressibility condition ${\rm div}_\gamma \bu=0$ does not hold, so the Stokes system must be solved with nonzero divergence constraint.  We employed a MATLAB code built on top of the iFEM library \cite{Ch09PP}.  The implementation of the degrees of freedom for the velocity space is described in \cite{DemlowNeilan23}.  Briefly, the DOF handler (degree of freedom data structure) must be modified to correctly translate between standard degrees of freedom defined on a Euclidean reference element and the global surface degrees of freedom.  In the case of $\mathbb{P}^2-\mathbb{P}^1$ elements it is possible to employ a standard implementation of the reference Euclidean velocity space, while for higher-order elements a modified implementation incorporating Gauss-Lobatto edge degrees of freedom is needed.

In our tests we employed standard Taylor-Hood quadratic ($\mathbb{P}^2-\mathbb{P}^1$) elements with degrees of freedom at the standard Lagrange nodes, which coincide with the Gauss-Lobatto nodes on element edges.  We also ran experiments with cubic ($\mathbb{P}^3-\mathbb{P}^2$) and quartic ($\mathbb{P}^4-\mathbb{P}^3$) elements, both using standard Lagrange nodes and using Gauss-Lobatto nodes as the edge degrees of freedom.  Interior degrees of freedom were chosen to ensure unisolvence.  Except as otherwise noted, an isoparametric surface approximation $k=r$ was used.  

In Figure \ref{fig_n1}, optimal order convergence in the energy norm $|\ipt \bu-\bu_h|_{H_h^1(\Gamma_{h,k})}+ \|p^e-p_h\|_{L_2(\Gamma_{h,k})}$ and in the velocity $L_2$ norm $\|\ipt \bu -\bu_h\|_{L_2(\Gamma_{h,k})}$ are verified for quadratic, cubic, and quartic Taylor-Hood elements using Gauss-Lobatto edge degrees of freedom.   

\begin{figure}[h]
\begin{center}
\includegraphics[scale=.25]{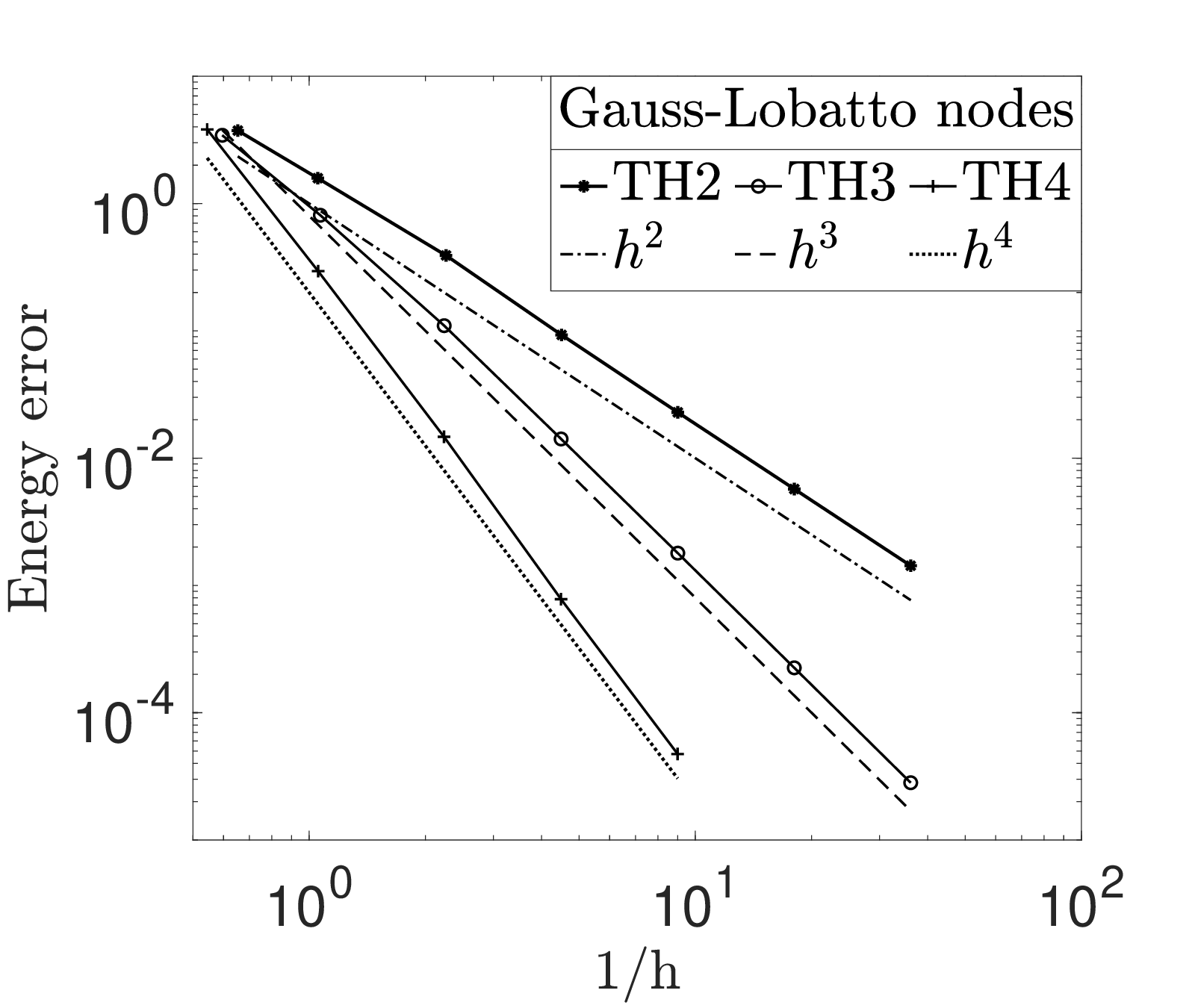}
\includegraphics[scale=.25]{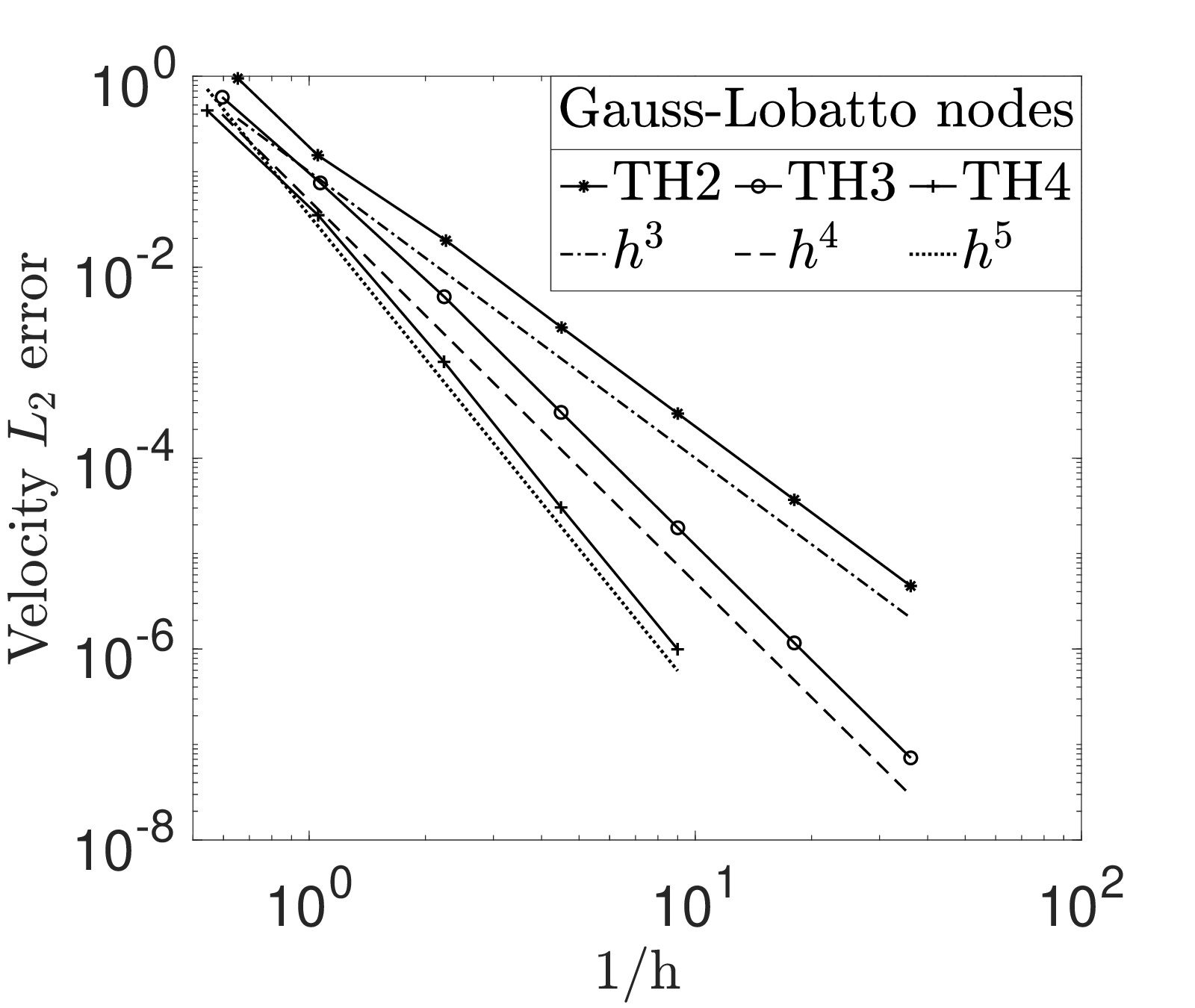}
\end{center}
\caption{Convergence in the energy norm $|\ipt \bu-\bu_h|_{H_h^1(\Gamma_{h,k})}+ \|p^e-p_h\|_{L_2(\Gamma_{h,k})}$ (left) and $L_2$ norm $\|\ipt \bu -\bu_h\|_{L_2(\Gamma_{h,k})}$ (right) for quadratic, cubic, and quartic Taylor-Hood elements using Gauss-Lobatto edge degrees of freedom.}
\label{fig_n1}
\end{figure}

In Figure \ref{fig_n2}, convergence histories in the energy norm 
and in the velocity $L_2$ norm 
are displayed for quadratic, cubic, and quartic Taylor-Hood elements using Lagrange edge degrees of freedom.   As above, convergence for quadratic elements is optimal since Lagrange and Gauss-Lobatto edge degrees of freedom coincide.  For cubic elements the orders of convergence in the energy and $L_2$ norms degenerate clearly to \rev{2 and 3}, respectively, which is suboptimal by one order.  For quartic elements the order of convergence in the energy norm is clearly decreasing to 3, whereas the decrease in the order of convergence in the $L_2$ norm is present but not as pronounced over the range of $h$ values tested.  

\begin{figure}[h]
\begin{center}
\includegraphics[scale=.25]{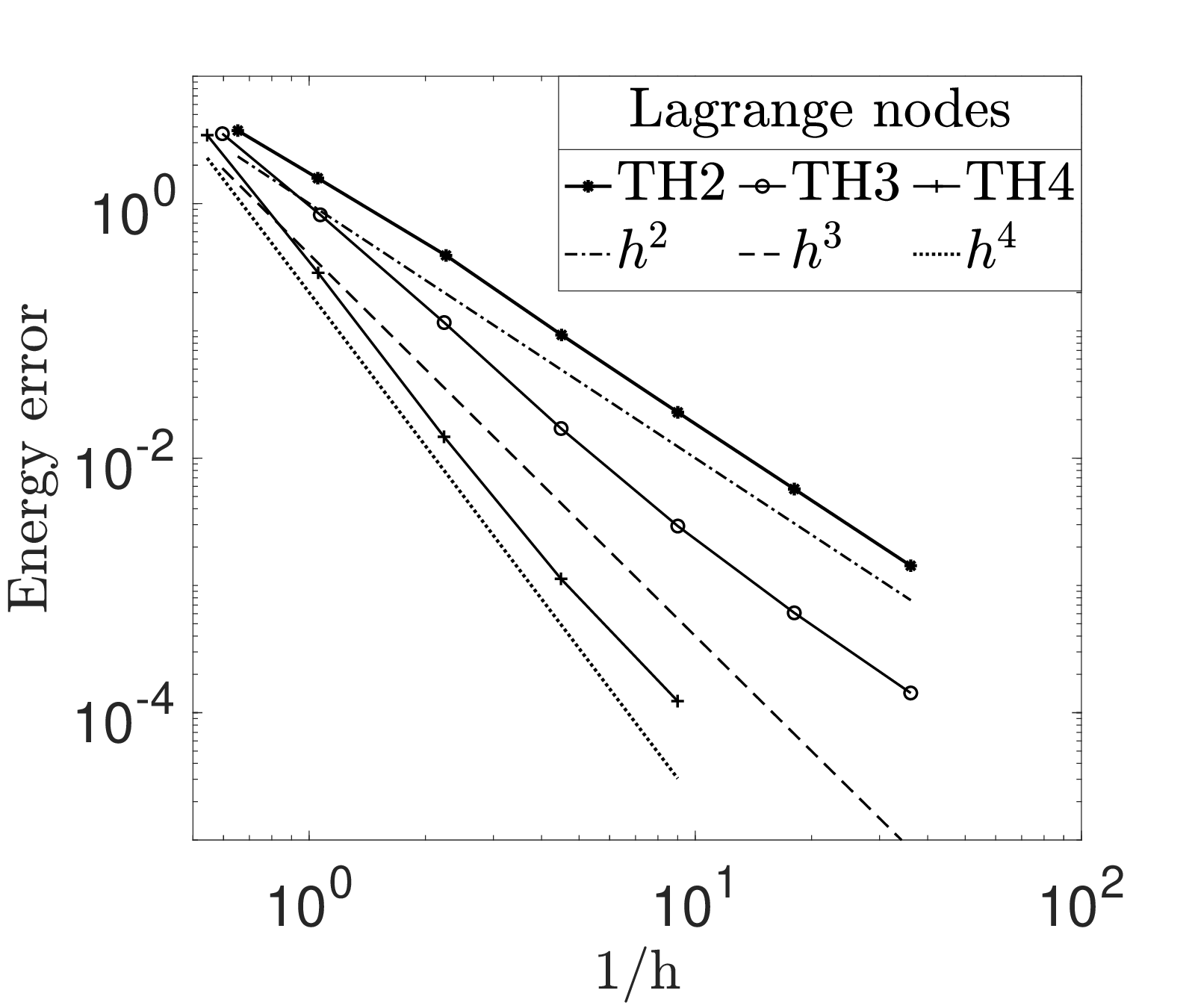}
\includegraphics[scale=.25]{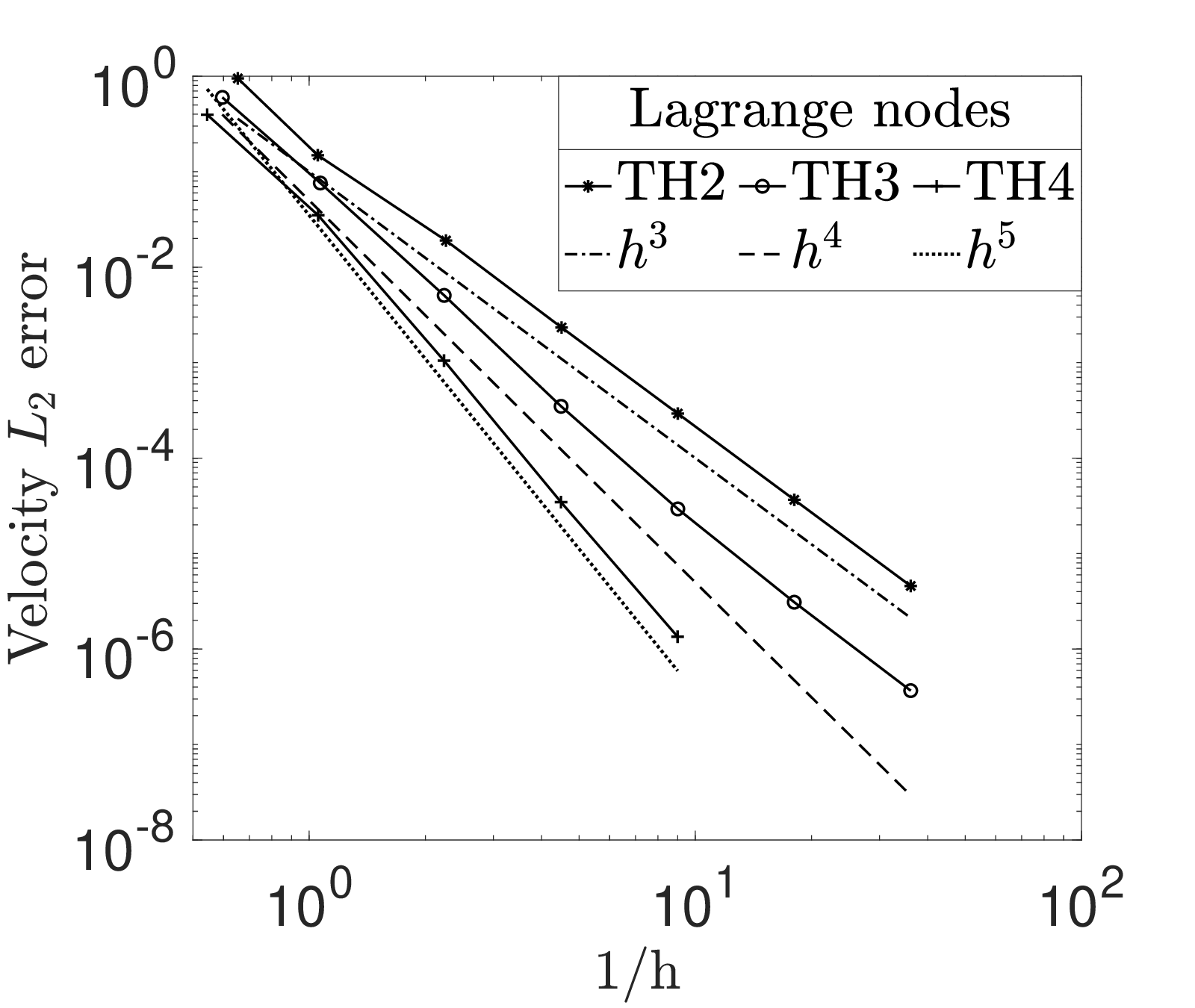}
\end{center}
\caption{\rev{Convergence} in the energy norm $|\ipt \bu-\bu_h|_{H_h^1(\Gamma_{h,k})}+ \|p^e-p_h\|_{L_2(\Gamma_{h,k})}$ (left) and $L_2$ norm $\|\ipt \bu -\bu_h\|_{L_2(\Gamma_{h,k})}$ (right) for quadratic, cubic, and quartic Taylor-Hood elements using Lagrange degrees of freedom.}
\label{fig_n2}
\end{figure}

Finally, we tested a superparametric approximation for cubic elements, that is, we employed $r=3$ and $k=4$.  The results in Figure \ref{fig_n3} indicate almost no change in the error when employing a higher-order surface approximation.  
\begin{figure}[h]
\begin{center}
\includegraphics[scale=.25]{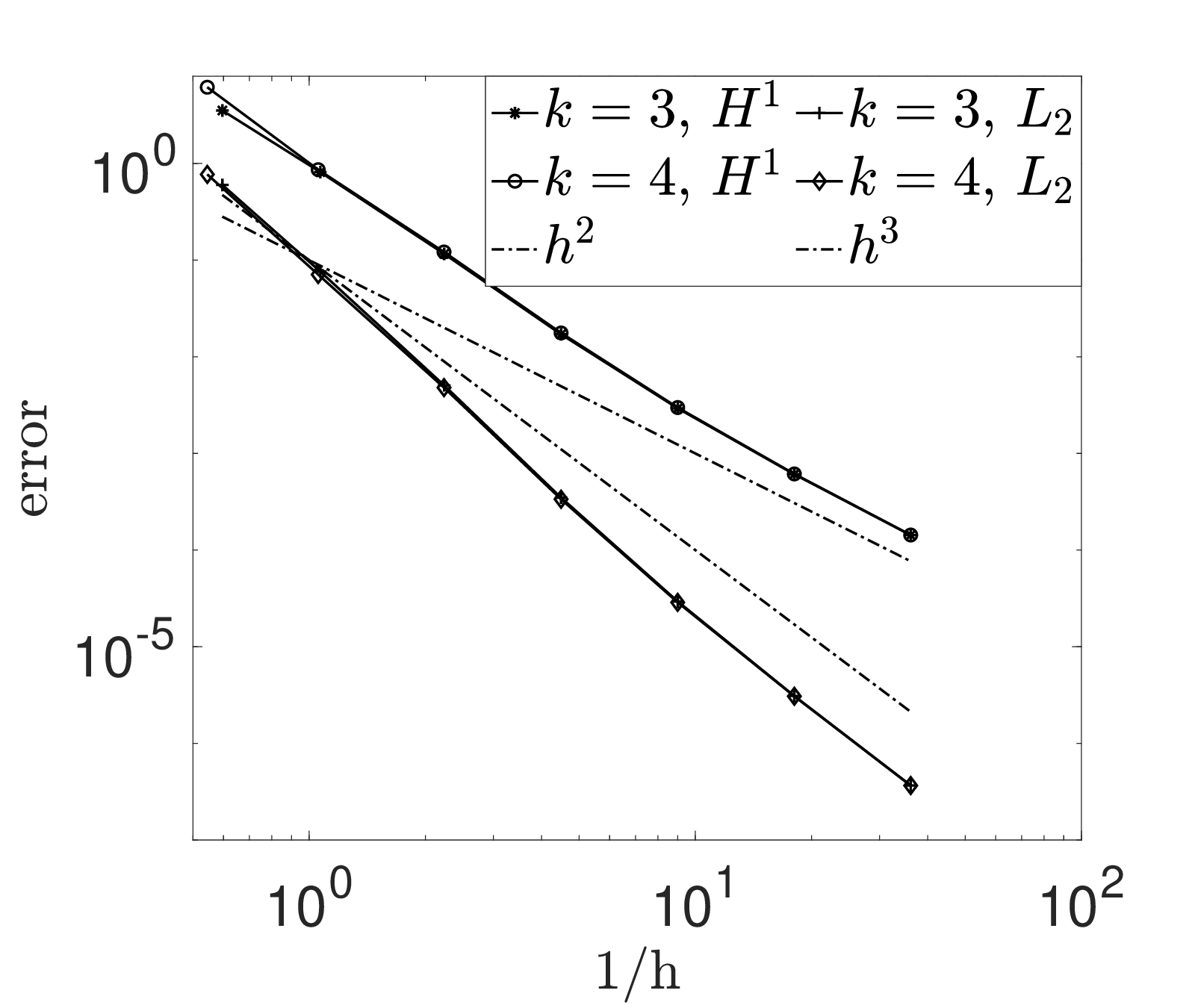}
\end{center}
\caption{\rev{Convergence} in the energy norm $|\ipt \bu-\bu_h|_{H_h^1(\Gamma_{h,k})}+ \|p^e-p_h\|_{L_2(\Gamma_{h,k})}$ and $L_2$ norm $\|\ipt \bu -\bu_h\|_{L_2(\Gamma_{h,k})}$ for cubic Taylor-Hood elements using Lagrange degrees of freedom, and cubic and quartic surface approximations.}
\label{fig_n3}
\end{figure}

Collectively these experiments verify that employing Gauss-Lobatto nodes as edge degrees of freedom is both sufficient and necessary to achieve optimal convergence using our method.  In addition, the nonconformity error is confirmed to be essentially independent from the geometric error since increasing the degree of the surface approximation does not lead to an optimal order of convergence when using Lagrange degrees of freedom.

\section*{Acknowledgements} The authors thank Orsan Kilicer for assistance with numerical computations.

\bibliographystyle{siam}
\bibliography{literatur}

\appendix


\section{Proof of Lemma \ref{lem:ChangeVarE}}
Let $r:[0,1]\to \bar e$ be some parameterization of $\bar e$ such that $\bar \bt_e\circ r(t) =  \frac{r'(t)}{|r'(t)|}$. 
Then $\bPsi\circ r$ is a parameterization of $e$, and we have
\begin{align*}
\int_e q 
&= \int_0^1 q((\bPsi\circ r)(t)) |(\bPsi\circ r)'(t)|\, dt = \int_0^1 (q\circ \bPsi)\circ r(t) |(\nab \bPsi\circ r(t))r'(t)|\, dt\\
&= \int_0^1 (q\circ \bPsi)\circ r(t) |(\nab \bPsi\bar \bt_e)\circ r(t))|\cdot |r'(t)|\, dt\\
&= \int_0^1 (\mu_e (q\circ \bPsi))\circ r(t) |r'(t)|\, dt\\
& = \int_{\bar e} \mu_e q\circ \bPsi.
\end{align*}

\section{Proof of Lemma \ref{lem:xPert}}
First, we write
\begin{align}\label{calPPsi}
\calP_{\bPsi_{\bar{K}}}=\calP_{\ba_k} \calP_{F^{-1}_{\bar K}} = 
\frac{\sqrt{ {\rm det} \nabla F_{\bar K}^\top \nabla F_{\bar K}}}{\sqrt{{\rm det} \nabla \ba_k^\top \nabla \ba_k}} \nabla \ba_k (\nabla F_{\bar K}^\top \nabla F_{\bar K})^{-1} \nabla F_{\bar K}^{\top}.
\end{align}
Therefore, for $\bx$ tangent to $\bar K$, and by writing $\ba_k = F_{\bar K}+\rev{\bell_k}$, we have 
\begin{align*}
\calP_{\bPsi_{\bar{K}}} \bx  
&= \frac{\sqrt{ {\rm det} \nabla F_K^\top \nabla F_{\bar K}}}{\sqrt{{\rm det} \nabla \ba_k^\top \nabla \ba_k}} (\nabla F_{\bar K}+\nab \bell_k) (\nabla F_{\bar K}^\top \nabla F_K)^{-1} \nabla F_{\bar K}^{\top}\bx\\
%
%
&= \bx + \left(\frac{\sqrt{ {\rm det} \nabla F_K^\top \nabla F_{\bar K}}}{\sqrt{{\rm det} \nabla \ba_k^\top \nabla \ba_k}}-1\right)\bx 
+ \frac{\sqrt{ {\rm det} \nabla F_K^\top \nabla F_{\bar K}}}{\sqrt{{\rm det} \nabla \ba_k^\top \nabla \ba_k}} \nab \bell_k (\nab F_{\bar K}^\top \nab F_{\bar K})^{-1} \nab F_{\bar K}^\top \bx.
\end{align*}
Now $(\nab F_{\bar K}^\top \nab F_{\bar K})(\nab \ba_k^\top \nab \ba_k)^{-1} = {\bf I} + {\bf B}(\nab \ba_k^\top \nab \ba_k)^{-1}$ for some ${\bf B}$ satisfying 
$\|{\bf B}\|_{L_\infty(\hat K)}\lesssim \|\rev{\nab \bell_k}\|_{L_\infty(\hat K)}\|\nab \ba_k\|_{L_\infty(\bar K)}+\|\rev{\nab \bell_k}\|_{L_\infty(\hat K)}^2$.
Using the estimates \eqref{eqn:FbarEst}--\eqref{eqn:baPert}, we have
\[
\det\left((\nab F_{\bar K}^\top \nab F_{\bar K})(\nab \ba_k^\top \nab \ba_k)^{-1}\right) = 1+O(h),
\]
and 
\[
\|\nab \bell_k (\nab F_{\bar K}^\top \nab F_{\bar K})^{-1} \nab F_{\bar K}^\top\|_{L_\infty(\hat K)}\lesssim h.
\]
Combining these two estimates then yields
\begin{align*}
\left|\calP_{\bPsi_{\bar{K}}} \bx -\bx\right| \lesssim h |\bx|
\end{align*}
for all vectors $\bx$ tangent to $\bar K$, and so
\begin{align*}
\left|\calP_{\bPsi^{-1}_{{K}}} \bx -\bx\right| \lesssim h |\calP_{\bPsi^{-1}_K} \bx|\lesssim h |\bx|
\end{align*}
for all vectors $\bx$ tangent to $K$.

\section{Proof of Lemma \ref{lem:PiolaOfa}}\label{app:ProofPiolaOfa}


We follow the arguments given in \cite[Lemma 2.3]{NeilanOtus21}.
Set $g = (\det(\nab \ba_K^\intercal \nab \ba_K))^{1/2}:\hat K\to \bbR$.  
Letting $E(m,j) = \{\alpha\in \bbN_0^m:\ |\alpha|=j\text{ and } \sum_{j=1}^m j \alpha_j = m\}$,
there holds the  chain rule \cite{Bernardi89}
\begin{align}
\label{chain_rule}
D^m (f_1\circ f_2) = \sum_{j=1}^m (D^j f_1)\circ f_2 \sum_{\alpha\in E(m,j)} c_\alpha \prod_{i=1}^m (D^i f_2)^{\alpha_i}
\end{align}
for some constants $c_\alpha\in \bbR$.

Now set $f =  \det(\nab \ba_K^\intercal \nab \ba_K)$, so that $g = f^{1/2}=f_1 \circ f_2$ with $f_1(x)=x^{1/2}$ and $f_2=f$.  A simple induction argument and \eqref{eqn:akGood} shows $|f|_{W^m_\infty(\hat K)}\lesssim h^{4+m}$.  Applying \eqref{chain_rule} and recalling \eqref{eqn:akGood} then yields for any multi-index $\beta$ with $|\beta|=m$,
\begin{align*}
\left|\frac{\p^m }{\p \hat x^\beta}g \right|
& =\left|\frac{\p^m }{\p \hat x^\beta}\left(f^{1/2}\right)\right| \lesssim \sum_{j=1}^m |f^{1/2-j}| \sum_{\alpha \in E(m,j)} \prod_{i=1}^m |(D^i f)|^{\alpha_i}
\\ & \lesssim \sum_{j=1}^m h^{4(1/2-j)} \sum_{\alpha \in E(m,j)} \prod_{i=1}^m h^{(4+i)\alpha_i}
 \lesssim \sum_{j=1}^m h^{2-4j} \sum_{\alpha \in E(m,j)} h^{4j+m}
\\ & \lesssim h^{m+2}.
\end{align*}
Applying \eqref{chain_rule} now with $f_1(x) = x^{-1}$ and $f_2=g$ and recalling that $g \approx h^2$ from \eqref{eqn:akGood}, we have for any multi-index $\alpha = (\alpha_1,\alpha_2)$ with $|\alpha|=m$
\begin{align*}
\left|\frac{\p^m }{\p \hat x^\alpha}\left(\frac1{g}\right)\right| 
& \lesssim \sum_{j=1}^m |g|^{-(j+1)} \sum_{\alpha \in E(m,j)} \prod_{i=1}^m |(D^i g)|^{\alpha_i}
\\ & \lesssim \sum_{j=1}^m h^{-2(j+1)} \sum_{\alpha \in E({m,j})} \prod_{i=1}^m h^{(i+2)\alpha_i}
\\ & \lesssim \sum_{j=1}^m h^{-2(j+1)} \cdot h^{m+2j}  \lesssim h^{m-2}.
\end{align*}
We then conclude
\begin{align*}
|g^{-1}|_{W^m_\infty(\hat K)} \lesssim h^{m-2}.
\end{align*}
We then apply the product rule and \eqref{eqn:akGood} to obtain
\begin{align*}
|\calP_{\ba_K}|_{W^m_\infty(\hat K)}
&\lesssim \sum_{\ell=0}^m |\nab \ba_K|_{W^\ell_\infty(\hat K)} |g^{-1}|_{W^{m-\ell}_\infty(\hat K)}
\lesssim h^{m-1}.
\end{align*}


Next, we write $(\calP_{\ba_K})^\dagger = g (\nab \ba_K^\intercal \nab \ba_K)^{-1}\nab \ba_K^\intercal = g^{-1} {\rm adj}(\nab \ba_K^\intercal \nab \ba_K) \nab \ba_K^\intercal$, note that $|{\rm adj}(\nab \ba_K^\intercal \nab \ba_K)|_{W_\infty^q(\hat K)}=|(\nab \ba_K^\intercal \nab \ba_K)|_{W^{q}_\infty(\hat K)}$ because $\nab \ba_K^\intercal \nab \ba_K$ is a $2 \times 2$ matrix, 
and apply similar arguments as above to obtain
\begin{align*}
\left|(\calP_{\ba_K})^\dagger\right|_{W^m_\infty(\hat K)}
&\lesssim \sum_{\ell=0}^m |g^{-1}|_{W^{m-\ell}_\infty(\hat K)} |{\rm adj}(\nab \ba_K^\intercal \nab \ba_K)\nab \ba_K|_{W^\ell_\infty(\hat K)}\\
&\lesssim \sum_{\ell=0}^m h^{m-\ell-2}  \sum_{q=0}^\ell |(\nab \ba_K^\intercal \nab \ba_K)|_{W^{q}_\infty(\hat K)} |\nab \ba_K|_{W^{\ell-q}_\infty(\hat K)}\\
%
&\lesssim \sum_{\ell=0}^m h^{m-\ell-2}  \sum_{q=0}^\ell  h^{\ell+3}
\lesssim h^{m+1}.
\end{align*}

Using \eqref{eqn:FbarEst}, we have $\|(\calP_{F_{\bar K}})^\dagger\|_{L_\infty(\hat K)}\lesssim h$. 
Consequently,  by \eqref{eqn:PiolaOfa} and since $(\calP_{F_{\bar K}})^\dagger$ is constant, we have
\begin{align*}
|\calP_{\bPsi}|_{W^m_\infty(\bar K)}
&\lesssim |F_{\bar K}^{-1}|_{W^1_\infty(\bar K)}^m
 |\calP_{\ba_K}|_{W^m_\infty(\hat K)}\|(\calP_{F_{\bar K}})^\dagger\|_{L_\infty(\hat K)}\lesssim 1,
\end{align*}
for $m\ge 1$.
Likewise, we have $\|\calP_{F_{\bar K}}\|_{L_\infty(\hat K)}\lesssim h^{-1}$ and so by \eqref{eqn:PiolaOfa}, \eqref{eqn:akGood},
and \eqref{chain_rule},
\[
\|\calP_{\bPsi^{-1}}\|_{L_\infty(K)} \lesssim \|\calP_{F_{\bar K}}\|_{L_\infty(\hat K)} \|(\calP_{\ba_K})^\dagger\|_{L_\infty(\hat K)}\lesssim 1,
\]
and
\begin{align*}
|\calP_{\bPsi^{-1}}|_{W^m_\infty(K)}
&\lesssim \|\calP_{F_{\bar K}}\|_{L_\infty(\hat K)}\sum_{j=1}^m |(\calP_{\ba_K})^\dagger|_{W^j_\infty(\hat K)}
\prod_{\alpha\in E_{(m,j)}} |\ba_K^{-1}|_{W^i_\infty(K)}^{\alpha_i}\\
&\lesssim h^{-1}\sum_{j=1}^m h^{-j} |(\calP_{\ba_K})^\dagger|_{W^j_\infty(\hat K)}\lesssim 1.
\end{align*}

To prove the remaining inequalities, we recall \eqref{chain_rule} and use \eqref{eqn:PiolaOfa}, \eqref{eqn:akGood}, and the product rule
to obtain
\begin{align}
\label{eqn:wbound}
\begin{aligned}
|\bw|_{H^m(K)}
& = |(\calP_{\ba_K} \hat \bw) \circ \ba_K^{-1} |_{H^m(K)}\lesssim 
h \sum_{j=1}^m \sum_{\ell=0}^j |\calP_{\ba_K}|_{W_\infty^{j-\ell}(\hat K)} |\hat \bw|_{H^\ell(\hat K)}
\sum_{\alpha\in E(m,j)} \prod_{i=1}^m |\ba_K^{-1}|_{W^i_\infty(K)}^{\alpha_i}\\
&\lesssim h \sum_{j=1}^m \sum_{\ell=0}^j h^{j-\ell-1} |\hat \bw|_{H^\ell(\hat K)}
\sum_{\alpha\in E(m,j)} \prod_{i=1}^m h^{-\alpha_i}\\
&\lesssim h \sum_{j=1}^m \sum_{\ell=0}^j h^{j-\ell-1} |\hat \bw|_{H^\ell(\hat K)} \cdot h^{-j}\\
%
%
&\lesssim \sum_{\ell=0}^m h^{-\ell} |\hat \bw|_{H^\ell(\hat K)}.
\end{aligned}
\end{align}
%
%
Likewise, we use \eqref{eqn:PiolaOfa} and  \eqref{eqn:akGood} to obtain
\begin{align*}
|\hat \bw|_{H^m(\hat K)} 
&=  |(\calP_{\ba^{-1}_K})  {(\bw\circ \ba_K)}|_{H^m(\hat K)}
\lesssim h^{-1} \sum_{r=0}^m |(\calP_{\ba_K})^\dagger|_{W^{m-r}_\infty(\hat K)}  \sum_{j=1}^r |\bw|_{H^{j}(K)} \sum_{\alpha\in E(r,j)} \prod_{i=1}^r |\ba_K|_{W^i_\infty(\hat K)}^{\alpha_i}\\
&\lesssim h^{-1}   \sum_{r=0}^m h^{m-r+1}  \sum_{j=1}^r |\bw|_{H^{j}(K)} \sum_{\alpha\in E(r,j)} \prod_{i=1}^r h^{i\alpha_i}\\
&\lesssim h^{-1}   \sum_{r=0}^m h^{m-r+1}  \sum_{j=1}^r |\bw|_{H^{j}(K)} \cdot h^{r}\lesssim h^m \|\bw\|_{H^m(K)}.
\end{align*}

Finally, the estimate \eqref{eqn:PhiPiolaNormEquiv} directly follows from \eqref{eqn:PiolaBounds} 
and \eqref{eqn:bPsiHSBound}.

\section{Proof of Lemma \ref{lem:DefCons}}
The proof of the estimate \eqref{eqn:DefChain}
mostly follows from the arguments in \cite[Lemma 2.2]{DemlowNeilan23}, so we only
present the main points.

First, by the chain rule (cf.~(A.1) in \cite{DemlowNeilan23})
we have
\begin{align*}
\nab (\bv\circ \bp^{-1}) \bPi 
&= \left(\nab \bv \bPi_h \big[{\bf I} - \frac{\bnu\otimes \bnu_h}{\bnu\cdot \bnu_h}\big]\big[{\bf I}- d {\bf H}\big]^{-1}\right)\circ \bp^{-1}.
\end{align*}
Next, set ${\bf H}_h = \nab \bnu_h$ on $K$ (where $\bnu_h$ is extended to $U_\delta$).
Then since $\bPi_h \bv = \bv$, there holds
\[
\nab \bv \bPi_h = \nab (\bPi_h \bv)\bPi_h = \left(\bPi_h \nab \bv- \bnu_h \otimes ({\bf H}_h \bv)\right)\bPi_h = \nab_{\Gamma_{h,k}} \bv - \bnu_h\otimes ({\bf H}_h \bv).
\]
Setting ${\bf L} = \mu_h^{-1} \big[\bPi - d {\bf H}\big]$, we have
\begin{equation}
\label{eqn:ChainRuleFun}
\begin{split}
\nab_\gamma \pt{\bv} 
& = {\bf L} \left( \big(\nab_{\Gamma_{h,k}}\bv )\big) \left[ {\bf I} - \frac{\bnu\otimes \bnu_h}{\bnu\cdot \bnu_h}\right]
\left[{\bf I}-d {\bf H}\right]^{-1}\right)\circ \bp^{-1}\\
&\qquad  - \left(({\bf L}\bnu_h) \otimes ({\bf H}_h \bv) \left[ {\bf I} - \frac{\bnu\otimes \bnu_h}{\bnu\cdot \bnu_h}\right]
\left[{\bf I}-d {\bf H}\right]^{-1}\right)\circ \bp^{-1}
+ \bPi \nab ({\bf L}\circ \bp^{-1}) \bv\circ \bp^{-1} \bPi.
\end{split}
\end{equation}
Using the estimates $|{\bf L} \bnu_h|\lesssim h^k$
and $|{\bf L}- \bPi_h|\lesssim h^k$
and applying
the same arguments as in \cite[Lemma 2.2]{DemlowNeilan23},
we obtain 
\begin{equation}\label{eqn:DefProof1}
\left|{\rm Def}_{\gamma} \pt{\bv} - ({\rm Def}_{\Gamma_{h,k}} \bv)\circ \bp^{-1}\right|\lesssim 
h^k \big(|(\nab_{\Gamma_{h,k}} \bv)\circ \bp^{-1}| + |\bv\circ \bp^{-1}|\big) + \left|\bPi \nab ({\bf L}\circ \bp^{-1}) \bv\circ \bp^{-1} \bPi\right|.
\end{equation}

Next, we have (cf.~\cite[(A.5)--(A.6)]{DemlowNeilan23})
\begin{align*}
 \left|\bPi \nab ({\bf L}\circ \bp^{-1}) \bv\circ \bp^{-1} \bPi\right|
 &\lesssim \left|\bPi \big(\nab {\bf L} \bv \bPi_h\big)\circ \bp^{-1}\right|,
\end{align*}
and 
\begin{align*}
\bPi \nab {\bf L} \bv \bPi_h
& = -\mu_h^{-1} \bPi\left[ ({\bf L}\bv)\otimes \nab \mu_h + \bnu \otimes ({\bf H} \bv)+({\bf H}\bv)\otimes \bnu + (\bnu\cdot \bv) {\bf H} + d \nab {\bf H}\bv\right]\bPi_h\\
& = -\mu_h^{-1} \left[ ({\bf L}\bv)\otimes (\bPi_h \nab \mu_h) +({\bf H}\bv)\otimes (\bPi_h \bnu) + (\bnu\cdot \bv) {\bf H}\bPi_h  + d \bPi \nab {\bf H}\bv\bPi_h\right]\\
& = -  ({\bf L}\bv)\otimes (\bPi_h \nab \mu_h)+O(h^k |\bv|)
\end{align*}
We write $\mu_h = \bnu\cdot \bnu_h \det({\bf I}-d {\bf H})$
and note that $\nab \det({\bf I}-d {\bf H}) = -\bnu {\rm tr}({\bf H}) +O(h^{k+1})$.
Consequently,
\begin{equation}
\label{eqn:JacobiFun}
\begin{split}
\nab \mu_h 
&= \nab (\bnu\cdot \bnu_h) \det({\bf I}-d {\bf H}) + (\bnu\cdot \bnu_h) \nab \det({\bf I}-d {\bf H})\\
& = \big({\bf H}\bnu_h + {\bf H}_h \bnu\big)\det({\bf I}-d {\bf H}) -  {\rm tr}({\bf H})(\bnu\cdot \bnu_h) \bnu +O(h^{k+1})\\
& = - {\rm tr}({\bf H}) \bnu + O(h^k).
\end{split}
\end{equation}
Thus using $|\bPi_h \bnu|\lesssim h^k$, we obtain
\begin{align*}
\big|\bPi \nab {\bf L} \bv \bPi_h\big|\lesssim h^k |\bv|,
\end{align*}
and so
\begin{align}\label{eqn:DefProof2}
 \left|\bPi \nab ({\bf L}\circ \bp^{-1}) \bv\circ \bp^{-1} \bPi\right|\lesssim h^k|\bv\circ \bp^{-1}|.
\end{align}
The estimate \eqref{eqn:DefChain} is then obtained by combining \eqref{eqn:DefProof1} and \eqref{eqn:DefProof2}.


\section{Proof of Lemma \ref{lem:GhIBound}}\label{app:GhBoundProof}
Recall
\[
G(\bv,\bw) = a(\pt{\bv},\pt{\bw}) - a_h(\bv,\bw).
\]
We make a change of variables, add and subtract terms, and recall \eqref{l2geo} to get
\begin{align*}
G(\ipt{\bw},\ipt{\bv})  
& \lesssim \int_\gamma {\rm Def}_\gamma {\bw}: {\rm Def}_\gamma {\bv} - \int_{\Gamma_{h,k}} {\rm Def}_{\Gamma_{h,k}} \ipt{\bw}: 
 {\rm Def}_{\Gamma_{h,k}} \ipt{\bv} +h^{k+1}\|\bw\|_{L_2(\gamma)} \|\bv \|_{L_2(\gamma)}\\
 &= \int_\gamma {\rm Def}_\gamma {\bw}: {\rm Def}_\gamma {\bv} - \int_{\gamma} \mu_h^{-1} \big({\rm Def}_{\Gamma_{h,k}} \ipt\bw)^\ell: 
\big( {\rm Def}_{\Gamma_{h,k}} \ipt\bv\big)^\ell +h^{k+1}\|\bw\|_{L_2(\gamma)} \|\bv \|_{L_2(\gamma)}\\
%
%
%
 &= \int_\gamma \big({\rm Def}_\gamma {\bw} - ({\rm Def}_{\Gamma_{h,k}} \ipt \bw)^\ell\big): {\rm Def}_\gamma {\bv} 
 +\int_\gamma \big({\rm Def}_\gamma {\bv}-({\rm Def}_{\Gamma_{h,k}} \ipt \bv)^\ell\big):  {\rm Def}_\gamma \bw\\ 
&\qquad 
+\int_\gamma \big(({\rm Def}_{\Gamma_{h,k}} \ipt{\bw})^\ell - {\rm Def}_\gamma \bw\big):\big({\rm Def}_\gamma {\bv}-({\rm Def}_{\Gamma_{h,k}} \ipt \bv)^\ell\big)
\\ 
&\qquad 
+ \int_{\gamma} (1-\mu_h^{-1}) \big({\rm Def}_{\Gamma_{h,k}} \ipt \bw)^\ell: 
\big( {\rm Def}_{\Gamma_{h,k}} \ipt \bv\big)^\ell +h^{k+1}\|\bw\|_{L_2(\gamma)} \|\bv \|_{L_2(\gamma)}.
 \end{align*}
 Therefore by Lemma \ref{lem:DefCons} and $|1-\mu_h^{-1}|\lesssim h^{k+1}$, we have
\begin{equation}
\label{eqn:StartOfLong}
 \begin{split}
\big|G(\ipt{\bw},\ipt{\bv})  \big|
&\lesssim 
  \left|\int_\gamma \big({\rm Def}_\gamma {\bw} - ({\rm Def}_{\Gamma_{h,k}} \ipt \bw)^\ell\big): {\rm Def}_\gamma {\bv} \right|
 +\left|\int_\gamma \big({\rm Def}_\gamma {\bv}-({\rm Def}_{\Gamma_{h,k}} \ipt \bv)^\ell\big):  {\rm Def}_\gamma \bw\right|\\
 &\qquad + h^{k+1}\|\bw\|_{H^1(\gamma)}\|\bv\|_{H^1(\gamma)}.
 \end{split}
 \end{equation}
Thus, we turn our attention to estimating 
 \begin{align}\label{eqn:GEStarting}
 \int_\gamma \big({\rm Def}_\gamma {\bv} - ({\rm Def}_{\Gamma_{h,k}} \ipt \bv)^\ell\big): {\rm Def}_\gamma {\bw} 
 = \int_\gamma \big(\nab_\gamma {\bv}-(\nab_{\Gamma_{h,k}} \ipt \bv)^\ell\big): {\rm Def}_\gamma {\bw}
\end{align}
for $\bw,\bv\in \bH^2_T(\gamma)$; in the last equality we employ ${\bf A}^\intercal:{\bf B}= {\bf A}:{\bf B}$ when ${\bf B}^\intercal={\bf B}$.

Our starting point is \eqref{eqn:ChainRuleFun}, which we rephrase here for convenience:
\begin{equation*}
\begin{split}
\nab_\gamma {\bv} 
& = {\bf L} \left( \big(\nab_{\Gamma_{h,k}}\ipt \bv )\big) \left[ {\bf I} - \frac{\bnu\otimes \bnu_h}{\bnu\cdot \bnu_h}\right]
\left[{\bf I}-d {\bf H}\right]^{-1}\right)^\ell\\
&\qquad  - \left(({\bf L}\bnu_h) \otimes ({\bf H}_h \ipt \bv) \left[ {\bf I} - \frac{\bnu\otimes \bnu_h}{\bnu\cdot \bnu_h}\right]
\left[{\bf I}-d {\bf H}\right]^{-1}\right)^\ell
+ \bPi \nab ({\bf L}^\ell) \ipt \bv^\ell \bPi.
\end{split}
\end{equation*}
We then add and subtract terms to get
\begin{equation}
\label{eqn:JsExpand}
\begin{split}
\nab_\gamma {\bv} 
& = (\nab_{\Gamma_{h,k}} \ipt \bv)^\ell + \left(\big[{\bf L}-{\bf I}\big] \nab_{\Gamma_{h,k}} \ipt \bv\left[{\bf I} - \frac{\bnu\otimes \bnu_h}{\bnu\cdot \bnu_h} \right]\big[{\bf I}-d {\bf H}\big]^{-1}\right)^\ell\\
&\qquad +\left(\nab_{\Gamma_{h,k}} \ipt \bv\left(\left[{\bf I} - \frac{\bnu\otimes \bnu_h}{\bnu\cdot \bnu_h} \right] \big[{\bf I}-d {\bf H}\big]^{-1} - \bPi_h\right)\right)^\ell\\
&\qquad\qquad  - \left(({\bf L}\bnu_h) \otimes ({\bf H}_h \ipt \bv) \left[ {\bf I} - \frac{\bnu\otimes \bnu_h}{\bnu\cdot \bnu_h}\right]
\left[{\bf I}-d {\bf H}\right]^{-1}\right)^\ell+ \bPi \nab {\bf L}^\ell \ipt \bv^\ell\bPi\\
& =:(\nab_{\Gamma_{h,k}} \ipt \bv)^\ell  + J_1+J_2+J_3+J_4.
\end{split}
\end{equation}
We now bound each $J_i$.

We use $|{\bf L}- \bPi|\lesssim h^{k+1}$ and $|d|\lesssim h^{k+1}$ to obtain
\begin{align*}
\int_\gamma J_1: {\rm Def}_\gamma {\bw}
& \lesssim \int_\gamma \left(\big[\bPi - {\bf I}\big] \nab_{\Gamma_{h,k}} \ipt \bv \left[{\bf I} - \frac{\bnu\otimes \bnu_h}{\bnu\cdot \bnu_h}\right]\right)^\ell: {\rm Def}_\gamma {\bw} + h^{k+1}\|{\bw}\|_{H^1(\gamma)} \|{\bv}\|_{H^1(\gamma)}.
\end{align*}
We then use the algebraic relation $({\bf A}{\bf B}):{\bf C} = ({\bf A}^\intercal {\bf C}):{\bf B}$ and $(\bPi-{\bf I}){\rm Def}_\gamma \pt{\bw} = 0$
to get
\begin{align}\label{eqn:J1Bound}
\left|\int_\gamma J_1: {\rm Def}_\gamma {\bw}\right| \lesssim h^{k+1} \|{\bw}\|_{H^1(\gamma)} \|{\bv}\|_{H^1(\gamma)}.
\end{align}

For $J_2$, we first use $|d|\lesssim h^{k+1}$ to obtain
\begin{align*}
\left[{\bf I} - \frac{\bnu\otimes \bnu_h}{\bnu\cdot \bnu_h} \right] \big[{\bf I}-d {\bf H}\big]^{-1} - \bPi_h
&\lesssim \left(\left[{\bf I} - \frac{\bnu\otimes \bnu_h}{\bnu\cdot \bnu_h} \right] - \bPi_h\right)+h^{k+1}\\
&= \bnu_h\otimes \bnu_h -\frac{\bnu\otimes \bnu_h}{\bnu\cdot \bnu_h} +h^{k+1}.
\end{align*}
Therefore, using $\nab_{\Gamma_{h,k}} \ipt \bv (\bnu_h\otimes \bnu_h) = 0$, $|\nab_{\Gamma_{h,k}} \ipt \bv \bnu|\lesssim h^k$, 
the algebraic identity $({\bf A} \ba \otimes \bb) : {\bf B}= ({\bf A}\ba) \cdot ({\bf B} \bb)$, 
and $|\bnu_h^\intercal {\rm Def}_{\gamma} {\bw}|\lesssim h^k$, we have
\begin{align}\label{eqn:J2Bound}
\left|\int_\gamma J_2: {\rm Def}_\gamma {\bw}\right|\lesssim (h^{2k}+h^{k+1})\|{\bv}\|_{H^1(\gamma)} \|{\bw}\|_{H^1(\gamma)}\lesssim h^{k+1}\|{\bv}\|_{H^1(\gamma)} \|{\bw}\|_{H^1(\gamma)}.
\end{align}

Next, using $|{\bf L}-\bPi|\lesssim h^{k+1}$, $|d|\lesssim h^{k+1}$, 
and the algebraic relation $((\ba\otimes \bb) {\bf A}):{\bf B} = ({\bf B} {\bf A}^\intercal \bb)\cdot \ba$,
we obtain
\begin{equation}
\label{eqn:J3Start}
\begin{split}
\left|\int_{\gamma} J_3: {\rm Def}_\gamma \pt{\bw}\right|
&\lesssim \left|\int_{\gamma} \left((\bPi \bnu_h)\otimes ({\bf H}_h \ipt \bv)\left[{\bf I}- \frac{\bnu\otimes \bnu_h}{\bnu\cdot \bnu_h}\right]\right)^\ell : {\rm Def}_\gamma {\bw}\right|+h^{k+1} \|{\bv}\|_{H^1(\gamma)} \|{\bw}\|_{H^1(\gamma)}\\
&= \left|\int_{\gamma} \left({\rm Def}_\gamma {\bw} \left(\left[{\bf I} - \frac{\bnu_h\otimes \bnu}{\bnu\cdot \bnu_h}\right]{\bf H}_h \ipt \bv\right)^\ell\right)\cdot (\bPi \bnu^\ell_h)\right|
+h^{k+1} \|{\bv}\|_{H^1(\gamma)} \|{\bw}\|_{H^1(\gamma)}.
\end{split}
\end{equation}
The desired bound for $\int_{\gamma} J_3: {\rm Def}_\gamma \pt{\bw}$ is completed in the case $k=1$ by using ${\bf H}_h=0$.
When $k \ge 2$, we recall that $\bnu^\intercal {\bf H}=0$, $|{\bf H}_h - {\bf H}| \lesssim h^{k-1}$,   $|\bPi\bnu_h^\ell| \lesssim h^k$,
and $|\ipt{\bv}^\ell-\bv|\lesssim h^k |\bv|$.  
These relationships yield after noting that $2k-1 \ge k+1$ when $k \ge 2$
\begin{equation}
\label{eqn:J3Start2}
\begin{split}
\left|\int_{\gamma} J_3: {\rm Def}_\gamma \pt{\bw}\right|
&\lesssim 
 \left|\int_{\gamma} \left({\rm Def}_\gamma {\bw} \left(\left[{\bf I} - \frac{\bnu_h\otimes \bnu}{\bnu\cdot \bnu_h}\right]({\bf H}_h-{\bf H}) \ipt \bv\right)^\ell\right)\cdot (\bPi \bnu^\ell_h)\right|\\
&\qquad+ \left |  \int_\gamma \left ({\rm Def}_\gamma \bw {\bf H} \ipt \bv^\ell \right ) \cdot (\bPi \bnu_h^\ell) \right | 
+ h^{k+1} \|{\bv}\|_{H^1(\gamma)} \|{\bw}\|_{H^1(\gamma)}\\
&\lesssim  \left |  \int_\gamma \left ({\rm Def}_\gamma \bw {\bf H} \bv \right ) \cdot (\bPi \bnu_h^\ell) \right | + h^{k+1} \|{\bv}\|_{H^1(\gamma)} \|{\bw}\|_{H^1(\gamma)}.
\end{split}
\end{equation}

Next, following arguments from \cite{HansboLarsonLarsson20}, we have for any $\bz\in \bW^1_1(\Gamma_{h,k})$,
\begin{equation}
\label{eqn:HLTrick}
\begin{split}
\int_{\Gamma_{h,k}} \bz \cdot (\bPi_h \bnu) 
&= \int_{\Gamma_{h,k}} \bz \cdot \nab_{\Gamma_{h,k}} d = -\int_{\Gamma_{h,k}}  d({\rm div}_{\Gamma_{h,k}} \bz) 
\lesssim h^{k+1} \|\bz\|_{W^1_1(\Gamma_{h,k})},
\end{split}
\end{equation}
and, for $\bz\in \bW^1_1(\gamma)$,
\begin{align*}
\int_{\gamma} \bz\cdot \bPi \bnu_h^\ell 
&= -\int_{\gamma} \bz\cdot \bPi_h^\ell  \bnu-\int_{\gamma}\bz\cdot  (\bPi-\bPi^\ell_h)(\bnu-\bnu^\ell_h)\\
& \lesssim  -\int_{\gamma} \bz \cdot (\nab_{\Gamma_{h,k}} d)^\ell + h^{2k}\|\bz\|_{L_1(\gamma)}\\
&\lesssim -\int_{\Gamma_{h,k}} \bz^e \cdot \nab_{\Gamma_{h,k}} d + h^{k+1}\|\bz\|_{L_1(\gamma)}\\
& = \int_K d({\rm div}_{\Gamma_{h,k}} \bz^e)  +h^{k+1} \|\bz\|_{L_1(\gamma)}\\
&\lesssim h^{k+1}\|\bz\|_{W^1_1(\gamma)}. 
\end{align*}
Taking  $\bz = {\rm Def}_\gamma \bw {\bf H} \bv $
and applying the estimate to \eqref{eqn:J3Start2} yields
\begin{align}\label{eqn:J3Bound}
\left|\int_{\gamma} J_3: {\rm Def}_\gamma {\bw}\right|
&\lesssim  
h^{k+1}\|{\bv}\|_{H^1(\gamma)}\|{\bw}\|_{H^2(\gamma)}.
\end{align}

To estimate $J_4$, we first use (cf.~\cite{DemlowNeilan23})
\begin{align*}
\bPi \nab {\bf L}^\ell \ipt \bv^\ell \bPi = \bPi \left(\nab {\bf L} \ipt \bv \bPi_h \left[{\bf I} - \frac{\bnu\otimes \bnu_h}{\bnu\cdot \bnu_h}\right]\big[{\bf I} - d {\bf H}\big]^{-1}\right)^\ell,
\end{align*}
and
\begin{align*}
\bPi \nab {\bf L} \ipt \bv \bPi_h
&= -\mu_h^{-1} \bPi\Big( ({\bf L} \ipt \bv)\otimes \nab \mu_h + \bnu\otimes ({\bf H}\ipt \bv)+({\bf H}\bv)\otimes \bnu + (\bnu\cdot \ipt \bv) {\bf H}+d \nab {\bf H }\ipt \bv\Big)\bPi_h\\
&= -\mu_h^{-1} \Big( ({\bf L} \ipt \bv)\otimes \nab \mu_h\bPi_h  +({\bf H}\ipt \bv)\otimes (\bPi_h \bnu) + (\bnu\cdot \ipt \bv) {\bf H}\bPi_h+d \bPi \nab {\bf H }\ipt \bv \bPi_h \Big).
\end{align*}
Therefore, by setting ${\bf N} = [{\bf I} - \frac{\bnu\otimes \bnu_h}{\bnu\cdot \bnu_h}]$, 
and noting $\bPi_h {\bf N} = {\bf N}$, ${\bf H} {\bf N}^\ell = {\bf H}$, $((\ba\otimes \bb){\bf A}):{\bf B} = ({\bf A} {\bf B}^\intercal \ba)\cdot \bb$,
and $\ipt\bv\cdot \bnu_h=0$, we have
\begin{align*}
\left|\int_\gamma J_4:{\rm Def}_\gamma {\bw}\right|
& \lesssim \left|\int_\gamma \left(\Big( ({\bf L} \ipt \bv)\otimes \nab \mu_h   +({\bf H}\ipt \bv)\otimes (\bPi_h \bnu) + ((\bPi_h \bnu)\cdot \ipt \bv) {\bf H}\Big) 
{\bf N}\right)^\ell : {\rm Def}_\gamma {\bw}  \right|\\
&\qquad + h^{k+1} \|{\bv}\|_{L_2(\gamma)}\|{\bw}\|_{H^1(\gamma)}\\
&=\Big|\int_\gamma \left( ({\bf L} \ipt \bv)\otimes ({\bf N}^\intercal \nab \mu_h)\right)^\ell:{\rm Def}_\gamma {\bw}
+\int_\gamma ({\bf N}^\ell {\rm Def}_\gamma {\bw} {\bf H} \ipt \bv^\ell)\cdot (\bPi_h \bnu)^\ell\\
&\qquad + \int_\gamma \left((\bPi_h \bnu)\cdot \ipt \bv\right)^\ell \bf H : {\rm Def}_\gamma {\bw}))\Big|+h^{k+1} \|{\bv}\|_{L_2(\gamma)}\|{\bw}\|_{H^1(\gamma)}\\
&\lesssim \Big|\int_\gamma \left( ({\bf L} \ipt \bv)\otimes ({\bf N}^\intercal \nab \mu_h)\right)^\ell:{\rm Def}_\gamma {\bw}
+\int_\gamma ({\rm Def}_\gamma {\bw} {\bf H} \bv)\cdot (\bPi_h \bnu)^\ell\\
&\qquad + \int_\gamma \left(\bPi_h \bnu\right)^\ell  \cdot (\bv ({\bf H} : {\rm Def}_\gamma {\bw}))\Big|+h^{k+1} \|{\bv}\|_{L_2(\gamma)}\|{\bw}\|_{H^1(\gamma)},
\end{align*}
where we used $|\bnu_h^\intercal {\rm Def}_\gamma\bw| +|\bPi_h \bnu| \lesssim h^{k}$
and $|\ipt{\bv}^\ell-\bv|\lesssim h^k|\bv|$ in the last inequality.
 Applying \eqref{eqn:HLTrick} 
 to  the second term with $\bz={\rm Def}_\gamma \bw {\bf H} \bv$ and to the third term with $\bz = \bv ({\bf H}^\ell : {\rm Def}_\gamma \bw)$ then yields
\begin{align*}
\left|\int_\gamma J_4:{\rm Def}_\gamma {\bw}\right|
&\lesssim \left|\int_\gamma \left( ({\bf L} \ipt \bv)\otimes ({\bf N}^\intercal \nab \mu_h)\right)^\ell:{\rm Def}_\gamma {\bw}\right| + h^{k+1}\|{\bv}\|_{H^1(\gamma)}\|{\bw}\|_{H^2(\gamma)}.
%
\end{align*}
Using (cf.~\eqref{eqn:JacobiFun})
\[
\nab \mu_h = {\bf H}\bnu_h + {\bf H}_h \bnu - {\rm tr}({\bf H})\bnu+O(h^{k+1}) = {\bf H}\bPi\bnu_h + {\bf H}_h \bPi_h \bnu - {\rm tr}({\bf H})\bnu+O(h^{k+1}),
\]
applying some algebraic identities, and applying \eqref{eqn:HLTrick} {after appropriate $O(h^{k+1})$ perturbations as above to ensure $\bz \in \bW_1^1(\gamma)$ as required}, we have
\begin{align*}
\left|\int_\gamma \left( ({\bf L} \ipt \bv)\otimes ({\bf N}^\intercal \nab \mu_h)\right)^\ell:{\rm Def}_\gamma {\bw}\right|
& \lesssim \left|\int_\gamma ({\bf H} {\bf N} {\rm Def}_\gamma \bw {\bf L} \ipt{\bv})\cdot (\bPi\bnu_h)\right|
 + \left|\int_\gamma ({\bf H}_h {\bf N} {\rm Def}_\gamma \bw {\bf L} \ipt{\bv})\cdot (\bPi_h\bnu)\right|\\
 &\qquad + \left|\int_\gamma {\rm tr}({\bf H}) ({\bf N} {\rm Def}_\gamma \bw {\bf L} \ipt{\bv})\cdot (\bPi_h \bnu)\right| + h^{k+1}\|\bv\|_{L_2(\gamma)}\|\bw\|_{H^1(\gamma)}\\
 &\lesssim h^{k+1}\|\bv\|_{H^1(\gamma)}\|\bw\|_{H^2(\gamma)}.
\end{align*}
Thus, we have the following bound for $J_4$:
\begin{align}\label{eqn:J4Bound}
\left|\int_\gamma J_4:{\rm Def}_\gamma {\bw}\right|\lesssim h^{k+1}\|\bw\|_{H^2(\gamma)}\|\bv\|_{H^1(\gamma)}.
\end{align}

We combine \eqref{eqn:GEStarting} and \eqref{eqn:JsExpand},
along with the estimates \eqref{eqn:J1Bound}, \eqref{eqn:J2Bound}, \eqref{eqn:J3Bound}, and \eqref{eqn:J4Bound}
to get
\begin{align}
 \left|\int_\gamma \big({\rm Def}_\gamma {\bv} - ({\rm Def}_{\Gamma_{h,k}} \ipt \bv)^\ell\big): {\rm Def}_\gamma {\bw} \right|
&\lesssim h^{k+1}\|\bv\|_{H^1(\gamma)}\|\bw\|_{H^2(\gamma)}.
\end{align}
Applying this estimate to \eqref{eqn:StartOfLong} twice then yields
\begin{align*}
|G(\ipt \bv,\ipt \bw)|\lesssim h^{k+1}\|\bv\|_{H^2(\gamma)}\|\bw\|_{H^2(\gamma)}.
\end{align*}

\end{document}